\documentclass[11pt]{article}

\usepackage[a4paper,margin=1in]{geometry}
\usepackage{amsmath}
\usepackage{amsfonts}
\usepackage{amssymb}
\usepackage{amsthm}
\usepackage{amsopn}
\usepackage{graphicx}
\usepackage{epstopdf}
\usepackage{booktabs}
\usepackage{multirow}
\usepackage[sort]{cite}
\usepackage{forest}
\usepackage{enumitem}
\usepackage{float}
\usepackage{aliascnt}
\usepackage{docmute}
\usepackage[hidelinks,hypertexnames=false]{hyperref}
\usepackage[capitalize,nameinlink,noabbrev]{cleveref}

\setlist[itemize]{itemsep=3pt,topsep=3pt}
\setlist[enumerate]{itemsep=3pt,topsep=3pt}

\DeclareGraphicsExtensions{.eps,.pdf,.png,.jpg}
\numberwithin{equation}{section}
\numberwithin{table}{section}
\numberwithin{figure}{section}

\newcommand{\F}{\mathcal{F}}
\newcommand{\N}{\mathbb{N}}

\newcommand{\T}{\mathcal{T}}
\newcommand{\opL}{\mathcal{L}}

\newcommand{\R}{\mathbb{R}}

\newcommand{\di}{\mathrm{d}}

\newcommand{\mO}{\mathcal{O}}

\newcommand{\E}[1]{\mathbb{E}\left[#1\right]}
\newcommand{\Et}[1]{\mathbb{E}_{t}\left[#1\right]}

\newcommand{\Etn}[1]{\mathbb{E}_{t_n}\left[#1\right]}

\newcommand{\Etni}[1]{\mathbb{E}_{t_{n, i}}\left[#1\right]}
\newcommand{\Etnj}[1]{\mathbb{E}_{t_{n, j}}\left[#1\right]}

\newcommand{\br}[1]{\left(#1\right)}
\newcommand{\abs}[1]{\left\vert#1\right\vert}
\newcommand{\vabs}[2]{#1\vert#2 #1\vert}
\newcommand{\rhof}[1]{\rho\left(#1\right)}
\newcommand{\srhof}[1]{\rho^2\left(#1\right)}
\newcommand{\norm}[1]{\left\|#1\right\|}
\newcommand{\bbr}[1]{\left\{#1\right\}}
\newcommand{\dbr}[1]{\left\lfloor#1\right\rceil}
\newcommand{\vbr}[2]{#1(#2 #1)}
\newcommand{\rbr}[1]{\left[#1\right]}

\newcommand{\sptext}[1]{\;\; \text{#1} \;\;}

\newcommand{\alignbox}[1]{\begin{gathered}#1\end{gathered}}

\newcommand{\tree}[1]{
	\begin{forest}
		for tree={grow=north, circle, draw, inner sep=1pt, font=\small},
		#1
	\end{forest}
}

\theoremstyle{plain}
\newtheorem{theorem}{Theorem}[section]

\newaliascnt{assumption}{theorem}
\newtheorem{assumption}[assumption]{Assumption}
\aliascntresetthe{assumption}

\newaliascnt{scheme}{theorem}
\newtheorem{scheme}[scheme]{Scheme}
\aliascntresetthe{scheme}

\newtheorem{proposition}[theorem]{Proposition}
\newtheorem{lemma}[theorem]{Lemma}
\newtheorem{corollary}[theorem]{Corollary}

\theoremstyle{definition}
\newtheorem{definition}[theorem]{Definition}
\newtheorem{example}[theorem]{Example}

\theoremstyle{remark}
\newtheorem{remark}[theorem]{Remark}

\crefname{theorem}{Theorem}{Theorems}
\Crefname{theorem}{Theorem}{Theorems}
\crefname{proposition}{Proposition}{Propositions}
\Crefname{proposition}{Proposition}{Propositions}
\crefname{lemma}{Lemma}{Lemmas}
\Crefname{lemma}{Lemma}{Lemmas}
\crefname{corollary}{Corollary}{Corollaries}
\Crefname{corollary}{Corollary}{Corollaries}
\crefname{definition}{Definition}{Definitions}
\Crefname{definition}{Definition}{Definitions}
\crefname{remark}{Remark}{Remarks}
\Crefname{remark}{Remark}{Remarks}
\crefname{hypothesis}{Hypothesis}{Hypotheses}
\Crefname{hypothesis}{Hypothesis}{Hypotheses}
\crefname{claim}{Claim}{Claims}
\Crefname{claim}{Claim}{Claims}
\crefname{assumption}{Assumption}{Assumptions}
\Crefname{assumption}{Assumption}{Assumptions}
\crefname{scheme}{Scheme}{Schemes}
\Crefname{scheme}{Scheme}{Schemes}
\crefname{example}{Example}{Examples}
\Crefname{example}{Example}{Examples}

\crefname{equation}{}{}
\Crefname{equation}{}{}

\newcommand{\keywords}[1]{\par\smallskip\noindent\textbf{Keywords.} #1}
\newcommand{\subclass}[1]{\par\smallskip\noindent\textbf{MSC2020.} #1}

\makeatletter
	\newcommand\figcaption{\def\@captype{figure}\caption}
	\newcommand\tabcaption{\def\@captype{table}\caption}
\makeatother

\hypersetup{
	pdftitle={Explicit Runge-Kutta Schemes for Backward Stochastic Differential Equations},
	pdfauthor={Shuixin Fang, Yue Qiu, and Weidong Zhao}
}

\title{Explicit Runge-Kutta Schemes for Backward Stochastic Differential Equations\thanks{Shuixin Fang was supported by the China Postdoctoral Science Foundation (Grant Nos. GZB20250705 and 2025M783150). Yue Qiu was supported by the National Natural Science Foundation of China (Grant No. 12471404), the Chongqing Natural Science Foundation (Grant No. CSTB2025NSCQ-GPX0731), and the Fundamental Research Funds for the Central Universities (Grant No. 2025CDJ-IAISYB-015). Weidong Zhao was supported by the National Natural Science Foundation of China (Grant Nos. 12371398 and 12071261).}}

\author{
	Shuixin Fang\thanks{SKLMS \& Institute of Computational Mathematics and Scientific/Engineering Computing, Academy of Mathematics and Systems Science, Chinese Academy of Sciences, Beijing 100190, China. Email: sxfang@amss.ac.cn}
	\and
	Yue Qiu\thanks{College of Mathematics and Statistics, Chongqing University, Chongqing 400044, China. Email: qiuyue@cqu.edu.cn}
	\and
	Weidong Zhao\thanks{School of Mathematics, Shandong University, Jinan, Shandong 250100, China. Corresponding author. Email: wdzhao@sdu.edu.cn}
}

\date{}

\begin{document}

\begingroup
\renewcommand{\RequirePackage}[2][]{}
\renewcommand{\bibliographystyle}[1]{}
\renewcommand{\bibliography}[1]{}

\maketitle

\begin{abstract}
The Butcher theory provides a powerful tool for analyzing order conditions of Runge-Kutta schemes for ordinary differential equations (ODEs); however, such a theory has not yet been well established for backward stochastic differential equations (BSDEs) -- motivating the current work to address this gap. Specifically, we propose a new class of explicit Runge-Kutta schemes for BSDEs. These schemes admit a concise formulation that closely mirrors their ODE counterparts. Building on this formulation, we extend the Butcher theory to the proposed schemes, thereby enabling a symbolic derivation of Taylor expansions for the local truncation errors, and yielding the order conditions. Our approach preserves the elegance and generality of the original Butcher theory: it avoids stage-by-stage error expansions and provides a systematic, stage-inductive analysis, applicable to schemes with any number of stages and any target order. Numerical experiments support the theoretical results.
\subclass{65C30 \and 60H35 \and 65C20}
\end{abstract}
\section{Introduction}

In this work, we focus on the numerical solution of Markovian backward stochastic differential equations (BSDEs) defined on $(\Omega, \mathcal{F}, \mathbb{F}, \mathbb{P})$:
\begin{equation}\label{bsde}
    Y_t = \varphi(X_T) + \int_t^T f(s, X_s, Y_s, Z_s)\,\mathrm{d}s - \int_t^T Z_s\,\mathrm{d}W_s, \quad t \in [0, T],
\end{equation}
where $X_t := X_0 + \sigma W_t$; $T > 0$ is the deterministic terminal time; $(\Omega, \mathcal{F}, \mathbb{F}, \mathbb{P})$ is a complete filtered probability space with $\mathbb{F} = (\mathcal{F}_t)_{0 \leq t \leq T}$ being the natural filtration generated by a standard $q$-dimensional Brownian motion $W = (W_t)_{0 \leq t \leq T}$; $\sigma \in \mathbb{R}^{d \times q}$ is the diffusion coefficient; $f: [0, T] \times \mathbb{R}^d \times \mathbb{R}^p \times \mathbb{R}^{p \times q} \to \mathbb{R}^p$ is the generator, and $\varphi: \mathbb{R}^d \to \mathbb{R}^p$ is the terminal condition of the BSDE. The stochastic integral with respect to $W_s$ is understood in the It\^o sense.
A pair of processes $(Y, Z)$ is called an $L^2$-adapted solution to the BSDE~\eqref{bsde} if it is $\mathbb{F}$-adapted, square-integrable, and satisfies \eqref{bsde}.

In 1990, Pardoux and Peng \cite{pardoux1990adapted} established the existence and uniqueness of solutions for nonlinear BSDEs. 
In 1991, Peng \cite{peng1991probabilistic} derived the nonlinear Feynman-Kac formula. This result reveals a connection between the BSDE~\eqref{bsde} and parabolic partial differential equations (PDEs), that is, 
under suitable regularity conditions, the solution $(Y, Z)$ of \eqref{bsde} admits the representation
\begin{equation}\label{eq_Ytut}
    Y_t = u(t, X_t), \quad Z_t = \nabla_x u(t, X_t) \sigma, \quad t \in [0, T).
\end{equation}
Here, $u: [0, T] \times \R^d \to \R^p$ is the classical solution to the following PDE:
\begin{equation}\label{eq_pde}
    \left\{\begin{aligned}
         & \opL_0 u(t, x)+f\left(t, x, u(t, x), \nabla_{x} u(t, x) \sigma\right)=0, \quad (t, x) \in [0, T) \times \R^d, \\
         & u(T, x) = \varphi(x), \quad x \in \R^d,
    \end{aligned}\right.
\end{equation}
where $\opL_0 := \partial_{t} + 1/2 \sum_{i, j = 1}^d \sum_{l=1}^q\sigma_{il} \sigma_{jl} \partial_{x_{i} x_{j}}^{2}$ is the second-order differential operator. 
The representations in \eqref{eq_Ytut} are the so-called Feynman-Kac formulas, 
which enable the development of numerical methods for solving the PDE~\eqref{eq_pde} via the associated BSDE~\eqref{bsde}, and vice versa.

Numerical methods for BSDEs have been extensively studied in recent years; see, e.g., \cite{Zhang2004numerical, gobet2005regression, zhao2006new, gobet2007error, zhao2010stable, Zhao2012A, Chassagneux2014Runge, Chassagneux2014Linear, zhao2014new, Chassagneux2015Numerical, Yang2020unified, Tang2022Stability, Wang2022Sinc, Fang2023Strong, Fang2023ODE} and references therein.
In particular, \cite{Chassagneux2014Runge} introduced a class of Runge-Kutta schemes for BSDEs and derived order conditions for $m$-stage schemes with order $r \leq 4$ and $m \leq 4$.
More recently, \cite{Tang2025Class} proposed another type of Runge-Kutta schemes and obtained order conditions for $m$-stage schemes up to order $r \leq 3$.
However, both analyses proceed via stage-by-stage error expansions.
As the number of error terms grows exponentially with $r$, the derivations become increasingly cumbersome and even impractical for high-order conditions.
A unified treatment that covers arbitrary $m$ and $r$ is still lacking.
This work aims to fill this gap.

Our main contributions are as follows.
We first propose a new class of explicit Runge-Kutta schemes for BSDEs, which are structurally distinct from those in \cite{Chassagneux2014Runge}.
Notably, the $Y$- and $Z$-components are computed in a unified manner using the same set of coefficients. This feature enables a unified treatment of local truncation errors.
Then we extend the original Butcher theory \cite{butcher2016numerical} to the BSDE setting, thereby establishing a symbolic calculus for deriving error expansions and order conditions.
Our approach preserves the elegance and generality of the original Butcher theory.
It avoids stage-by-stage error expansions tied to specific instances of Runge-Kutta schemes;
instead, it offers an analysis by induction on the stages, applicable to schemes with an arbitrary number of stages and arbitrary target orders.
Finally, we present numerical experiments to validate the derived order conditions.

The remainder of this paper is organized as follows.
In \cref{sec_RefRK}, we introduce the explicit Runge-Kutta schemes for BSDEs and present order conditions up to order 5.
\Cref{sec_theo_tree} is devoted to extending the original Butcher theory to the BSDE context.
The extended theory is then applied in \cref{sec_theoana} to establish the consistency of the proposed schemes.
Numerical experiments and concluding remarks are provided in \cref{sec_numtest}.
Supplementary Materials include an enumeration of the extended Butcher trees, coefficients for the fourth and fifth order Runge-Kutta schemes for BSDEs, explicit expressions of the derived order conditions, and a complete proof of the global error estimates.

\section{Runge-Kutta schemes for BSDEs}\label{sec_RefRK}

This section begins by reformulating the BSDE as a system of reference ordinary differential equations (ODEs), upon which we propose our explicit Runge-Kutta schemes. Subsequently, we present the corresponding order conditions up to order 5. 
The general order conditions and their rigorous proofs are deferred to \cref{sec_theo_tree,sec_theoana}.

\subsection{Reference ODEs}

For $0 \leq t \leq s \leq T$, taking the conditional expectation $\Et{\,\cdot\,} := \mathbb{E}\rbr{\,\cdot\,|\F_{t}}$ on both sides of the BSDE \eqref{bsde}, we obtain the following integral equation:
\begin{equation}\label{eq_inty}
    \Et{Y_s} = \Et{\varphi(X_T)} + \int_s^T \Et{f(r, X_r, Y_r, Z_r)}\di r,
    \quad s \in[t, T].
\end{equation}
To derive a second integral equation, we rewrite the BSDE~\eqref{bsde} into
\begin{equation*}
    Y_{t} = Y_s + \int_{t}^s f(r, X_r, Y_r, Z_r) \di r - \int_{t}^s Z_r \di W_r, \quad s \in [t, T].
\end{equation*}
Multiplying both sides by $\Delta W_{t, s}^{\top} := W_s^{\top} - W_{t}^{\top}$, taking the conditional expectation $\Et{\,\cdot\,}$, we obtain 
\begin{equation}\label{eq_pre_intyw}
\begin{aligned}
    \Et{Y_{t} \Delta W_{t, s}^{\top}} =\;& \Et{Y_{s} \Delta W_{t, s}^{\top}} + \int_{t}^{s} \Et{f(r, X_r, Y_r, Z_r)\Delta W_{t, r}^{\top}} \di r \\
    & - \Et{\int_{t}^s Z_r \di W_r \; \Delta W_{t, s}^{\top}},
\end{aligned}
\end{equation}
where we have used $\Delta W_{t, s} = \Delta W_{t, r} + \Delta W_{r, s}$ and $\Et{f(r, X_r, Y_r, Z_r) \Delta W_{r, s}^{\top}} = 0$ to simplify the second term on the right side of \eqref{eq_pre_intyw}.

By the properties of conditional expectation and the It\^o isometry, the first and last terms of \eqref{eq_pre_intyw} simplify as follows:
\begin{align}
    \Et{Y_{t} \Delta W_{t, s}^{\top}} &= Y_{t}\, \Et{\Delta W_{t, s}^{\top}} = 0,  \label{eq_yt_dw}\\
    \Et{\int_{t}^s Z_r \di W_r \; \Delta W_{t, s}^{\top}} &= \Et{\int_{t}^s Z_r \di W_r \; \int_t^s \di W_r^{\top}} = \int_{t}^s \Et{Z_r} \di r.  \label{eq_intz_dw}
\end{align}
Using the Feynman-Kac formula \eqref{eq_Ytut}, the first term on the right side of \eqref{eq_pre_intyw} can be rewritten as
\begin{equation}\label{eq_ydwt_z}
\begin{aligned}
    \Et{Y_{s} \Delta W_{t, s}^{\top}} =\;& \Et{u(s, X_t + \sigma \Delta W_{t, s}) \Delta W_{t, s}^{\top}} \\
    =\;& \int_{\R^q} u(s, X_t + \sigma \sqrt{s - t}\, w) \sqrt{s - t}\, w^{\top} \phi(w) \di w \\
    =\;& \int_{\R^q} \nabla_x u(s, X_t + \sigma \sqrt{s - t}\, w)\, \sigma \sqrt{s - t} \sqrt{s - t} \,\phi(w) \di w \\
    =\;& (s - t) \Et{Z_s},
\end{aligned}
\end{equation}
where $\phi(w) := (2\pi)^{-q/2} \exp(-|w|^2/2)$ is the density function of the standard $q$-dimensional normal distribution, and the third equality follows from the Stein identity \cite[Lemma 2]{Stein1981Estimation}.

Substituting \eqref{eq_yt_dw}, \eqref{eq_intz_dw}, and \eqref{eq_ydwt_z} into \eqref{eq_pre_intyw}, we obtain the second integral equation
\begin{equation}\label{eq_intz}
    0 = \br{s - t} \Et{Z_s} + \int_{t}^{s}\Et{f(r, X_r, Y_r, Z_r)\Delta W_{t, r}^{\top}} \di r - \int_{t}^{s}\Et{Z_r}\di r
\end{equation}
for $s \in [t, T]$. 
Assuming sufficient regularity conditions on $f$, we differentiate both sides of \eqref{eq_inty} and \eqref{eq_intz} with respect to $s$, which yields the following system of reference ODEs:
\begin{equation}\label{eq_refode0}
    \left\{
    \begin{aligned}
        \frac{\di \Et{Y_s}}{\di s} & = -\Et{f(s, X_s, Y_s, Z_s)}, \\
        \frac{\di \Et{Z_s}}{\di s} & = -\Et{\frac{f(s, X_s, Y_s, Z_s)\Delta W_{t, s}^{\top}}{s - t}},
    \end{aligned}\right.\quad s \in (t, T]. 
\end{equation}
The reference ODE~\eqref{eq_refode0} can be rewritten into the following concise form:
\begin{equation}\label{eq_refode}
    \frac{\di \Et{\Theta_s}}{\di s} = -\Et{F_{t}(s, \Theta_s)}, \quad s \in (t, T],
\end{equation}
where $\Theta_s$ and $F_{t}(s, \Theta_s)$ are both $\R^p \times \R^{p \times q}$-valued random variables given by
\begin{equation}\label{eq_def_thetaF}
    \Theta_s := \br{Y_s, Z_s}, \;\;\; F_{t}(s, \theta) := \vbr{\bigg}{f(s, X_s, \theta), f(s, X_s, \theta) \frac{\Delta W_{t, s}^{\top}}{s - t}}, \;\;\; \theta \in \R^p \times \R^{p \times q}.
\end{equation}
With $t$ fixed in $[0, T)$, the new reference equation \eqref{eq_refode} is a local ODE system with respect to the unknown $s \to \br{\Et{Y_s}, \Et{Z_s}}$ for $s \in (t, T]$.
Benefitting from this, many numerical methods for ODEs can be directly applied to \eqref{eq_refode}.
\begin{remark}\label{rmk_dsEZs}
    When sufficient regularity is available, the right side of the second equation in \eqref{eq_refode0} admits a finite limit as $s \to t+$, i.e.,
    \begin{equation*}
        \left.\frac{\di \Et{Z_s}}{\di s}\right|_{s=t+} = - \nabla_x \tilde{f}\br{t, X_t} \sigma \sptext{with} \tilde{f}\br{t, x} = f\br{t, x, u\br{t, x}, \nabla_x u\br{t, x}\sigma},
    \end{equation*}
    which follows from \eqref{eq_Ytut} and the Stein identity \cite[Lemma 2]{Stein1981Estimation}.
\end{remark}

\subsection{Runge-Kutta schemes}\label{sec_rk_sch}
On the time interval $[0, T]$, we introduce a regular time partition:
\begin{equation}\label{eq_timepart}
    \pi_N: 0 = t_0 < t_1 < \cdots < t_N = T, \quad C_{\pi_N} := \frac{\max_{0\leq n\leq N-1} \Delta t_n}{\min_{0\leq n\leq N-1} \Delta t_n} \leq C_0 < +\infty,
\end{equation}
where $\Delta t_n := t_{n+1} - t_n$ and $C_0$ is a constant independent of $\pi_N$.
For $n = 0, 1, \ldots, N-1$, define $\Delta t := \max_{0\leq n\leq N-1} \Delta t_n$ and $\Delta W_{n+1} := W_{t_{n+1}} - W_{t_n}$.
Let $\Theta^n := \br{Y^n, Z^n}$ denote the numerical approximation to $\Theta_t$ at $t = t_n$.
Throughout this paper, the summation $\sum_{j=i}^m \cdot$ is understood to be zero whenever $i > m$.

Based on the reference ODEs~\eqref{eq_refode} and the theory of Runge-Kutta methods for ODEs, we propose our explicit Runge-Kutta scheme as follows.

\begin{scheme}[Explicit Runge-Kutta scheme]\label{sch_rk}
    Given $\Theta^N$, for $n = N-1, N-2, \cdots, 0$, solve $\Theta^n$ by
    \begin{equation*}
        \left\{\begin{aligned}
            \Theta^{n, i} &= \Etni{\Theta^{n+1}} + \Delta t_n \sum_{j=i+1}^{m} a_{ij} \Etni{F_{t_{n, i}}\br{t_{n, j}, \Theta^{n, j}}}, \quad i = m, \cdots, 1, \\
            \Theta^{n}    &= \Etn{\Theta^{n+1}} + \Delta t_n \sum_{i=1}^m b_i \Etn{F_{t_{n}}\br{t_{n, i}, \Theta^{n, i}}},
        \end{aligned}\right.
    \end{equation*}
    where $t_{n, i} := t_{n+1} - c_i \Delta t_n$, and all coefficients $a_{ij}$, $b_i$, and $c_i$ are real scalars subject to the constraint $0 = c_m < c_{m-1} < \cdots < c_1 < c_0 = 1$.
\end{scheme}

In line with the classical notation in the ODE literature, \Cref{sch_rk} can be represented using the following Butcher tableau \cite{butcher2016numerical}:
\begin{equation}\label{eq_ButTab}
    \begin{array}{c|ccccc}
        c_m     & a_{m, m}   & a_{m, m-1}   & \cdots & a_{m, 1}   \\
        c_{m-1} & a_{m-1, m} & a_{m-1, m-1} & \cdots & a_{m-1, 1} \\
        \vdots  & \vdots     & \vdots       & \ddots & \vdots     \\
        c_1     & a_{1, m}   & a_{1, m-1}   & \cdots & a_{1, 1}   \\\hline
        1       & b_m        & b_{m-1}      & \cdots & b_1
    \end{array},
\end{equation}
where the coefficients $a_{ij}$ with $j \le i$ are not used by the explicit \Cref{sch_rk}; they are included solely for notational convenience in the exposition and analysis.
We also remark that the ordering of $a_{ij}$, $b_i$, and $c_i$ in \eqref{eq_ButTab} differs from the classical ODE case. 
This distinction arises because the stage index $i$ in \Cref{sch_rk} is arranged to decrease from $m$ to $1$, in contrast to the increasing order typically used for ODEs. 
Such a backward progression of $i$ is adopted to be consistent with the backward iteration of the time index $n$ in \Cref{sch_rk}.

To guarantee the consistency of \Cref{sch_rk}, its coefficients must satisfy a set of so-called order conditions. 
\Cref{tab_ordcond} summarizes these order conditions for \Cref{sch_rk} up to order $5$. 
The general order conditions for arbitrary order, together with rigorous proofs, will be presented in \cref{sec_ord_condi} (see \Cref{thm_consistent,prop_Cr_table}) following the development of the Butcher theory for \Cref{sch_rk}.

\begin{table}[tb]
    \centering
    \renewcommand\arraystretch{1.2}
    \caption{
    Order conditions for \Cref{sch_rk}; see \cref{eq_defCr} for conditions of arbitrary order.
    }
    \label{tab_ordcond}
    \begin{tabular}{@{}c|c|c|c|c|c@{}}
        \toprule
        Order      & 1       & 2       & 3       & 4        & 5        \\ \midrule
        Conditions & (1)--(4) & (1)--(5) & (1)--(7) & (1)--(11) & (1)--(20) \\ \bottomrule
    \end{tabular}%
    \\ \vspace{0.3em}
    \begin{tabular}{@{}ll|ll@{}}
        \toprule
        No.  & Conditions                                                   & No.  & Conditions                                                            \\ \midrule
        (1)  & $0 = c_m < c_{m-1} < \cdots < c_1 < c_0 = 1$                 & (11) & $\sum_{i, j, k=1}^m b_i a_{i j} a_{j k} c_k=\frac{1}{24}$             \\
        (2)  & $a_{ij} = 0$ if $1 \leq j \leq i \leq m$                     & (12) & $\sum_{i=1}^m b_i c_i^4=\frac{1}{5}$                                  \\
        (3)  & $c_i = \sum_{j=1}^m a_{ij} \sptext{for} i = 1, 2, \cdots, m$ & (13) & $\sum_{i, j=1}^m b_i c_i^2 a_{i j} c_j=\frac{1}{10}$                  \\
        (4)  & $\sum_{i=1}^m b_i=1$                                         & (14) & $\sum_{i, j=1}^m b_i c_i a_{i j} c_j^2=\frac{1}{15}$                  \\
        (5)  & $\sum_{i=1}^m b_i c_i=\frac{1}{2}$                           & (15) & $\sum_{i, j, k=1}^m b_i c_i a_{i j} a_{j k} c_k=\frac{1}{30}$         \\
        (6)  & $\sum_{i=1}^m b_i c_i^2=\frac{1}{3}$                         & (16) & $\sum_{i=1}^m b_i (\sum_{j=1}^m a_{i j} c_j )^2=\frac{1}{20}$         \\
        (7)  & $\sum_{i, j=1}^m b_i a_{i j} c_j=\frac{1}{6}$                & (17) & $\sum_{i, j=1}^m b_i a_{i j} c_j^3=\frac{1}{20}$                      \\
        (8)  & $\sum_{i=1}^m b_i c_i^3=\frac{1}{4}$                         & (18) & $\sum_{i, j, k=1}^m b_i a_{i j} c_j a_{j k} c_k=\frac{1}{40}$         \\
        (9)  & $\sum_{i, j=1}^m  b_i c_i a_{i j} c_j=\frac{1}{8}$           & (19) & $\sum_{i, j, k=1}^m b_i a_{i j} a_{j k} c_k^2=\frac{1}{60}$           \\
        (10) & $\sum_{i, j=1}^m b_i a_{i j} c_j^2=\frac{1}{12}$             & (20) & $\sum_{i, j, k, l=1}^m b_i a_{i j} a_{j k} a_{k l} c_l=\frac{1}{120}$ \\ \bottomrule
    \end{tabular}%
\end{table}

\Cref{tab_ordcond} provides a straightforward approach for constructing explicit Runge-Kutta schemes for BSDEs. 
In fact, the conditions (2)--(20) in \Cref{tab_ordcond} are identical to those in ODE cases~\cite{butcher2016numerical,Hairer1993Solving,Lambert1991Numerical}, while the condition (1) is an additional requirement arising from the backward dynamics of the BSDE~\eqref{bsde}. 
Therefore, specific instances of \Cref{sch_rk} can be constructed by adopting the Butcher tableau~\eqref{eq_ButTab} from the ODE case, subject to the additional order condition (1). 
Following this approach, we present below several Butcher tableaux for \Cref{sch_rk}, where coefficients $a_{ij}$ not explicitly shown are understood to be zero.

\begin{itemize}
    \item the 1-stage 1st-order scheme:
          \begin{equation}\renewcommand{\arraystretch}{1.2} \label{eq_sch_euler}
              \text{Euler:} \quad\begin{array}{l|l}
                  0        &   \\
                  \hline 1 & 1
              \end{array};
          \end{equation}
    \item the 2-stage 2nd-order Runge-Kutta schemes (see, e.g., \cite[Section 321]{butcher2016numerical} and \cite[Theorem 1.5]{Chassagneux2014Runge}):
        \begin{equation}\renewcommand{\arraystretch}{1.2}\label{eq_sch_RK2c1}
            \text{RK}\br{2; c_1}: \quad \begin{array}{c|cc}
                0        &                  & \\
                c_1      & c_1              & \\
                \hline 1 & 1-\frac{1}{2c_1} & \frac{1}{2c_1}
            \end{array} \sptext{with} 0 < c_1 < 1;
        \end{equation}
    \item the 3-stage 3rd-order Runge-Kutta scheme (see, e.g., \cite[Section 320]{butcher2016numerical} and \cite[Corollary 1.3]{Chassagneux2014Runge}):
          \begin{equation}\renewcommand{\arraystretch}{1.2}\label{eq_sch_RK3c1c2}
              \text{RK}\br{3; c_1, c_2}: \quad \begin{array}{c|ccc}
                  0   & & & \\
                  c_2 &  c_2  & & \\
                  c_1 & \frac{c_1\left(3 c_2-3 c_2^2 - c_1\right)}{c_2\left(2-3 c_2\right)} & \frac{c_1\left(c_1 - c_2\right)}{c_2\left(2-3 c_2\right)} & \\
                  \hline 1 & \frac{-3 c_1+6 c_2 c_1+2-3 c_2}{6 c_2 c_1} & \frac{3 c_1-2}{6 c_2\left(c_1-c_2\right)} & \frac{2-3 c_2}{6 c_1\left(c_1-c_2\right)}
              \end{array}
          \end{equation}
          with $0 < c_2 < c_1 < 1$ and $c_2 \neq \frac{2}{3}$.
\end{itemize}

Alternatively, we can numerically find the coefficients subject to the conditions in \Cref{tab_ordcond}.
The following two-step procedure provides an example:
\begin{enumerate}
    \item Set $c_i = 1 - i/m$ for $i = 0, 1, \cdots, m$ to fulfill the condition (1) in \Cref{tab_ordcond};
    \item Use optimization algorithms to minimize $\sum_{i, j=1}^m a_{ij}^2 + \sum_{i=1}^m b_i^2$ over all $a_{ij}$ and $b_i$ subject to the remaining $r$th order conditions in \Cref{tab_ordcond}.  
\end{enumerate}
Following the above procedure, we find the coefficients for \Cref{sch_rk} achieving orders $r=4$ and $5$; the corresponding coefficients are provided in the Supplementary Materials.
Moreover, leveraging \Cref{tab_ordcond}, one can formulate more sophisticated optimization problems to construct instances of \Cref{sch_rk} with improved performance.

\begin{remark}[Explicit order barrier]
    According to the Runge-Kutta theory for ODEs (see, e.g., \cite[Section 322]{butcher2016numerical}), when $m = 4$, the conditions (2)--(11) in \Cref{tab_ordcond} force $c_1 = 1$, which contradicts the condition (1). 
    Consequently, \Cref{sch_rk} suffers from an explicit order barrier: no 4th-order instance of \Cref{sch_rk} can be constructed when $m \leq 4$.
    In contrast, in the ODE setting explicit 4 stage Runge-Kutta schemes of order 4 do exist. 
    Hence, consistently with \cite[Theorem 1.7]{Chassagneux2014Runge}, \Cref{sch_rk} encounters the explicit order barrier earlier than its ODE counterparts. 
    This earlier barrier is caused by the additional constraint (1) in \Cref{tab_ordcond}.
\end{remark}

\begin{remark}[Comparison with \cite{Chassagneux2014Runge}]\label{rmk_ord_condi_ode}
    \Cref{sch_rk} differs significantly from existing Runge-Kutta schemes proposed in the literature for BSDEs, e.g., \cite[Definition 1.1]{Chassagneux2014Runge}.
    In particular, the solutions $Y_{t_n}$ and $Z_{t_n}$ are approximated by a unified scheme rather than two separate schemes with different coefficients.
    This feature facilitates the extension of Butcher theory to \Cref{sch_rk}, as detailed in \cref{sec_theo_tree}.
    The derived order conditions bypass stage-by-stage error expansions and are applicable to instances of \Cref{sch_rk} with an arbitrary number of stages and arbitrary target orders. 
    This constitutes a major novelty of our work compared to \cite{Chassagneux2014Runge}.     
\end{remark}

\section{Butcher theory for BSDEs}\label{sec_theo_tree}

In this section, we develop a Butcher theory for BSDEs, which will serve as the fundamental tool for establishing the consistency of \Cref{sch_rk} in the next section.

To motivate the forthcoming development, we first discuss the challenges that arise when directly transplanting the ODE approach \cite{butcher2016numerical,Hairer1993Solving} to \Cref{sch_rk}.
In the ODE setting, consistency is typically analyzed by Taylor expanding both the exact solution and the numerical scheme in time.
The resulting terms are organized via Butcher trees, which furnish a symbolic calculus for the associated coefficients and derivatives.
In the BSDE setting, the reference ODE~\eqref{eq_refode} involves two time variables: besides the evolution time $s$, a conditioning time $t$ enters through $\Et{\,\cdot\,}$.
This conditioning time is intrinsic to BSDEs and cannot be merged as an entry of the state variable $\Theta_s$, unlike in the ODE setting.
Moreover, the conditioning time cannot be regarded as fixed, since $t_{n,i}$ in $\Etni{\,\cdot\,}$ from \Cref{sch_rk} varies with the time-step size $\Delta t$.
Consequently, Taylor expansions must account for the $t$-dependence of $\Et{\,\cdot\,}$, generating additional terms that classical Butcher trees do not capture.
This necessitates an extension of Butcher theory.

Some frequently used notations in the subsequent development are as follows.
Define $\N = \bbr{0, 1, 2, \cdots}$ and $\N^+ = \bbr{1, 2, \cdots}$.
For $l \in \N^+$, denote by $\mathbb{A}_l$ the set of permutations of $1, 2, \cdots, l$, i.e.,
\begin{equation}\label{eq_defAl}
    \mathbb{A}_l := \Big\{\br{p_1, p_2, \cdots, p_l}: \bbr{p_i}_{i=1}^l \subset \bbr{i}_{i=1}^l; \; p_i = p_j \sptext{if and only if} i=j \Big\},
\end{equation}
and $\mathbb{A}_0 := \emptyset$.
For any two sets $A$ and $B$, define $A \backslash B := \bbr{x: x \in A, x \notin B}$. 
Denote by $C_b^{k_1, k_2, k_3, k_4}$ the set of functions $(t, x, y, z) \mapsto g(t, x, y, z)$ with uniformly bounded continuous partial derivatives up to orders $k_1$, $k_2$, $k_3$, $k_4$ with respect to $t \in [0, T]$, $x \in \R^d$, $y \in \R^p$, $z \in \R^{p\times q}$, respectively.
And the notations $C_b^{k_1}$, $C_b^{k_1, k_2}$, $C_b^{k_1, k_2, k_3}$ are defined analogously.
For a random vector $\eta$ depending on $h > 0$, denote $\eta = \mO\br{h}$ if there exists a constant $C$ independent of $h$ such that $\abs{\eta} \leq Ch$ a.s..
Additionally, if $C$ is further independent of $n \in \N$, we say $\eta = \mO\br{h}$ uniformly for $n \in \N$.
For a sequence of differential operators $\mathcal{L}_1, \mathcal{L}_2, \cdots, \mathcal{L}_m$, we denote their composition by $\prod_{i=1}^m \mathcal{L}_i := \mathcal{L}_1 \circ \mathcal{L}_2 \circ \cdots \circ \mathcal{L}_m$.

\subsection{Labelled numbered trees}

The extension of Butcher theory starts with the concept of labeled numbered trees (LN-trees), as defined in \Cref{def_tree0} and illustrated in \Cref{fig_diaram_lntree,fig_equal_tree}.

\begin{definition}[LN-tree and its branch]\label{def_tree0}
Set $\widetilde{\T}^0 := \emptyset$ and $\widetilde{\T}^1 := \{[\;]_k: k \in \N\}$, where $[\;]_k$ denotes the LN-tree consisting of a single node numbered by the integer $k$; this node is called the root of $[\;]_k$.
For $q = 2, 3, \cdots$, define recursively
\begin{equation}\label{eq_defT}
    \widetilde{\T}^{q} := \Big\{\rbr{\Upsilon_1\cdots\Upsilon_l}_k: \bbr{\Upsilon_i}_{i=1}^l \subset \bigcup_{j=1}^{q-1} \widetilde{\T}^{j},\; k \in \N, \; l \in \N \Big\},
\end{equation}
where
\begin{itemize}
\item[(a)] 
If $l=0$, set $\rbr{\Upsilon_1 \cdots \Upsilon_l}_k := \rbr{\;}_k$; if $l \in \N^{+}$, then $\rbr{\Upsilon_1\cdots\Upsilon_l}_k$ denotes the LN-tree obtained by introducing a new node numbered by $k$ and connecting it to the roots of the LN-trees $\Upsilon_1, \Upsilon_2, \cdots, \Upsilon_l$; this newly added node is the root of $\rbr{\Upsilon_1 \cdots \Upsilon_l}_k$;
\item[(b)] 
An LN-tree $\Upsilon^{\prime}$ is called a branch of $\rbr{\Upsilon_1 \cdots \Upsilon_l}_k$ if $\Upsilon^{\prime} \in \bbr{\Upsilon_i}_{i=1}^l$, or $\Upsilon^{\prime}$ is a branch of $\Upsilon_i$ for some $i \in \bbr{1, 2, \cdots, l}$; here, for preciseness, all LN-trees in $\widetilde{\T}^1$ are understood to have no branch.
\end{itemize}
Finally, define the set of all LN-trees by $\widetilde{\T} := \bigcup_{q\in \N} \widetilde{\T}^{q}$.
\end{definition}

\begin{figure}[t]
\centering
\begin{gather*}
    \overline{\Upsilon}_1 = [[\;]_{k_3} [\;]_{k_4}]_{k_2} = \alignbox{\tree{[$k_2$ [$k_4$] [$k_3$]]}}, \quad \overline{\Upsilon}_2 = [[\;]_{k_6}]_{k_5} =\; \alignbox{\tree{[$k_5$ [$k_6$]]}}, \quad \overline{\Upsilon}_3 = [\;]_{k_7} = \; \alignbox{\tree{[$k_7$]}}\\
    \overline{\Upsilon}_4 = [\overline{\Upsilon}_1 \overline{\Upsilon}_2 \overline{\Upsilon}_3]_{k_1} \; = \; \rbr{\alignbox{\tree{[$k_2$ [$k_4$] [$k_3$]]} \quad \tree{[$k_5$ [$k_6$]]} \quad \tree{[$k_7$]}}}_{k_1} = \alignbox{\tree{[$k_1$ [$k_7$] [$k_5$ [$k_6$]] [$k_2$ [$k_4$] [$k_3$]]]}}
\end{gather*}
\caption{
Schematic illustration of four LN-trees.
The notation $\alignbox{\tree{[$k$]}}$ denotes a node numbered by $k$.
The nodes numbered by $k_2$, $k_5$, $k_7$, and $k_1$ are the roots of $\overline{\Upsilon}_1$, $\overline{\Upsilon}_2$, $\overline{\Upsilon}_3$, and $\overline{\Upsilon}_4$, respectively.
The LN-trees $\overline{\Upsilon}_1$, $\overline{\Upsilon}_2$, $\overline{\Upsilon}_3$, $[\;]_{k_3}$, $[\;]_{k_4}$, and $[\;]_{k_6}$ are all branches of $\overline{\Upsilon}_4$.
The LN-tree $\overline{\Upsilon}_3$ has no branches.
Assigning specific values to $k_1, k_2, \cdots, k_7$ yields concrete instances of LN-trees; see \Cref{fig_equal_tree}.
}\label{fig_diaram_lntree}
\end{figure}

\begin{figure}[t]
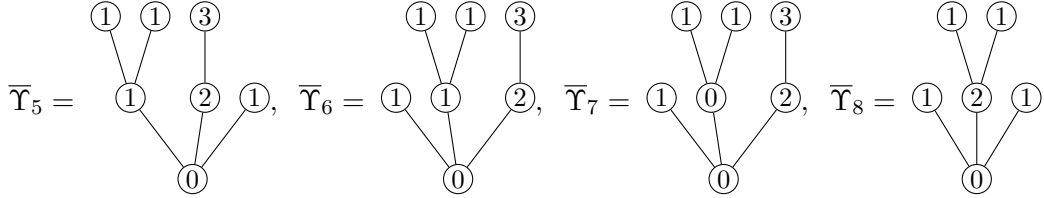

\centering
\begin{gather*}
    \overline{\Upsilon}_5 = \;\alignbox{\tree{[$0$ [$1$] [$2$ [$3$]] [$1$ [$1$] [$1$]]]}}, \;\;\overline{\Upsilon}_6 = \;\alignbox{\tree{[$0$ [$2$ [$3$]] [$1$ [$1$] [$1$]] [$1$]]}}, \;\; \overline{\Upsilon}_7 = \;\alignbox{\tree{[$0$ [$2$ [$3$]] [$0$ [$1$] [$1$]] [$1$]]}},  \;\;\overline{\Upsilon}_8 = \;\alignbox{\tree{[$0$ [$1$] [$2$ [$1$] [$1$]] [$1$]]}}
\end{gather*}
\caption{
Four LN-trees with specific node numbers.
Here, $\overline{\Upsilon}_5$ is a concrete instance of $\overline{\Upsilon}_4$ in \Cref{fig_diaram_lntree}, obtained by assigning specific values to $k_1, k_2, \cdots, k_7$.
Permuting the branches of $\overline{\Upsilon}_5$ yields $\overline{\Upsilon}_6$, whereas $\overline{\Upsilon}_7$ and $\overline{\Upsilon}_8$ cannot be obtained from $\overline{\Upsilon}_5$ by any branch permutation.
The similarity between $\overline{\Upsilon}_5$ and $\overline{\Upsilon}_6$ will be formalized in \cref{sec_equival}.
}\label{fig_equal_tree}
\end{figure}

\Cref{def_tree0} immediately yields the following proposition, which describes the monotonic relationships among the sets $\widetilde{\T}^q$ for $q \in \N^+$ and provides a simplified form of \eqref{eq_defT}.
\begin{proposition}\label{prop_mono_T}
(i) For $q \in \N$, it holds that $\widetilde{\T}^q \subset \widetilde{\T}^{q+1}$. 
(ii) The set $\widetilde{\T}^q$ given in \eqref{eq_defT} can be expressed as 
\begin{equation}\label{eq_express_Tq}
    \widetilde{\T}^{q} = \bbr{\rbr{\Upsilon_1\cdots\Upsilon_l}_k: \bbr{\Upsilon_i}_{i=1}^l \subset \widetilde{\T}^{q-1},\; k \in \N, \; l \in \N}, \quad q = 2, 3, \cdots.
\end{equation}
\end{proposition}

\begin{remark}[Connections to Butcher trees for ODEs]
The LN-tree plays a role analogous to the labeled tree in the ODE setting \cite[p.~146]{Hairer1993Solving}. 
It is, however, richer in structure because each node carries an integer $k_i$; see \Cref{fig_diaram_lntree}. 
These integers differ from the labels used in \cite[Section 305]{butcher2016numerical} and \cite[p.~146]{Hairer1993Solving}, 
since in an LN-tree, distinct nodes may share the same values of $k_i$; see \Cref{fig_equal_tree}. 
As will be seen in \cref{sec_expand_R}, this additional numbering flexibility is essential for representing the error expansion of \Cref{sch_rk}.
\end{remark}

\subsection{Functions associated with LN-trees}
 
In this subsection, we define some functions of the form $\phi: \widetilde{\T} \to \R$.
The definition is guided by the following observations.
\begin{enumerate}
\item 
The set $\widetilde{\T}$ admits the decomposition $\widetilde{\T} = \bigcup_{q=1}^{\infty} (\widetilde{\T}^q \backslash \widetilde{\T}^{q-1})$.
Hence, defining $\phi(\cdot)$ amounts to prescribing its values on each difference set $\widetilde{\T}^q \backslash \widetilde{\T}^{q-1}$ for $q = 1, 2, \cdots$.
By \Cref{prop_mono_T}(i), these difference sets are pairwise disjoint.
Thus, assigning $\phi(\Upsilon)$ separately on every $\widetilde{\T}^q \backslash \widetilde{\T}^{q-1}$ introduces no ambiguity.

\item 
By \Cref{prop_mono_T}(ii), for $q \geq 2$, every LN-tree $\Upsilon \in \widetilde{\T}^q \backslash \widetilde{\T}^{q-1}$ can be represented as $\Upsilon = \rbr{\Upsilon_1 \cdots \Upsilon_l}_k$ with $\Upsilon_1, \cdots, \Upsilon_l \in \widetilde{\T}^{q-1}$ and $k, l \in \N$.
Consequently, defining $\phi(\Upsilon)$ reduces to prescribing its dependence on $k$, $l$, and the values $\phi(\Upsilon_1), \cdots, \phi(\Upsilon_l)$; namely,
\begin{equation}\label{eq_def_phi}
    \phi(\Upsilon) := Q\vbr{\big}{k, l, \phi(\Upsilon_1), \cdots, \phi(\Upsilon_l)},
\end{equation}
where $Q$ is a certain function, and $\phi(\Upsilon_1), \cdots, \phi(\Upsilon_l)$ are defined on $\widetilde{\T}^{q-1} = \bigcup_{l=1}^{q-1} (\widetilde{\T}^l \backslash \widetilde{\T}^{l-1})$ by the same recursive procedure.

\item 
In summary, it suffices to first specify $\phi\br{\rbr{\;}_k}$ for $\rbr{\;}_k \in \widetilde{\T}^1$, and then, for $q = 2, 3, \cdots$, specify the function $Q$ in \eqref{eq_def_phi}. This recursive construction makes $\phi\br{\cdot}$ well-defined on $\widetilde{\T}$. 
\end{enumerate}

We are now ready to present Definition~\ref{def_order}, which defines two numerical characteristics of an LN-tree.

\begin{definition}[Order and factorial]\label{def_order}
    For $\rbr{\;}_k \in \widetilde{\T}^{1}$, we define its order $\abs{\rbr{\;}_k}$ and factorial $\gamma\br{\rbr{\;}_k}$ by
    \begin{equation}\label{eq0_recur_alp}
        \vabs{\big}{\rbr{\;}_k} := 1 + k, \quad \gamma\br{\rbr{\;}_k} := k!,
    \end{equation}
    respectively. For $q = 2, 3, \cdots$ and $\Upsilon = \rbr{\Upsilon_1 \cdots \Upsilon_{l}}_{k} \in \widetilde{\T}^{q} \backslash \widetilde{\T}^{q-1}$, we define its order $\abs{\Upsilon}$ and factorial $\gamma\br{\Upsilon}$ by
        \begin{equation}\label{eq_recur_alp}
            \abs{\Upsilon} := 1 + k + \sum_{i=1}^l \abs{\Upsilon_i}, \quad \gamma\br{\Upsilon} := k! \, l! \prod_{i=1}^l \gamma\br{\Upsilon_i},
        \end{equation}
    respectively. By induction on $q = 1, 2, \cdots$, these definitions of $\abs{\cdot}$ and $\gamma\br{\cdot}$ extend to $\widetilde{\T}$.
\end{definition}

In the following, we use \Cref{exam_1} to illustrate how to calculate the order and factorial for a specific LN-tree.

\begin{example}[Order and factorial]\label{exam_1}
Consider $\overline{\Upsilon}_1, \overline{\Upsilon}_2, \overline{\Upsilon}_3, \overline{\Upsilon}_4$ given in \Cref{fig_diaram_lntree}.
\begin{itemize}
\item 
Their orders are given by $\abs{\overline{\Upsilon}_1} = 1 + k_2 + (1 + k_3 + 1 + k_4)$, $\abs{\overline{\Upsilon}_2} = 1 + k_5 + (1 + k_6)$, $\abs{\overline{\Upsilon}_3} = 1 + k_7$, and $\abs{\overline{\Upsilon}_4} = 1 + k_1 + \abs{\overline{\Upsilon}_1} + \abs{\overline{\Upsilon}_2} + \abs{\overline{\Upsilon}_3}$, respectively.
In fact, the order of an LN-tree is given by the total number of its nodes plus the sum of the labels on those nodes.

\item Their factorials are calculated by
$\gamma\br{\overline{\Upsilon}_1} = k_2! \times 2! \times (k_3!  \times 0! \times k_4!  \times 0!)$, $\gamma\br{\overline{\Upsilon}_2} = k_5! \times 1! \times (k_6!  \times 0!)$, $\gamma\br{\overline{\Upsilon}_3} = k_7! \times 0!$ and $\gamma\br{\overline{\Upsilon}_4} = k_1! \times 3! \times \gamma\br{\overline{\Upsilon}_1} \times \gamma\br{\overline{\Upsilon}_2} \times \gamma\br{\overline{\Upsilon}_3}$, respectively.
\end{itemize}
\end{example}

In what follows, Definitions~\ref{def_weight} and \ref{def_elediff} introduce the elementary coefficients and the elementary differentials, respectively. 
The elementary coefficients will play a central role in the order conditions given in \cref{sec_ord_condi}.

\begin{definition}[Elementary coefficient]\label{def_weight}
    For $\rbr{\;}_k \in \widetilde{\T}^{1}$, the elementary coefficients $A_j^m\br{[\;]_k}$, $j = 0, 1, \cdots, m$, are defined as
    \begin{align}
         & A_i^m\br{[\;]_k} := \sum_{j=1}^{m} a_{ij} \br{c_i - c_j}^{k} - \frac{c_i^{k + 1}}{k + 1}, \quad i = m, m-1, \cdots, 1, \label{eq_defAi} \\
         & A_0^m\br{[\;]_k} := \sum_{j=1}^{m} b_j \br{c_0 - c_j}^{k} - \frac{1}{k + 1}. \label{eq_defA0}
    \end{align}
    For $q = 2, 3, \cdots$, and $\Upsilon = \rbr{\Upsilon_1 \cdots \Upsilon_l}_k \in \widetilde{\T}^{q} \backslash \widetilde{\T}^{q-1}$, the elementary coefficients $A_j\br{\Upsilon}$, $j = 0, 1, \cdots, m$, are defined as
    \begin{align}
         & A_i^m\br{\Upsilon} := \sum_{j = 1}^{m} a_{ij} \br{c_i - c_j}^{k} \prod_{p = 1}^l A_{j}^m\br{\Upsilon_{p}}, \quad i = m, m-1, \cdots, 1, \label{eq_recur_Ai} \\
         & A_0^m\br{\Upsilon} := \sum_{j = 1}^{m} b_j \br{c_0 - c_j}^{k} \prod_{p = 1}^l A_{j}^m\br{\Upsilon_{p}}. \label{eq_recur_A0}
    \end{align}
    By induction on $q = 1, 2, \cdots$, the definition of $\Upsilon \mapsto A_j^m\br{\Upsilon}$ is extended to $\Upsilon \in \widetilde{\T}$.
\end{definition}

\begin{definition}[Elementary differential]\label{def_elediff}
    For $\rbr{\;}_k \in \widetilde{\T}^{1}$, define the elementary differential $t \mapsto F_{t}^{[\;]_k}$ by
    \begin{equation}\label{eq_defFt0_Up}
        F_{t}^{[\;]_k} := \lim_{s \to t+} F_{t, s}^{[\;]_k}, \;\;\; F_{t, s}^{[\;]_k} := \partial_s^k \Et{F_{t}\br{s, \Theta_s}}, \;\; s \in (t, T), \;\; t \in [0, T). 
    \end{equation}
    For $q = 2, 3, \cdots$, and $\Upsilon = \rbr{\Upsilon_1 \cdots \Upsilon_l}_k \in \widetilde{\T}^{q} \backslash \widetilde{\T}^{q-1}$, define the elementary differential $t \mapsto F_{t}^{\Upsilon}$ by
    \begin{equation}\label{eq_recur_F}
        F_{t}^{\Upsilon} := \lim_{s \to t+} F_{t, s}^{\Upsilon}, \;\;\; F_{t, s}^{\Upsilon} := \partial_s^k \Et{\br{F_s^{\Upsilon_1} \cdot \partial_{\theta}} \br{F_s^{\Upsilon_2} \cdot \partial_{\theta}} \cdots \br{F_s^{\Upsilon_l} \cdot \partial_{\theta}} F_{t}\br{s, \Theta_s}}
    \end{equation}
    for $s \in (t, T)$ and $t \in [0, T)$,
    where $F_{t}\br{s, \Theta_s}$ is given in \eqref{eq_def_thetaF},
    and $F_s^{\Upsilon_p} \cdot \partial_{\theta} := \sum_{i=1}^p \sum_{j=1}^{q+1} F_{s, ij}^{\Upsilon_p} \partial_{\theta_{ij}}$ 
    with $F_{s, ij}^{\Upsilon_p}$ and $\theta_{ij}$ denoting the $\br{i, j}$th element of $F_{s}^{\Upsilon_p}$ and $\theta \in \R^{p \times (1 + q)}$, respectively.
    In \eqref{eq_recur_F}, the partial derivatives $\partial_{\theta}$ do not act on $F_s^{\Upsilon_1}, F_s^{\Upsilon_2}, \cdots, F_s^{\Upsilon_l}$, but only on the second argument of $F_{t}\br{s, \Theta_s}$.
    By induction on $q$, the definition of $\Upsilon \mapsto F_{t}^{\Upsilon}$ is extended to $\Upsilon \in \widetilde{\T}$.
\end{definition}
 
The existence of $F_{t}^{\Upsilon}$ and $F_{t, s}^{\Upsilon}$ in \eqref{eq_defFt0_Up} and \eqref{eq_recur_F} is nontrivial, since the convergence of the limits as $s \to t+$ depends on the regularity of $u$ and $f$. 
The precise regularity conditions will be given in \Cref{prop_existence_F} in the next section.

\subsection{Equivalence between LN-trees}\label{sec_equival}

Recalling \Cref{fig_equal_tree}, the LN-trees $\overline{\Upsilon}_5$ and $\overline{\Upsilon}_6$ differ only by a permutation of branches. 
Their topological structures coincide. 
Hence their numerical characteristics, elementary coefficients, and elementary differentials coincide. 
We formalize this equivalence in \Cref{def_equival_T}.

\begin{definition}[Equivalence relation]\label{def_equival_T}
    Any two LN-trees $\rbr{\;}_k, \rbr{\;}_l \in \widetilde{\T}^1$ are said to be equivalent (denoted as $\rbr{\;}_k \sim \rbr{\;}_l$) in $\widetilde{\T}^1$, if and only if $k = l$.
    For $q = 2, 3, \cdots$, any two LN-trees $\Upsilon, \Upsilon^{\prime} \in \widetilde{\T}^q$ are said to be equivalent (denoted as $\Upsilon \sim \Upsilon^{\prime}$) in $\widetilde{\T}^q$, if and only if one of the following conditions holds:
    \begin{itemize}
        \item there exists $k \in \N$, such that, $\Upsilon = \Upsilon^{\prime} = \rbr{\;}_k$;
        \item there exist $k \in \N$, $l \in \N^+$ and $\br{p_1, p_2, \cdots, p_l} \in \mathbb{A}_l$, such that 
        \begin{equation}\label{eq_equival_T}
            \Upsilon = \rbr{\Upsilon_1\cdots\Upsilon_l}_k, \quad \Upsilon^{\prime} = \rbr{\Upsilon_1^{\prime}\cdots\Upsilon_l^{\prime}}_k, \quad \Upsilon_{p_i} \sim \Upsilon_i^{\prime} \sptext{in} \widetilde{\T}^{q-1}, \; i = 1, 2, \cdots, l.
        \end{equation}
    \end{itemize}
\end{definition}

Definition~\ref{def_equival_T} obviously leads to the following proposition, 
where the equivalence relation is extended to $\widetilde{\T} = \bigcup_{q \in \N^+} \widetilde{\T}^q$.

\begin{proposition}\label{prop_equivalT}
    Let $p,q \in \N^+$ and $\Upsilon, \Upsilon' \in \widetilde{\T}^q$. 
    Then $\Upsilon \sim \Upsilon'$ in $\widetilde{\T}^q$ if and only if $\Upsilon \sim \Upsilon'$ in $\widetilde{\T}^{q+p}$.
    Hence the relation ``$\sim$'' extends uniquely to $\widetilde{\T}$.
\end{proposition}

Combining \cref{def_tree0,def_order,def_weight,def_elediff,def_equival_T}, we immediately obtain the following proposition, which summarizes the shared properties of equivalent LN-trees.

\begin{proposition}[Properties preserved under equivalence]\label{prop_Labs_equ}
    For any $\Upsilon, \Upsilon^{\prime} \in \widetilde{\T}$ satisfying $\Upsilon \sim \Upsilon^{\prime}$, the following propositions hold
    \begin{itemize}
        \item[(i)] $\Upsilon$ and $\Upsilon^\prime$ have the same roots;
        \item[(ii)] $\abs{\Upsilon} = \abs{\Upsilon^{\prime}}$ and $\gamma\br{\Upsilon} = \gamma\br{\Upsilon^{\prime}}$;
        \item[(iii)] $F_{t}^{\Upsilon} = F_{t}^{\Upsilon^{\prime}}$, $F_{t, s}^{\Upsilon} = F_{t, s}^{\Upsilon^{\prime}}$ for $s \in (t, T)$, $t \in [0, T)$;
        \item[(iv)] $A_j^m\br{\Upsilon} = A_j^m\br{\Upsilon^{\prime}}$ for $j = 0, 1, \cdots, m$.
    \end{itemize}
\end{proposition}

\subsection{Unlabelled numbered trees}

\Cref{prop_Labs_equ} shows that two equivalent LN-trees are identical in all essential aspects and differ only in their representation.
For brevity, we introduce equivalence classes of LN-trees, called unlabelled numbered trees (ULN-trees), formalized in \Cref{def_Ntree}.
Two illustrative examples are shown in \Cref{fig_ntree}.

\begin{definition}[ULN-tree]\label{def_Ntree}
Let $\Upsilon$ be any LN-trees in $\widetilde{\T}$.
\begin{itemize}
\item[(i)] 
The ULN-tree $\dbr{\Upsilon}$ is defined as the equivalence class of $\Upsilon$ modulo ``$\sim$'', i.e. $\dbr{\Upsilon} := \{\Upsilon^{\prime} \in \widetilde{\T}: \Upsilon^{\prime} \sim \Upsilon\}$. 
If the context is clear, $\dbr{\Upsilon}$ is abbreviated as $\Upsilon$.

\item[(ii)] 
The root of $\dbr{\Upsilon}$ is defined to be the root of a representative $\Upsilon^{\prime} \in \dbr{\Upsilon}$.

\item[(iii)] 
Define $\alpha\br{\dbr{\Upsilon}}$ to be the cardinality of $\dbr{\Upsilon}$, i.e., $\alpha\br{\dbr{\Upsilon}} := \# \{\Upsilon^{\prime} \in \widetilde{\T}: \Upsilon^{\prime} \sim \Upsilon\}$.

\item[(iv)] 
For $q \in \N$, the sets $\T$ and $\T^q$ of ULN-trees are given by the quotient sets $\widetilde{\T} / \sim$ and $\widetilde{\T}^q / \sim$, respectively.
\end{itemize}
\end{definition}

\begin{figure}[t]
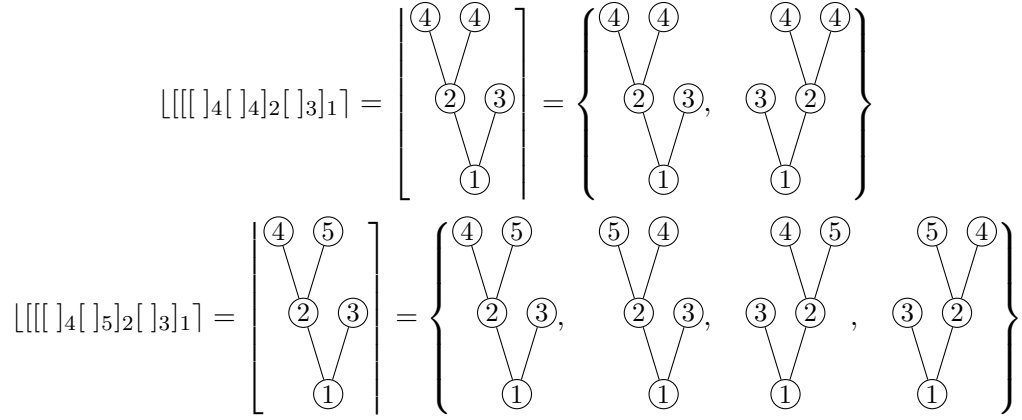

    \centering
    \begin{gather*}
        \dbr{[[[\;]_4 [\;]_4]_2 [\;]_3]_1} = \dbr{\alignbox{\tree{[1 [3] [2 [4] [4]]]}}} = \bbr{\alignbox{\tree{[1 [3] [2 [4] [4]]]}}, \quad \alignbox{\tree{[1 [2 [4] [4]] [3]]}}}\\
        \dbr{[[[\;]_4 [\;]_5]_2 [\;]_3]_1} = \dbr{\alignbox{\tree{[1 [3] [2 [5] [4]]]}}} = \bbr{\alignbox{\tree{[1 [3] [2 [5] [4]]]}}, \quad \alignbox{\tree{[1 [3] [2 [4] [5]]]}}, \quad \alignbox{\tree{[1 [2 [5] [4]] [3]]}}, \quad \alignbox{\tree{[1 [2 [4] [5]] [3]]}}}
    \end{gather*}
    \caption{A schematic diagram of the ULN-trees $\dbr{[[[\;]_4 [\;]_4]_2 [\;]_3]_1}$ and $\dbr{[[[\;]_4 [\;]_5]_2 [\;]_3]_1}$, where ``$\bbr{\cdots}$'' denotes a set. }\label{fig_ntree}
\end{figure}

Combining \Cref{prop_mono_T} with \Cref{def_Ntree}(iv) yields $\T = \bigcup_{q \in \N} \T^q$ and $\T^q \subset \T^{q+1}$ for all $q \in \N$, which is the ULN-tree counterpart of \Cref{prop_mono_T}(i).
The next definition extends the notions of order, factorial, elementary coefficients, and elementary differentials from LN-trees to ULN-trees.
\Cref{prop_Labs_equ} ensures these quantities are well defined, that is, independent of the chosen representative $\Upsilon^{\prime} \in \dbr{\Upsilon}$.
Supplementary Materials provide diagrams and the associated functions for ULN-trees in $\T_r$ for $r$ up to $5$.

\begin{definition}[Functions associated with ULN-trees]\label{def_func_Ntree}
    For $\dbr{\Upsilon} \in \T$, its order $\abs{\dbr{\Upsilon}}$, factorial $\gamma\br{\dbr{\Upsilon}}$, elementary coefficients $A_j^m\br{\dbr{\Upsilon}}$ and elementary differentials $F_{t}^{\dbr{\Upsilon}}$ are respectively defined as the counterparts of a representative $\Upsilon^{\prime} \in \dbr{\Upsilon}$, i.e.,
    \begin{equation*}
    \begin{aligned}
        &\abs{\dbr{\Upsilon}} := \abs{\Upsilon^{\prime}}, \quad \gamma\br{\dbr{\Upsilon}} := \gamma\br{\Upsilon^{\prime}}, \\
        &A_j^m\br{\dbr{\Upsilon}} := A_j^m\br{\Upsilon^{\prime}}, \quad  F_{t}^{\dbr{\Upsilon}} := F_{t}^{\Upsilon^{\prime}}, \quad  F_{t, s}^{\dbr{\Upsilon}} := F_{t, s}^{\Upsilon^{\prime}}
    \end{aligned}
    \end{equation*}
    for $s \in (t, T)$, $t \in [0, T)$ and $j = 0, 1, \cdots, m$.
\end{definition}

The following proposition serves as the ULN-tree analogue of \Cref{def_tree0} and \Cref{prop_mono_T}(ii).
It provides a recursive construction of ULN-trees and the sets $\T^q$ for $q \in \N$, thereby simplifying their representation.

\begin{proposition}[Recursive construction of ULN-trees]\label{prop_const_ntree}
For $k, l \in \N$ and $\dbr{\Upsilon_1}, \dbr{\Upsilon_2}, \cdots, \dbr{\Upsilon_l} \in \T$, we define 
\begin{equation}\label{eq_defprod}
    \rbr{\dbr{\Upsilon_1}\cdots\dbr{\Upsilon_l}}_k := \bbr{\Upsilon \in \widetilde{\T}: \Upsilon \sim \rbr{\Upsilon_1^{\prime}\cdots\Upsilon_l^{\prime}}_k, \Upsilon_i^{\prime} \in \dbr{\Upsilon_i}, \;\; i = 1, 2, \cdots, l}
\end{equation}
with the convention that $\rbr{\dbr{\Upsilon_1}\cdots\dbr{\Upsilon_l}}_k = \bbr{\rbr{\;}_k}$ if $l = 0$.
The following statements then hold:
\begin{itemize}
    \item[(i)] For any $\Upsilon_1, \Upsilon_2, \cdots, \Upsilon_l \in \widetilde{\T}$ and $k \in \N$, the set $\rbr{\dbr{\Upsilon_1}\cdots\dbr{\Upsilon_l}}_k$ defined in \eqref{eq_defprod} is a ULN-tree and satisfies
    \begin{equation}\label{eq_property_dbr}
        \rbr{\dbr{\Upsilon_1}\cdots\dbr{\Upsilon_l}}_k = \dbr{\rbr{\Upsilon_1 \cdots \Upsilon_l}_k}. 
    \end{equation}
    \item[(ii)] For $q \in \N$, we have
        \begin{equation*}
            \T^{q+1} = \bbr{\rbr{\dbr{\Upsilon_1} \cdots \dbr{\Upsilon_l}}_k: \bbr{\dbr{\Upsilon_i}}_{i=1}^l \subset \T^{q},\; k \in \N, \; l \in \N}.
        \end{equation*}
\end{itemize}
\end{proposition}

\begin{proof}
    \textbf{(i)}
    Let $\Upsilon_1, \Upsilon_2, \cdots, \Upsilon_l \in \widetilde{\T}$ and $k \in \N$.
    To prove \eqref{eq_property_dbr}, we shall prove $\dbr{\rbr{\Upsilon_1 \cdots \Upsilon_l}_k} \subset \rbr{\dbr{\Upsilon_1}\cdots\dbr{\Upsilon_l}}_k$ and $\dbr{\rbr{\Upsilon_1 \cdots \Upsilon_l}_k} \supset \rbr{\dbr{\Upsilon_1}\cdots\dbr{\Upsilon_l}}_k$.

    For any $\rbr{\Upsilon_1 \cdots \Upsilon_l}_k \in \widetilde{\T}$, \Cref{def_Ntree}(i) implies that
    \begin{equation}\label{eq_rbrk_dbr}
        \dbr{\rbr{\Upsilon_1 \cdots \Upsilon_l}_k} = \bbr{\Upsilon^{\prime} \in \widetilde{\T}: \Upsilon^{\prime} \sim \rbr{\Upsilon_1 \cdots \Upsilon_l}_k}.
    \end{equation}
    Comparing \eqref{eq_defprod} with \eqref{eq_rbrk_dbr}, we deduce that $\dbr{\rbr{\Upsilon_1 \cdots \Upsilon_l}_k} \subset \rbr{\dbr{\Upsilon_1}\cdots\dbr{\Upsilon_l}}_k$.

    For any $\Upsilon \in \rbr{\dbr{\Upsilon_1}\cdots\dbr{\Upsilon_l}}_k$, by the definition in \eqref{eq_defprod}, there exist $\Upsilon_i^{\prime} \in \dbr{\Upsilon_i}$ for $i = 1, 2, \cdots, l$, such that,
    \begin{equation}\label{eq1_const_ntree}
        \Upsilon \sim \rbr{\Upsilon_1^{\prime}\cdots\Upsilon_l^{\prime}}_k, \quad \Upsilon_i^{\prime} \sim \Upsilon_i, \quad i = 1, 2, \cdots, l,
    \end{equation}
    where $\Upsilon_i^{\prime} \sim \Upsilon_i$ follows from \Cref{def_Ntree}(i).
    By \eqref{eq_equival_T} in \Cref{def_equival_T}, \eqref{eq1_const_ntree} implies $\Upsilon \sim \rbr{\Upsilon_1\cdots\Upsilon_l}_k$.
    Combining this with \eqref{eq_rbrk_dbr} yields $\Upsilon \in \dbr{\rbr{\Upsilon_1 \cdots \Upsilon_l}_k}$.
    Hence $\dbr{\rbr{\Upsilon_1 \cdots \Upsilon_l}_k} \supset \rbr{\dbr{\Upsilon_1}\cdots\dbr{\Upsilon_l}}_k$. 
    
    \textbf{(ii)}
    Let $q \in \N$.
    Firstly, we shall prove
    \begin{equation}\label{eq_Tqp1_supset}
        \T^{q+1} \supset \bbr{\rbr{\dbr{\Upsilon_1} \cdots \dbr{\Upsilon_l}}_k: \bbr{\dbr{\Upsilon_i}}_{i=1}^l \subset \T^{q},\; k \in \N, \; l \in \N}.
    \end{equation}
    For any $k, l \in \N$ and $\bbr{\dbr{\Upsilon_i}}_{i=1}^l \subset \T^{q}$, \Cref{def_Ntree} (iv) implies $\bbr{\Upsilon_i}_{i=1}^l \in \widetilde{\T}^q$.
    Consequently, \eqref{eq_property_dbr} and \Cref{prop_mono_T}(ii) yields
    \begin{equation*}
        \rbr{\dbr{\Upsilon_1} \cdots \dbr{\Upsilon_l}}_k = \dbr{\rbr{\Upsilon_1 \cdots \Upsilon_l}_k}, \quad \rbr{\Upsilon_1\cdots\Upsilon_l}_k \in \widetilde{\T}^{q+1},
    \end{equation*}
    respectively. 
    The above result shows that $\rbr{\dbr{\Upsilon_1} \cdots \dbr{\Upsilon_l}}_k \in \T^{q+1}$, recalling \Cref{def_Ntree} (iv), again.
    This establishes \eqref{eq_Tqp1_supset}.

    Then we shall prove
    \begin{equation}\label{eq_Tqp1_subset}
        \T^{q+1} \subset \bbr{\rbr{\dbr{\Upsilon_1} \cdots \dbr{\Upsilon_l}}_k: \bbr{\dbr{\Upsilon_i}}_{i=1}^l \subset \T^{q},\; k \in \N, \; l \in \N}.
    \end{equation}
    For any $\dbr{\Upsilon} \in \T^{q+1}$, by \Cref{def_Ntree} (iv) we have $\Upsilon \in \widetilde{\T}^{q+1}$.
    By \Cref{prop_mono_T}(ii), there exist $k, l \in \N$ and $\Upsilon_1, \Upsilon_2, \cdots, \Upsilon_l \in \widetilde{\T}^q$ such that $\Upsilon = \rbr{\Upsilon_1 \cdots \Upsilon_l}_k$.
    Hence $\dbr{\Upsilon} = \dbr{\rbr{\Upsilon_1 \cdots \Upsilon_l}_k} = \rbr{\dbr{\Upsilon_1}\cdots\dbr{\Upsilon_l}}_k$, where the second equality follows from \eqref{eq_property_dbr}.
    This establishes \eqref{eq_Tqp1_subset}.

    The desired result is obtained by combining \eqref{eq_Tqp1_supset} and \eqref{eq_Tqp1_subset}.
\end{proof}

By \Cref{prop_const_ntree}(ii), every ULN-tree $\Upsilon \in \T \backslash \T^1$ can be written as a composition of ULN-trees that themselves admit a similar decomposition, mirroring the recursive construction of LN-trees in \Cref{def_tree0}.
This recursive structure allows us to extend the notion of branches from \Cref{def_tree0} to ULN-trees, as formalized in \Cref{def_branch_ntree}.

\begin{definition}[Branch of ULN-tree]\label{def_branch_ntree}
    Any $[\;]_k \in \T^1$ is defined to have no branches.
    For $q = 2, 3, \cdots$, the ULN-tree $\rbr{\Upsilon_1 \cdots \Upsilon_l}_k \in \T^q \backslash \T^{q-1}$ is said to have a branch $\Upsilon^{\prime}$ if there exists $i_0 \in \bbr{1, 2, \cdots, l}$ such that either $\Upsilon^{\prime} = \Upsilon_{i_0}$ or $\Upsilon^{\prime}$ is a branch of $\Upsilon_{i_0}$.
\end{definition}

Building on \Cref{def_branch_ntree}, we next state \Cref{prop_sim_ordcondi}, which clarifies how the elementary coefficients of a ULN-tree are related to those of its branches.
This proposition will be used in \cref{sec_ord_condi} to simplify the order conditions for \Cref{sch_rk}.

\begin{proposition}\label{prop_sim_ordcondi}
    Let $\Upsilon \in \T \backslash \T^1$ and $\Upsilon^{\prime}$ be a branch of $\Upsilon$.
    If 
    \begin{equation}\label{eq_bar_eq0}
        A_i^m\br{\Upsilon^{\prime}} = 0, \;\; i = 0, 1, \cdots, m,
    \end{equation}
    then the above result also holds with $\Upsilon$ in place of $\Upsilon^{\prime}$, i.e.,
    \begin{equation}\label{eq1_bar_eq0}
        A_i^m\br{\Upsilon} = 0, \;\; i = 0, 1, \cdots, m. 
    \end{equation}
\end{proposition}

\begin{proof}
    We first prove \eqref{eq1_bar_eq0} is valid for $\Upsilon = \rbr{\Upsilon_1\cdots \Upsilon_l}_k \in \T^2 \backslash \T^1$.
    By \Cref{prop_const_ntree}(ii), $\Upsilon \in \T^2 \backslash \T^1$ implies $\Upsilon_1, \Upsilon_2, \cdots, \Upsilon_l \in \T^1$. 
    By \Cref{def_branch_ntree}, each $\Upsilon_1, \Upsilon_2, \cdots, \Upsilon_l \in \T^1$ has no branches. 
    Consequently, if $\Upsilon^{\prime}$ is a branch of $\rbr{\Upsilon_1\cdots \Upsilon_l}_k$, then it must be one of $\Upsilon_1, \Upsilon_2, \cdots, \Upsilon_l$; that is, $\Upsilon^{\prime} = \Upsilon_{p_0}$ for some $p_0 \in \bbr{1, 2, \cdots, l}$. 
    Therefore, \eqref{eq_bar_eq0} yields
    \begin{equation*}
        \prod_{p = 1}^l A_i^m\br{\Upsilon_p} = 0, \quad i = 0, 1, \cdots, m.
    \end{equation*}
    Consequently,
    \begin{equation}\label{eq_Ajm}
        A_j^m\br{\Upsilon} = \sum_{j' = 1}^{m} a_{jj'} \br{c_j - c_{j'}}^{k} \prod_{p = 1}^l A_{j'}^m\br{\Upsilon_p} = 0, \quad j = 0, 1, \cdots, m,
    \end{equation}
    where we adopt the convention $a_{0j'} := b_{j'}$ and the first equality follows from the definitions in \eqref{eq_recur_Ai} and \eqref{eq_recur_A0}.
    This establishes \eqref{eq1_bar_eq0} for $\Upsilon \in \T^2 \backslash \T^1$.
    
    We now prove \eqref{eq1_bar_eq0} by mathematical induction.
    Specifically, for any $q \ge 2$, assuming \eqref{eq1_bar_eq0} holds for all $\Upsilon \in \T^q \backslash \T^{1}$, we then show that \eqref{eq1_bar_eq0} also holds for every $\Upsilon \in \T^{q+1} \backslash \T^{1}$.
    Actually, for any $\Upsilon \in \T^{q+1} \backslash \T^{1}$, \Cref{prop_const_ntree}(ii) implies that $\Upsilon$ can be expressed as $\Upsilon = \rbr{\Upsilon_1 \cdots \Upsilon_l}_k$ for some $k, l \in \N$ and $\bbr{\Upsilon_p}_{p=1}^l \subset \T^q$.
    By \Cref{def_branch_ntree}, there exists $p_0 \in \bbr{1, 2, \cdots, l}$, such that, $\Upsilon^{\prime} = \Upsilon_{p_0}$ or $\Upsilon^{\prime}$ is a branch of $\Upsilon_{p_0}$.
    If $\Upsilon^{\prime} = \Upsilon_{p_0}$, then $A_i^m\br{\Upsilon_{p_0}} = A_i^m\br{\Upsilon^{\prime}}$. Together with \eqref{eq_bar_eq0}, this implies
    \begin{equation}\label{eq_proof_brancheq0}
        A_i^m\br{\Upsilon_{p_0}} = 0, \;\; i = 0, 1, \cdots, m.
    \end{equation}
    Otherwise, $\Upsilon^{\prime}$ is a branch of $\Upsilon_{p_0}$. By \Cref{def_branch_ntree}, all ULN-trees in $\T^1$ have no branches; hence $\Upsilon_{p_0} \in \T^q \backslash \T^{1}$ and the induction hypothesis applies, which again yields \eqref{eq_proof_brancheq0}.
    Hence, in either case \eqref{eq_proof_brancheq0} holds and further implies
    \begin{equation}\label{eq1_recur_Ai}
        \prod_{p = 1}^l A_{i}^m\br{\Upsilon_{p}} = 0, \;\; i = 0, 1, \cdots, m.
    \end{equation}
    Further using the definitions in \eqref{eq_recur_Ai} and \eqref{eq_recur_A0}, we can deduce the same result as in \eqref{eq_Ajm} for $\Upsilon \in \T^{q+1} \backslash \T^{1}$, which again yields \eqref{eq1_bar_eq0}.
    This concludes the induction step and completes the proof.
\end{proof}

\section{Consistency analysis}\label{sec_theoana}

In this section, we begin by expanding the local truncation errors of \Cref{sch_rk} using the extended Butcher trees developed in \cref{sec_theo_tree}, and then derive the order conditions.

The analysis in this section relies on the following assumptions.

\begin{assumption}\label{assu_regu}
    The solution $\br{Y, Z}$ of the BSDE~\eqref{bsde} satisfies the Feynman-Kac formula~\eqref{eq_Ytut} with 
    \begin{equation*}
        u \in C_b^{r+1, 2r+4}, \quad f \in C_b^{r+1, 2r+3, 2r+3, 2r+3}. 
    \end{equation*}
    Moreover, the initial value $X_0$ in \eqref{bsde} satisfies $\mathbb{E}[|X_0|^k] < \infty$ for all $k \in \N^+$.
\end{assumption}

\begin{assumption}\label{assu_coeff}
    There exists a constant $C>0$, independent of the time partition $\pi_N$ defined in \eqref{eq_timepart}, such that for $1 \leq i \leq j \leq m$,
    \begin{equation}\label{eq_ajieq0_maxabc}
        \abs{a_{ij}} + \abs{b_i} + \abs{c_i} + \abs{c_0} \leq C, \;\; a_{ji} = 0, \;\; c_{i-1} - c_{i} > C^{-1}.
    \end{equation}
\end{assumption}

\Cref{assu_regu} ensures the well-posedness of the elementary differentials defined in \Cref{def_elediff}, and thus justifies the Taylor expansions used in the consistency analysis.
In principle, the boundedness assumptions on the derivatives of $u$ and $f$ in \Cref{assu_regu} can be relaxed to polynomial-growth conditions, at the cost of more delicate remainder estimates; see \cite[Appendix A]{Cai2026Deep} for an example of such an analysis.
For \Cref{assu_coeff}, the first requirement in \eqref{eq_ajieq0_maxabc} is mild under regular time partitions; 
the second is intrinsic to explicit schemes; 
the third reflects the order condition (1) in \Cref{tab_ordcond}, strengthened by a uniform constant $C^{-1}$, and is likewise mild for regular partitions.

The following proposition establishes sufficient conditions for the well-posedness of \Cref{def_elediff} for LN-trees with order up to $r+1$.

\begin{proposition}\label{prop_existence_F}
    Under \Cref{assu_regu}, for any $\Upsilon \in \widetilde{\T}_{r+1} := \big\{\Upsilon \in \widetilde{\T}: \abs{\Upsilon} \leq r+1 \big\}$,
    the limits $\lim_{s \to t+}$ in \eqref{eq_defFt0_Up} and \eqref{eq_recur_F} exist.
    Consequently, the elementary differentials $F_{t}^{\Upsilon}$ and $F_{t,s}^{\Upsilon}$ exist for all $t \in [0, T)$ and $s \in (t, T)$.
\end{proposition}

\begin{proof}
    The proof is divided into four parts.

    \subsubsection*{(i) The operator $\mathcal{H}_{t, s}$ and its iteration formula}

    For any $0 \leq t < s < T$ and any $\R^p$-valued function $g \in C_b^{0, 1}$, define the operator $\mathcal{H}_{t, s}$ by
    \begin{equation}\label{eq_def_H}
        \mathcal{H}_{t, s}(g) := \vbr{\big}{\Et{g\br{s, X_s}},\, \Et{g\br{s, X_s} \Delta W_{t, s}^{\top} \,/ (s - t)}}.
    \end{equation}
    Following the same argument as in \eqref{eq_ydwt_z}, with $g$ in place of $u$,
    we have 
    \begin{equation}
        \Et{g(s, X_s) \Delta W_{t, s}^{\top}} = (s - t) \Et{\mathcal{L}_1 g(s, X_s)},
    \end{equation}
    where the differential operator $\mathcal{L}_1$ is defined by
    \begin{equation*}
        \mathcal{L}_1 g(s, X_s) := \partial_x g(s, X_s) \sigma.
    \end{equation*}
    Then the operator $\mathcal{H}_{t, s}$ in \eqref{eq_def_H} admits the following representation:
    \begin{equation}\label{eq1_H}
        \mathcal{H}_{t, s}(g) = \vbr{\big}{\Et{g\br{s, X_s}},\, \Et{\mathcal{L}_1 g(s, X_s)}}
    \end{equation}
    for any $g \in C_b^{0, 1}$ and $0 \leq t < s \leq T$.
    By the moment assumption on $X_0$ and the definition of $X$ in \eqref{bsde}, it follows that, for any $t \in [0, T]$,
    \begin{equation*}
        \sup_{s \in [t, T]} \Et{|X_s|^k} < \infty, \quad k \in \N^+.
    \end{equation*}
    This bound, together with the representation \eqref{eq1_H}, implies that $\mathcal{H}_{t, s}(g)$ is well defined for any $g \in C_b^{0, 1}$.  
    
    For any $m \geq 1$, assume $g$ satisfies the following regularity: 
    \begin{equation}\label{eq_reg2}
        g \in C_b^{m, 2m+1}, 
    \end{equation}
    By the Dynkin formula \cite[Theorem 7.4.1]{Oksendal2003Stochastic}, we have 
    \begin{align*}
        \Et{g\br{s, X_s}} &= g\br{t, X_t} + \int_t^s \Et{\opL_0 g\br{r, X_r}} \mathrm{d}r,\\
        \Et{\mathcal{L}_1 g\br{s, X_s}} &= \mathcal{L}_1 g\br{t, X_t} + \int_t^s \Et{\opL_0 \mathcal{L}_1 g\br{r, X_r}} \mathrm{d}r,
    \end{align*} 
    where the differential operator $\opL_0$ is defined immediately after \eqref{eq_pde}.
    Substituting the above identities into \eqref{eq1_H} and differentiating both sides with respect to $s$, we obtain
    \begin{equation}\label{eq_partial_H}
        \partial_s \mathcal{H}_{t, s}(g) = \vbr{\big}{\Et{\opL_0 g\br{s, X_s}},\, \Et{\mathcal{L}_1 \opL_0 g\br{s, X_s}}} = \mathcal{H}_{t, s}(\opL_0 g),
    \end{equation}
    where the second equality follows from the commutation of $\opL_0$ and $\mathcal{L}_1$ together with the representation \eqref{eq1_H} of $\mathcal{H}_{t, s}$.
    As shown above, $\mathcal{H}_{t, s}(h)$ is well defined whenever $h \in C_b^{0, 1}$.
    The regularity \eqref{eq_reg2} allows us to apply \eqref{eq_partial_H} successively to $g, \opL_0 g, \cdots, \opL_0^{m-1} g$, and to obtain
    \begin{equation}\label{eq_iterate_H1}
        \partial_s^k \mathcal{H}_{t, s}(g) = \mathcal{H}_{t, s}(\opL_0^k g), \quad s \in (t, T), \;\; t \in [0, T), \;\; k = 0, 1, \cdots, m,
    \end{equation}
    for any $m \in \N$ and any $g$ satisfying \eqref{eq_reg2}.

    \subsubsection*{(ii) Elementary differentials for single-node trees}
    
    Comparing \eqref{eq_def_thetaF} with the definition of $\mathcal{H}_{t, s}$ in \eqref{eq_def_H}, we have 
    \begin{equation}\label{eq_FH}
        \Et{F_{t}\br{s, \Theta_s}} = \mathcal{H}_{t, s}(\tilde{f}), \quad s \in (t, T), \; t \in [0, T).
    \end{equation}
    where the function $\tilde{f}: [0, T] \times \R^d \to \R^p$ is defined by
    \begin{equation*}
    \tilde{f}(t, x) := f\br{t, x, u\br{t, x}, \mathcal{L}_1 u\br{t, x}}, \quad (t, x) \in [0, T] \times \R^d. 
    \end{equation*}
    Since $u \in C_b^{r+1, 2r+4}$ and $\mathcal{L}_1 u \in C_b^{r+1, 2r+3}$, repeated applications of the chain rule together with the regularity of $f$ yield 
    \begin{equation}\label{eq_reg_tf}
        \tilde{f} \in C_b^{r+1, 2r + 3}. 
    \end{equation}
    This regularity allows us to apply \eqref{eq_iterate_H1} with $g = \tilde{f}$ and $m = r + 1$ to obtain 
    \begin{equation}\label{eq_iterate_H2}
        \partial_s^k \Et{F_{t}\br{s, \Theta_s}} = \partial_s^k \mathcal{H}_{t, s}(\tilde{f}) = \mathcal{H}_{t, s}(\opL_0^k \tilde{f})
    \end{equation}
    for $k = 0, 1, \cdots, r+1$, where the first equality follows from \eqref{eq_FH} and the second equality follows from \eqref{eq_iterate_H1}.
    
    Recalling the second definition in \eqref{eq_recur_F}, the representation \eqref{eq_iterate_H2} implies the existence of $F_{t, s}^{[\;]_k}$ together with the representation
    \begin{equation}\label{eq_FH1}
        F_{t, s}^{[\;]_k} = \mathcal{H}_{t, s}(\opL_0^k \tilde{f}) = \vbr{\big}{\Et{\opL_0^k \tilde{f}\br{s, X_s}},\, \Et{\mathcal{L}_1 \opL_0^k \tilde{f}\br{s, X_s}}}
    \end{equation}
    for $k = 0, 1, \cdots, r+1$, where the second equality follows from \eqref{eq1_H}.
    Moreover, the regularity stated in \eqref{eq_reg_tf} allows us to pass to the limit $s \to t+$ in \eqref{eq_FH1}. 
    This yields the existence of $F_t^{[\;]_k}$ together with the representation
    \begin{equation}\label{eq_FH2}
        F_t^{[\;]_k} = \vbr{\big}{\opL_0^k \tilde{f}\br{t, X_t},\, \mathcal{L}_1 \opL_0^k \tilde{f}\br{t, X_t}}, \quad k = 0, 1, \cdots, r+1.
    \end{equation}

    \subsubsection*{(iii) Induction statement}
    
    We prove by induction on $q = 1, 2, \cdots, r+1$ that the following statement holds:
    for every $\Upsilon \in \widetilde{\T}_{r+1} \cap \widetilde{\T}^q$, there exists a deterministic function $g^{\Upsilon}: [0, T] \times \R^d \to \R^p$ such that the following hold:
    \begin{align}
        &F_{t, s}^{\Upsilon} = \mathcal{H}_{t, s}(g^{\Upsilon}), \quad s \in (t, T), \label{eq_ind1} \\
        &F_t^{\Upsilon} = \vbr{\big}{g^{\Upsilon}\br{t, X_t},\, \mathcal{L}_1 g^{\Upsilon}\br{t, X_t}}, \quad t \in [0, T), \label{eq_ind2} \\
        &g^{\Upsilon} \in C_b^{2 + r - \abs{\Upsilon}, \,3 + 2(r + 1 - \abs{\Upsilon})}. \label{eq_ind3}
    \end{align}
    By \Cref{def_tree0}, all $\Upsilon \in \widetilde{\T}^1$ are of the form $[\;]_k$ for some $k \in \N$. 
    Thus \eqref{eq_FH1} and \eqref{eq_FH2} imply the claim holds for $q = 1$ with $g^{[\;]_k} := \opL_0^k \tilde{f}$, $k = 0, 1, \cdots, r$.
    
    For any $\Upsilon \in \widetilde{\T}_{r+1}$, the inequality $\abs{\Upsilon} \leq r + 1$ implies that $g^{\Upsilon} \in C_b^{1, 3}$ for any $\Upsilon \in \widetilde{\T}_{r+1}$. 
    Then the result~\eqref{eq_iterate_H1} with $k=1$ guarantee the continuity of $s \mapsto F_{t, s}^{\Upsilon}$ at $s = t+$. 
    Hence, the induction statement above implies the existence of the two limits in \eqref{eq_defFt0_Up} and \eqref{eq_recur_F} for the corresponding $\Upsilon$.
    The remaining part of the proof is devoted to establishing the induction statement.
    
    \subsubsection*{(iv) Induction step}

    Assume that the claim holds for all trees in $\widetilde{\T}_{r+1} \cap \widetilde{\T}^{q-1}$. 
    By \Cref{prop_mono_T}(ii), any $\Upsilon \in (\widetilde{\T}_{r+1} \cap \widetilde{\T}^{q}) \backslash \widetilde{\T}^{q-1}$ admits the representation $\Upsilon = \rbr{\Upsilon_1 \cdots \Upsilon_l}_k$ with 
    \begin{equation*}
        \Upsilon_1, \cdots, \Upsilon_l \in \widetilde{\T}^{q-1}, \quad k \in \N, \quad l \in \N^+,
    \end{equation*}
    and $\Upsilon_1, \cdots, \Upsilon_l \in \widetilde{\T}_{r+1}$.
    Then, for each $j = 1, 2, \cdots, l$, the induction hypothesis~\eqref{eq_ind2} yields
    \begin{equation*}
        F_t^{\Upsilon_j} = \vbr{\big}{g^{\Upsilon_j}\br{t, X_t},\, \mathcal{L}_1 g^{\Upsilon_j}\br{t, X_t}}, \quad t \in [0, T).
    \end{equation*}
    Define the deterministic function $h^{\Upsilon}: [0, T] \times \R^d \to \R^p$ by
    \begin{equation}\label{eq_def_hUp}
        h^{\Upsilon}(t, x) := \prod_{j=1}^l \bbr{g^{\Upsilon_j}\br{t, x} \partial_y + \mathcal{L}_1 g^{\Upsilon_j}\br{t, x} \partial_z} f\br{t, x, u\br{t, x}, \mathcal{L}_1 u\br{t, x}},
    \end{equation}
    where $\partial_y$ and $\partial_z$ are applied to the $(y, z)$ arguments of $(t, x, y, z) \mapsto f(t, x, y, z)$. 
    
    Now we analyze the regularity of $h^{\Upsilon}$.
    The induction hypothesis~\eqref{eq_ind3} implies that for each $j = 1, 2, \cdots, l$,
    \begin{equation*}
        g^{\Upsilon_j} \in C_b^{r + 2 - \abs{\Upsilon_j}, \,3 + 2(r + 1 - \abs{\Upsilon_j})}, \quad \mathcal{L}_1 g^{\Upsilon_j} \in C_b^{r + 2 - \abs{\Upsilon_j}, \,2 + 2(r + 1 - \abs{\Upsilon_j})},
    \end{equation*}
    where the second inclusion uses that $\mathcal{L}_1$ acts only on $x$.
    The regularity of $u$ stated in \Cref{assu_regu} ensures that 
    \begin{equation*}
        u \in C_b^{r+1, 2r+4}, \quad \mathcal{L}_1 u \in C_b^{r+1, 2r+3}.
    \end{equation*}
    Moreover, the right-hand side of \eqref{eq_def_hUp} only involves $(y, z)$-derivatives of $f$ up to order $l$, and thus the regularity of $f$ in \Cref{assu_regu} ensures the resulting derivatives of $f$ are of class $C_b^{r+1,\, 2r+3,\, 2r+3 - l,\, 2r+3 - l}$.
    In summary, the right-hand side of \eqref{eq_def_hUp} consists of linear combinations and compositions of functions with regularity at least $C_b^{r_1, \,r_2}$, where
    \begin{align*}
        r_1 &:= \min\left\{r + 1, \,r + 2 - \max_{j = 1, 2, \cdots, l}\abs{\Upsilon_j}\right\}, \\ 
        r_2 &:= \min\left\{2r + 4, \, 2r + 3 - l, \, 2 + 2(r + 1 - \max_{j = 1, 2, \cdots, l}\abs{\Upsilon_j})\right\}. 
    \end{align*}
    Therefore, we have $h^{\Upsilon} \in C_b^{r_1, r_2}$.
    Using also the definition of $\abs{\Upsilon}$ in \eqref{eq_recur_alp}, we obtain
    \begin{align*}
        r_1 - k &\geq r + 1 - k - \sum_{j=1}^l \abs{\Upsilon_j} = r + 2 - \abs{\Upsilon} , \\ 
        r_2  - 2 k &\geq 2 + 2(r + 1 - \sum_{j=1}^l \abs{\Upsilon_j} - k - 1) = 3 + 2(r + 1 - \abs{\Upsilon}), 
    \end{align*}
    which implies 
    \begin{equation}\label{eq_reg_hUp}
        h^{\Upsilon} \in C_b^{2 + r - \abs{\Upsilon} + k,\, 3 + 2(r + 1 - \abs{\Upsilon} + k)}, \quad \opL_0^k h^{\Upsilon} \in C_b^{2 + r - \abs{\Upsilon},\, 3 + 2(r + 1 - \abs{\Upsilon})}.
    \end{equation}
    Here the second result is the counterpart of \eqref{eq_ind3} for $h^{\Upsilon}$. 
      
    By the induction hypothesis~\eqref{eq_ind2}, the second definition in \eqref{eq_recur_F} can be rewritten in terms of $\mathcal{H}_{t, s}$ and $h^{\Upsilon}$ as
    \begin{equation*}
        F_{t, s}^{\Upsilon} = \partial_s^k \Et{\br{F_s^{\Upsilon_1} \cdot \partial_{\theta}} \cdots \br{F_s^{\Upsilon_l} \cdot \partial_{\theta}} F_t\br{s, \Theta_s}} = \partial_s^k \mathcal{H}_{t, s}(h^{\Upsilon}) = \mathcal{H}_{t, s}(\opL_0^k h^{\Upsilon}), 
    \end{equation*}
    where the last equality follows from the regularity of $h^{\Upsilon}$ stated in \eqref{eq_reg_hUp} and the iterate formula \eqref{eq_iterate_H1}.
    The above identity implies the analogue of \eqref{eq_ind1} with $\opL_0^k h^{\Upsilon}$ in place of $g^{\Upsilon}$.  
    
    Applying the representation \eqref{eq1_H} once more, we obtain
    \begin{equation*}
        F_{t, s}^{\Upsilon} = \mathcal{H}_{t, s}(\opL_0^k h^{\Upsilon}) = \vbr{\big}{\Et{\opL_0^k h^{\Upsilon}\br{s, X_s}},\, \Et{\mathcal{L}_1 \opL_0^k h^{\Upsilon}\br{s, X_s}}}. 
    \end{equation*}
    In particular, $F_{t, s}^{\Upsilon}$ exists for all $s \in (t, T)$ and $t \in [0, T)$.
    Letting $s \to t+$, we also obtain the analogue of \eqref{eq_FH2} with $\opL_0^k h^{\Upsilon}$ in place of $g^{\Upsilon}$, that is,
    \begin{equation*}
        F_t^{\Upsilon} = \vbr{\big}{\opL_0^k h^{\Upsilon}\br{t, X_t},\, \mathcal{L}_1 \opL_0^k h^{\Upsilon}\br{t, X_t}}.
    \end{equation*}
    This completes the induction step and hence the proof.
\end{proof}

\subsection{Expansions of local truncation errors}\label{sec_expand_R}
 
For $n = 0, 1, \cdots, N-1$ and $i = 0, 1, \cdots, m$, define $R^{n,i}$ and $R^{n}$ as the local truncation errors of the $i$-th internal stage and the final stage of \Cref{sch_rk}, respectively, namely,
\begin{equation}\label{eq_Rni}
    R^{n, i} := \widehat{\Theta}^{n, i} - \Theta_{t_{n, i}}, \quad R^{n} := R^{n, 0} := \widehat{\Theta}^{n, 0} - \Theta_{t_{n}},
\end{equation}
where $\widehat{\Theta}^{n, i}$ are given by
\begin{equation}\label{eq_defhattht}
\left\{\begin{aligned}
    &\widehat{\Theta}^{n, i} := \Etni{\Theta_{t_{n+1}}} + \Delta t_n \sum_{j=i+1}^m a_{ij} \Etni{F_{t_{n, i}} \br{t_{n, j}, \widehat{\Theta}^{n, j}}}, \;\; i = m, \cdots, 1, \\
    &\widehat{\Theta}^{n, 0} := \Etn{\Theta_{t_{n+1}}} + \Delta t_n \sum_{j=1}^m b_{j} \Etn{F_{t_{n}} \br{t_{n, j}, \widehat{\Theta}^{n, j}}}.
\end{aligned}\right.
\end{equation}
The error term $R^{n}$ defined in \eqref{eq_Rni} is the local truncation error of \Cref{sch_rk} at $t = t_n$.

\Cref{thm_taylor} below is one of the main results of this paper. 
Since $R^{n}=R^{n,0}$, it yields a Taylor expansion of the local truncation error for \Cref{sch_rk}.

\begin{theorem}[Error expansion]\label{thm_taylor}
Under \Cref{assu_regu,assu_coeff}, the local truncation errors defined in \eqref{eq_Rni} admit the following expansion:
\begin{equation} \label{eq2_Rnj_indc}
    R^{n, j} = \sum_{\Upsilon \in \T_r} \frac{\alpha\br{\Upsilon}}{\gamma\br{\Upsilon}} A_j^m\br{\Upsilon} \br{\Delta t_n}^{\abs{\Upsilon}} F_{t_{n, j}}^{\Upsilon} + \mO\br{\br{\Delta t}^{r+1}}, \quad \Delta t \to 0
\end{equation}
for $j = 0, 1, \cdots, m$, where the remainder term in \eqref{eq2_Rnj_indc} holds uniformly for $0 \leq n \leq N-1$, and
\begin{equation}\label{eq_defTrm}
    \T_r := \bbr{\Upsilon \in \T: \abs{\Upsilon} \leq r}
\end{equation}
with $\T$ introduced in \Cref{def_Ntree}(iv), denoting the set of all ULN-trees.
\end{theorem}

\begin{proof}
    
    The proof is divided into six parts.
    
    \subsubsection*{(i) Reformulation of \eqref{eq2_Rnj_indc}}
    
    We first show that \eqref{eq2_Rnj_indc} is implied by the following expansion:
    \begin{equation}\label{eq_Rnj_indc}
        R^{n, j} = \sum_{\Upsilon \in \widetilde{\T}_r} \frac{A_j^m\br{\Upsilon}}{\gamma\br{\Upsilon}} \br{\Delta t_n}^{\abs{\Upsilon}} F_{t_{n, j}}^{\Upsilon} + \mO\br{\br{\Delta t}^{r+1}}
    \end{equation}
    for $j = 0, 1, \cdots, m$, where
    \begin{equation}\label{eq_deftiTrm}
        \widetilde{\T}_r := \{\Upsilon \in \widetilde{\T}: \abs{\Upsilon} \leq r\},
    \end{equation}
    and $\widetilde{\T}$ denotes the set of all LN-trees introduced at the end of \Cref{def_tree0}.
    In fact, \Cref{prop_Labs_equ} ensures that the summands in \eqref{eq_Rnj_indc} coincide for any two equivalent LN-trees $\Upsilon$.
    Hence, we group the summands by equivalence classes, 
    and use $\alpha(\Upsilon)$ introduced in \Cref{def_Ntree}(iii) to indicate the number of LN-trees equivalent to $\Upsilon$. 
    This yields \eqref{eq2_Rnj_indc}.
    
    It remains to prove \eqref{eq_Rnj_indc} for $j = 0, 1, \cdots, m$. 
    We first verify the case $j = m$.
    From \eqref{eq_Rni} we obtain $R^{n, m} = 0$.
    In addition, by \Cref{def_weight}, the second condition in \eqref{eq_ajieq0_maxabc}, and $c_m = 0$ assumed in \Cref{sch_rk}, we obtain $A_m^m\br{\Upsilon} = 0$ for any $\Upsilon \in \widetilde{\T}$.
    These two observations establish \eqref{eq_Rnj_indc} for $j = m$.
    
    In the remainder of the proof, we shall prove \eqref{eq_Rnj_indc} by induction on $j = m, m-1, \cdots, 0$.
    Specifically, for any $i \in \bbr{0, 1, \cdots, m-1}$, assuming \eqref{eq_Rnj_indc} holds for $j = m, m-1, \cdots, i+1$, we then show that \eqref{eq_Rnj_indc} also holds for $j = i$.
    For notational convenience, set $a_{0i} := b_i$ for $i = 1, 2, \cdots, m$.
    With this convention, we substitute \eqref{eq_defhattht} into \eqref{eq_Rni} and rewrite the two equations in \eqref{eq_Rni} in a unified form
    \begin{equation}\label{eq_RniEni}
        R^{n, i} = E^{n, i} + \Delta t_n \sum_{j = i+1}^m a_{ij} \Etni{\Delta F_{t_{n, i}}^{n, j}}, \quad i = m, m-1, \cdots, 0,
    \end{equation}
    where
    \begin{align}
        & E^{n, i} := \Etni{\Theta_{t_{n+1}}} - \Theta_{t_{n, i}} + \Delta t_n \sum_{j=i+1}^m a_{ij} \Etni{F_{t_{n, i}}\br{t_{n, j}, \Theta_{t_{n, j}}}}, \label{eq_Eni} \\
        & \Delta F_{t_{n, i}}^{n, j} := F_{t_{n, i}} (t_{n, j}, \widehat{\Theta}^{n, j}) - F_{t_{n, i}}(t_{n, j}, \Theta_{t_{n, j}}). \label{eq_RFnj}
    \end{align}

    \subsubsection*{(ii) Expansion of $E^{n, i}$}
    For $i = 0, 1, \cdots, m$ and $l = 1, 2, \cdots, r$, it holds that
    \begin{equation}\label{eq_ptETheta_F}
        \partial_t^l \Etni{\Theta_t} = -\partial_t^{l-1} \Etni{F_{t_{n, i}}(t, \Theta_t)} = - F_{t_{n, i}, t}^{[\;]_{l-1}}, \quad t \in (t_{n, i}, T],
    \end{equation}
    where the first and second equalities follow from \eqref{eq_refode} and the second equation in \eqref{eq_defFt0_Up}, respectively.
    Using \eqref{eq_ptETheta_F}, we rewrite \eqref{eq_Eni} into
    \begin{equation}\label{eq_Eni2}
    \begin{aligned}
        E^{n, i} =\;& \int_{t_{n, i}}^{t_{n+1}} \partial_t \Etni{\Theta_t} \di t + \Delta t_n \sum_{j=1}^m a_{ij} F_{t_{n, i}, t_{n, j}}^{[\;]_0}\\
        =\;& -\int_{t_{n, i}}^{t_{n+1}} F_{t_{n, i}, t}^{[\;]_0} \di t + \Delta t_n \sum_{j=1}^m a_{ij} F_{t_{n, i}, t_{n, j}}^{[\;]_0},
    \end{aligned}
    \end{equation}
    where we have invoked the second condition in \eqref{eq_ajieq0_maxabc} to extend the summation range to $j = 1, \cdots, m$.
    By Taylor's expansion of $s \mapsto F_{t, s}^{[\;]_0}$ at $s = t$, 
    we have 
    \begin{equation}\label{eq_Fts}
        F_{t, s}^{[\;]_0} = \sum_{l=0}^{r-1} \frac{(s - t)^l}{l!} F_{t}^{[\;]_l} + \mO\br{\abs{s - t}^r}, 
    \end{equation}
    where $F_{t}^{[\;]_l}$ is defined in \eqref{eq_defFt0_Up}.
    Substituting \eqref{eq_Fts} into \eqref{eq_Eni2}, we obtain
    \begin{equation}\label{eq_expandEni}
    \begin{aligned}
        E^{n, i} =\;& \sum_{l=1}^r \vbr{\Big}{-\frac{c_i^l}{l!} + \sum_{j=1}^m \frac{a_{ij} \br{c_i - c_j}^{l-1}}{(l-1)!}} \br{\Delta t_n}^{l} F_{t_{n, i}}^{[\;]_{l-1}} + \mO\br{\abs{\Delta t_n}^{r+1}} \\
        =\;& \sum_{l=1}^{r} \frac{A_i^m\br{[\;]_{l-1}}}{\gamma\br{[\;]_{l-1}}} \br{\Delta t_n}^{l} F_{t_{n, i}}^{[\;]_{l-1}} + \mO\br{\abs{\Delta t_n}^{r+1}}                    \\
        =\;& \sum_{\Upsilon \in \widetilde{\T}_r^1} \frac{A_i^m\br{\Upsilon}}{\gamma(\Upsilon)} \br{\Delta t_n}^{\abs{\Upsilon}} F_{t_{n, i}}^{\Upsilon} + \mO\br{\abs{\Delta t_n}^{r+1}},
    \end{aligned}
    \end{equation}
    where the second equality follows from the definitions in \eqref{eq0_recur_alp} and \eqref{eq_defAi}, and $\widetilde{\T}_r^1$ in the last line is defined by
    \begin{equation}\label{eq_defTr1}
        \widetilde{\T}_r^1 := \bbr{\rbr{\;}_{k} \in \widetilde{\T}: 0 \leq k \leq r-1}. 
    \end{equation}

    \subsubsection*{(iii) Expansion of $\sum_{j = i+1}^m a_{ij} \mathbb{E}_{t_{n,i}}[\Delta F_{t_{n, i}}^{n, j}]$}
    
    By the first equality in \eqref{eq_expandEni}, we have $E^{n, i} = \mO\br{\Delta t_n}$.
    Substituting this into \eqref{eq_RniEni} in place of $E^{n, i}$, we obtain a prior estimate of $R^{n, i}$ as
    \begin{equation}\label{eq_preest_Rni}
        R^{n, i} = \mO\br{\Delta t_n}, \quad i = m, m-1, \cdots, 0.
    \end{equation}
    Applying the Taylor expansion to \eqref{eq_RFnj}, we obtain
    \begin{align}\label{eq_EtniDelFninj}
        \Etni{\Delta F_{t_{n, i}}^{n, j}} =\; & \sum_{l=1}^{r-1} \frac{1}{l!} \Etni{\br{R^{n, j} \cdot \partial_{\theta}}^l F_{t_{n, i}}\br{t_{n, j}, \Theta_{t_{n, j}}}} + \mO\br{\abs{\Delta t_n}^{r}}
    \end{align}
    for $j = m, m-1, \cdots, i+1$, where the remainder term follows from \eqref{eq_preest_Rni}.
    By the induction hypothesis, namely \eqref{eq_Rnj_indc} holds for $j = m, m-1, \cdots, i+1$, we substitute \eqref{eq_Rnj_indc} into \eqref{eq_EtniDelFninj} to obtain  
    \begin{equation}\label{eq_expand_dtDF}
        \Delta t_n \sum_{j = i+1}^m a_{ij} \Etni{\Delta F_{t_{n, i}}^{n, j}} = I_{i} + \mO\br{\br{\Delta t}^{r+1}},
    \end{equation}    
    where $I_{i}$ is defined by
    \begin{equation}\label{eq_defIi}
        I_{i} = \Delta t_n \sum_{j = i+1}^m \sum_{l=1}^{r-1} \frac{a_{ij}}{l!} E_{i, j, l},
    \end{equation}
    \begin{equation*}
        E_{i, j, l} := \mathbb{E}_{t_{n, i}}\Big[\vbr{\Big}{\sum_{\Upsilon \in \widetilde{\T}_r} \frac{A_j^m\br{\Upsilon}}{\gamma\br{\Upsilon}} \br{\Delta t_n}^{\abs{\Upsilon}} F_{t_{n, j}}^{\Upsilon} \cdot \partial_{\theta}}^l F_{t_{n, i}}\br{t_{n, j}, \Theta_{t_{n, j}}} \Big]. 
    \end{equation*}
    
    For any function $\phi(\cdot)$ defined on $\widetilde{\T}_r$, we have the expansion
    \begin{equation*}
        \vbr{\Big}{\sum_{\Upsilon \in \widetilde{\T}_r} \phi(\Upsilon)}^{l} = \sum_{\Upsilon_1 \in \widetilde{\T}_r} \sum_{\Upsilon_2 \in \widetilde{\T}_r} \cdots \sum_{\Upsilon_l \in \widetilde{\T}_r} \prod_{q=1}^l \phi(\Upsilon_q).
    \end{equation*}
    Applying the above expansion to \eqref{eq_defIi}, we obtain 
    \begin{equation}\label{eq0_Etni_med}
        I_{i} = \sum_{j = i+1}^m \sum_{l=1}^{r-1} \sum_{\Upsilon_1 \in \widetilde{\T}_r} \sum_{\Upsilon_2 \in \widetilde{\T}_r} \cdots \sum_{\Upsilon_l \in \widetilde{\T}_r} G\br{i, j, l, \Upsilon_1, \Upsilon_2, \cdots, \Upsilon_l},
    \end{equation}
    where
    \begin{equation}\label{eq_defGijl}
        \begin{aligned}
            & G\br{i, j, l, \Upsilon_1, \Upsilon_2, \cdots, \Upsilon_l}\\
            :=\; & \Delta t_n \frac{a_{ij}}{l!} \bbr{\prod_{p = 1}^l \frac{A_j^m\br{\Upsilon_p}}{\gamma\br{\Upsilon_p}} \br{\Delta t_n}^{\abs{\Upsilon_p}}} \Etni{\prod_{q = 1}^l \br{F_{t_{n, j}}^{\Upsilon_q} \cdot \partial_{\theta}} F_{t_{n, i}}\br{t_{n, j}, \Theta_{t_{n, j}}}} \\
            =\;  & \frac{a_{ij}}{l!} \bbr{\prod_{p = 1}^l \frac{A_j^m\br{\Upsilon_p}}{\gamma\br{\Upsilon_p}}} \br{\Delta t_n}^{\abs{\rbr{\Upsilon_1 \cdots \Upsilon_l}_0}} F_{t_{n, i}, t_{n, j}}^{\rbr{\Upsilon_1 \cdots \Upsilon_l}_0}
        \end{aligned}
    \end{equation}
    with the second equality following from the first equation in \eqref{eq_recur_alp} and the second equation in \eqref{eq_recur_F}.
    
    Applying the Taylor expansion to the function $t \mapsto F_{t_{n, i}, t}^{\rbr{\Upsilon_1 \cdots \Upsilon_l}_0}$ at $t = t_{n, i}$, we obtain
    \begin{equation}\label{eq_taylor_Ftninj}
        F_{t_{n, i}, t_{n, j}}^{\rbr{\Upsilon_1 \cdots \Upsilon_l}_0} = \sum_{k=0}^{r-1} \frac{\br{c_i - c_j}^k}{k!} \br{\Delta t_n}^k F_{t_{n, i}}^{\rbr{\Upsilon_1 \cdots \Upsilon_l}_k} + \mO\br{\br{\Delta t}^{r}},
    \end{equation}
    where the derivatives are expressed using the notation introduced in \eqref{eq_recur_F}.
    Inserting \eqref{eq_taylor_Ftninj} into \eqref{eq_defGijl}, we obtain
    \begin{align}\label{eq1_Etni_med}
        G\br{i, j, l, \Upsilon_1, \Upsilon_2, \cdots, \Upsilon_l} =\; & a_{ij} \sum_{k=0}^{r-1} g\br{i, j, k, l, \Upsilon_1, \Upsilon_2, \cdots, \Upsilon_l} + \mO\br{\br{\Delta t}^{r+1}}
    \end{align}
    with
    \begin{equation}\label{eq_defgijkl}
        \begin{aligned}
            & g\br{i, j, k, l, \Upsilon_1, \Upsilon_2, \cdots, \Upsilon_l}\\
            :=\; & \frac{\br{c_i - c_j}^k }{k! \; l!} \bbr{\prod_{p = 1}^l \frac{A_j^m\br{\Upsilon_p}}{\gamma\br{\Upsilon_p}}} \br{\Delta t_n}^{\abs{\rbr{\Upsilon_1 \cdots \Upsilon_l}_k}} F_{t_{n, i}}^{\rbr{\Upsilon_1 \cdots \Upsilon_l}_k},
        \end{aligned}
    \end{equation}
    where the exponents of $\Delta t_n$ are obtained by $\abs{\rbr{\Upsilon_1 \cdots \Upsilon_l}_k} = k + \abs{\rbr{\Upsilon_1 \cdots \Upsilon_l}_0}$ following from the first equation in \eqref{eq_recur_alp}.
    
    Substituting \eqref{eq1_Etni_med} into \eqref{eq0_Etni_med}, and then inserting the resulting equation into \eqref{eq_expand_dtDF}, we obtain
    \begin{equation}\label{eq_DtaEtniDF}
        \Delta t_n \sum_{j = i+1}^m a_{ij} \Etni{\Delta F_{t_{n, i}}^{n, j}} = \tilde{I}_i + \mO\br{\br{\Delta t}^{r+1}}
    \end{equation}
    with
    \begin{equation}\label{eq_defHikl}
        \begin{aligned}
            \tilde{I}_i :=\; & \sum_{j = i+1}^m \sum_{k=0}^{r-1} \sum_{l=1}^{r-1} \sum_{\Upsilon_1 \in \widetilde{\T}_r} \sum_{\Upsilon_2 \in \widetilde{\T}_r} \cdots \sum_{\Upsilon_l \in \widetilde{\T}_r}  a_{ij} g\br{i, j, k, l, \Upsilon_1, \Upsilon_2, \cdots, \Upsilon_l} \\
            =\;              & \sum_{k=0}^{r-1} \sum_{l=1}^{r-1} \sum_{\Upsilon_1 \in \widetilde{\T}_r} \sum_{\Upsilon_2 \in \widetilde{\T}_r} \cdots \sum_{\Upsilon_l \in \widetilde{\T}_r} \sum_{j = 1}^m a_{ij} g\br{i, j, k, l, \Upsilon_1, \Upsilon_2, \cdots, \Upsilon_l},
        \end{aligned}
    \end{equation}
    where, in the second line, we use the second condition in \eqref{eq_ajieq0_maxabc} to extend the summation range over $j$.

    \subsubsection*{(iv) Simplification of $\tilde{I}_i$}

    For any $\Upsilon = \rbr{\Upsilon_1\cdots\Upsilon_l}_k \in \widetilde{\T} \backslash \widetilde{\T}^1$, we define
    \begin{equation}\label{eq_deftildeAi}
        B_i^m\br{\Upsilon} := \sum_{j=1}^{m} a_{ij} \frac{\br{c_i - c_j}^k }{k! \; l!} \prod_{p = 1}^l \frac{1}{\gamma\br{\Upsilon_p}} A_j^m\br{\Upsilon_p}. 
    \end{equation}
    Rearranging the terms in \eqref{eq_deftildeAi}, we obtain
    \begin{equation}\label{eq_defBim}
        B_i^m\br{\Upsilon} = \frac{1}{k! \; l! \prod_{p = 1}^l \gamma\br{\Upsilon_p}} \sum_{j=1}^{m} a_{ij} \br{c_i - c_j}^k \prod_{p = 1}^l A_j^m\br{\Upsilon_p} = \frac{A_i^m\br{\Upsilon}}{\gamma\br{\Upsilon}},
    \end{equation}
    where the second equality is obtained by using the notation defined in \eqref{eq_recur_alp} and \eqref{eq_recur_Ai}.

    Multiplying both sides of \eqref{eq_defgijkl} by $a_{ij}$,
    summing over $j = 1, 2, \cdots, m$, and then invoking \eqref{eq_deftildeAi}, we obtain
    \begin{equation*}
        \sum_{j=1}^{m} a_{ij} g\br{i, j, k, l, \Upsilon_1, \Upsilon_2, \cdots, \Upsilon_l} = B_i^m\br{\rbr{\Upsilon_1\cdots\Upsilon_l}_k} \br{\Delta t_n}^{\abs{\rbr{\Upsilon_1 \cdots \Upsilon_l}_k}} F_{t_{n, i}}^{\rbr{\Upsilon_1 \cdots \Upsilon_l}_k}.
    \end{equation*}
    Inserting the above equation into \eqref{eq_defHikl}, we deduce
    \begin{equation}\label{eq_tildeIi}
    \begin{aligned}
        \tilde{I}_i =\; & \sum_{k=0}^{r-1} \sum_{l=1}^{r-1} \sum_{\Upsilon_1 \in \widetilde{\T}_r} \sum_{\Upsilon_2 \in \widetilde{\T}_r} \cdots \sum_{\Upsilon_l \in \widetilde{\T}_r} B_i^m\br{\rbr{\Upsilon_1\cdots\Upsilon_l}_k} \br{\Delta t_n}^{\abs{\rbr{\Upsilon_1 \cdots \Upsilon_l}_k}} F_{t_{n, i}}^{\rbr{\Upsilon_1 \cdots \Upsilon_l}_k} \\
        =\; & \sum_{\Upsilon \in \widetilde{\T}\br{r}} B_i^m\br{\Upsilon} \br{\Delta t_n}^{\abs{\Upsilon}} F_{t_{n, i}}^{\Upsilon}\\
        =\;& \sum_{\Upsilon \in \widetilde{\T}\br{r}} \frac{A_i^m\br{\Upsilon}}{\gamma\br{\Upsilon}} \br{\Delta t_n}^{\abs{\Upsilon}} F_{t_{n, i}}^{\Upsilon},
    \end{aligned}
    \end{equation}
    where $\widetilde{\T}\br{r}$ is defined by
    \begin{equation}\label{eq_defTbrr}
        \widetilde{\T}\br{r} := \bbr{\rbr{\Upsilon_1\cdots\Upsilon_l}_k \in \widetilde{\T}: \bbr{\Upsilon_i}_{i=1}^l \subset \widetilde{\T}_r, \; 0 \leq k \leq r-1, \; 1 \leq l \leq r-1},
    \end{equation}
    and the last equality follows from \eqref{eq_defBim}.
    
    \subsubsection*{(v) Simplification of $\widetilde{\T}_r \bigcap \widetilde{\T}\br{r}$}
    
    Comparing \eqref{eq_defTr1} with \eqref{eq_defTbrr}, we notice that \eqref{eq_defTbrr} enforces $l \ge 1$ for the elements of $\widetilde{\T}\br{r}$, which is in conflict with $l = 0$ for elements of $\widetilde{\T}_r^1$.
    Thus $\widetilde{\T}_r^1 \cap \widetilde{\T}\br{r} = \emptyset$, which further implies $\widetilde{\T}_r \backslash \widetilde{\T}_r^1 \supset \widetilde{\T}_r \bigcap \widetilde{\T}\br{r}$. 
        
    We now turn to prove $\widetilde{\T}_r \backslash \widetilde{\T}_r^1 \subset \widetilde{\T}_r \bigcap \widetilde{\T}\br{r}$. 
    By the definition of $\abs{\,\cdot\,}$ given in \Cref{def_order}, we have 
    \begin{equation}\label{eq_abs_lb}
        \abs{\Upsilon} \geq 1, \quad \forall \Upsilon \in \widetilde{\T}, 
    \end{equation}
    which can be proved by induction on $\widetilde{\T}^1, \widetilde{\T}^2, \cdots$. 
    From the definitions in \eqref{eq_deftiTrm} and \eqref{eq_defTr1}, any $\Upsilon \in \widetilde{\T}_r \backslash \widetilde{\T}_r^1$ can be written as $\Upsilon = \rbr{\Upsilon_1\cdots\Upsilon_l}_k$ with $\Upsilon_1, \cdots, \Upsilon_l \in \widetilde{\T}$ and
    \begin{equation}\label{eq_condiTrTr1}
        l \geq 1, \quad 1 + k + \sum_{i=1}^l \abs{\Upsilon_i} \leq r, \quad k \geq 0, 
    \end{equation}
    where $k \geq 0$ follows from \eqref{eq_defT}. 
    Using \eqref{eq_abs_lb}, the first two inequalities in \eqref{eq_condiTrTr1} imply
    \begin{equation}\label{eq2_condiTrTr1}
        l \leq r - 1, \quad k \leq r - 1, \quad \abs{\Upsilon_i} \leq r, \quad i = 1, 2, \cdots, l. 
    \end{equation}
    The last inequality in \eqref{eq_condiTrTr1} implies
    \begin{equation}\label{eq3_condiTrTr1}
        \Upsilon_i \in \widetilde{\T}_r, \quad i = 1, 2, \cdots, l
    \end{equation}
    with $\widetilde{\T}_r$ defined in \eqref{eq_deftiTrm}.
    Combining \eqref{eq_condiTrTr1}, \eqref{eq2_condiTrTr1}, and \eqref{eq3_condiTrTr1}, we conclude that $\Upsilon$ satisfies all the conditions in \eqref{eq_defTbrr}. 
    Hence $\widetilde{\T}_r \backslash \widetilde{\T}_r^1 \subset \widetilde{\T}\br{r}$, which further yields $\widetilde{\T}_r \backslash \widetilde{\T}_r^1 \subset \widetilde{\T}_r \bigcap \widetilde{\T}\br{r}$. 
    Consequently, we obtain
    \begin{equation}\label{eq_tildeTir}
        \widetilde{\T}_r \bigcap \widetilde{\T}\br{r} = \widetilde{\T}_r \backslash \widetilde{\T}_r^1.
    \end{equation}
    
    \subsubsection*{(vi) Closure of induction}
    
    By \eqref{eq_tildeTir}, $\widetilde{\T}\br{r}$ can be decomposed into the disjoint union of $\widetilde{\T}_r \backslash \widetilde{\T}_r^1$ and $\widetilde{\T}\br{r} \backslash \widetilde{\T}_r$.
    Hence \eqref{eq_tildeIi} implies
    \begin{equation}\label{eq_tilIi}
        \tilde{I}_i = \sum_{\Upsilon \in \widetilde{\T}_r \backslash \widetilde{\T}_r^1} \frac{A_i^m\br{\Upsilon}}{\gamma\br{\Upsilon}} \br{\Delta t_n}^{\abs{\Upsilon}} F_{t_{n, i}}^{\Upsilon} + \mO\br{\br{\Delta t}^{r+1}}, 
    \end{equation}
    where the remainder term collects all summands with $\Upsilon \in \widetilde{\T}\br{r} \backslash \widetilde{\T}_r$, and the order $\br{\Delta t}^{r+1}$ follows from the observation that $\abs{\Upsilon} \geq r+1$ for any $\Upsilon \in \widetilde{\T}\br{r} \backslash \widetilde{\T}_r$, following from the definition of $\widetilde{\T}_r$ in \eqref{eq_deftiTrm}.
    Inserting \eqref{eq_tilIi} into \eqref{eq_DtaEtniDF}, we obtain
    \begin{equation*}
    \begin{aligned}
        \Delta t_n \sum_{j = i+1}^m a_{ij} \Etni{\Delta F_{t_{n, i}}^{n, j}} =\;& \sum_{\Upsilon \in \widetilde{\T}_r \backslash \widetilde{\T}_r^1} \frac{A_i^m\br{\Upsilon}}{\gamma\br{\Upsilon}} \br{\Delta t_n}^{\abs{\Upsilon}} F_{t_{n, i}}^{\Upsilon} + \mO\br{\br{\Delta t}^{r+1}}.
    \end{aligned}
    \end{equation*}
    Inserting the above equation and \eqref{eq_expandEni} into \eqref{eq_RniEni}, we obtain
    \begin{equation}\label{eq_Rni_indc}
        R^{n, i} = \sum_{\Upsilon \in \widetilde{\T}_r} \frac{A_i^m\br{\Upsilon}}{\gamma\br{\Upsilon}} \br{\Delta t_n}^{\abs{\Upsilon}} F_{t_{n, i}}^{\Upsilon} + \mO\br{\br{\Delta t}^{r+1}}, 
    \end{equation}
    which is exactly \eqref{eq_Rnj_indc} with $j = i$.
    This completes the induction step and establishes \eqref{eq_Rnj_indc} for all $j = 0,1,\cdots,m$.
\end{proof}

\subsection{General order conditions}\label{sec_ord_condi}

Setting $j=0$ in \eqref{eq2_Rnj_indc} yields the expansion of the local truncation error $R^n = R^{n,0}$.
Further, if the coefficients of \Cref{sch_rk} satisfy
\begin{equation}\label{eq_ord_cond1}
    A_0^m\br{\Upsilon} = 0, \quad \forall \Upsilon \in \T_r,
\end{equation}
then $R^n= \mO\br{(\Delta t)^{r+1}}$.
Thus, \eqref{eq_ord_cond1} serves as the $r$th order conditions for \Cref{sch_rk}. 
However, these conditions are cumbersome to check in practice.
Since the number $\# \T_r$ of ULN-trees in $\T_r$ is $\#\T_r = 1, 3, 8, 21, 57$ for $r = 1, 2, 3, 4, 5$, respectively.
The collection of conditions quickly becomes unwieldy as $r$ increases. 

To address this issue, we introduce the following simplified $r$th order condition $C\br{r}$:
\begin{equation}\label{eq_defCr}
C\br{r}: \left\{\begin{aligned}
    & A_i^m\br{[\;]_0} = 0, \sptext{for} i = 0, 1, \cdots, m,    \\
    & A_0^m\br{\Upsilon} = 0, \sptext{for} \Upsilon \in \T_{r-},
\end{aligned}\right.
\end{equation}
where $\T_{r-} := \bbr{\Upsilon \in \T_r: \Upsilon \neq [\;]_0 \sptext{and} [\;]_0 \sptext{is not a branch of} \Upsilon}$ and $A_i^m\br{\cdot}$ are elementary coefficients given in \Cref{def_weight}.
In the following, we list the elements of $\T_{r-}$ for $r$ upto $5$:
\begin{align}
    & \T_{1-} = \emptyset, \quad \T_{2-} = \big\{[\;]_1 \big\}, \quad \T_{3-} = \T_{2-} \bigcup \big\{[\;]_2, \; [[\;]_{1}]_{0}\big\}, \label{eq1_enum_T}\\
    & \T_{4-} = \T_{3-} \bigcup \Big\{[\;]_3, \; [[\;]_2]_0, \; [[\;]_1]_1, \; [[[\;]_1]_0]_0 \Big\},  \label{eq2_enum_T}\\
    & \begin{aligned}
        \T_{5-} = \T_{4-} \bigcup \;&\Big\{[\;]_4, \; [[\;]_3]_0, \; [[\;]_2]_1, \; [[\;]_1]_2, \; [[[\;]_2]_0]_0, \; [[\;]_1 [\;]_1]_0, \; [[[\;]_1]_1]_0, \\
        & \;\; [[[\;]_1]_0]_1, \; [[[ [\;]_1]_0]_0]_0\Big\}.
    \end{aligned}  \label{eq3_enum_T}
\end{align}
\Cref{prop_sim_ordcondi} shows that $C\br{r}$ is sufficient for \eqref{eq_ord_cond1}.
The total number of constraints in $C\br{r}$ is $\# \T_{r-} + m + 1$, where $m$ is the number of stages of \Cref{sch_rk}, and $\# \T_{r-} = 0, 1, 3, 7, 16$ for $r = 1, 2, 3, 4, 5$, respectively.
Thus, for higher-order \Cref{sch_rk}, the condition $C\br{r}$ is typically much simpler than \eqref{eq_ord_cond1}.

The following theorem, a direct corollary of \Cref{thm_taylor} and \Cref{prop_sim_ordcondi}, constitutes the main result of this paper.
\begin{theorem}[Order conditions]\label{thm_consistent}
    Under \Cref{assu_regu,assu_coeff},
    if the coefficients of \Cref{sch_rk} satisfy the conditions in $C\br{r}$ in \eqref{eq_defCr},
    then \Cref{sch_rk} is consistent of order $r$ in the sense of
    \begin{equation}\label{eq_maxRn}
        \max_{0 \leq n \leq N-1} \bigl| R^n \bigr| = \mO\vbr{\big}{\br{\Delta t}^{r+1}},
    \end{equation}
    where $R^n$ is the local truncation error defined in \eqref{eq_Rni}.
\end{theorem}

The order conditions in \eqref{eq_defCr} can be expressed explicitly as an algebraic system in the coefficients $a_{ij}, b_i, c_i$ of \Cref{sch_rk} by substituting the expressions for $A_i^m(\Upsilon)$ given in \Cref{def_weight}; see the Supplementary Materials for details.
However, this system does not match the list in \Cref{tab_ordcond} term by term, because the table has been reorganized to facilitate comparison with classical ODE results.
\Cref{prop_Cr_table} establishes the equivalence between these two formulations.
The proof is elementary but tedious, and is therefore deferred to the Supplementary Materials.

\begin{proposition}\label{prop_Cr_table}
    \Cref{assu_coeff} implies the conditions (1) and (2) in \Cref{tab_ordcond}.
    And the following equivalences hold:
    \begin{align*}
        & C(1) \iff \text{(3)--(4)},  && C(2) \iff \text{(3)--(5)},  && C(3) \iff \text{(3)--(7)},\\
        & C(4) \iff \text{(3)--(11)},  && C(5) \iff \text{(3)--(20)}. &&
    \end{align*}
    Here, $C(1), C(2), \cdots, C(5)$ are defined in \eqref{eq_defCr}, 
    and ``(i)--(j)'' denotes the set of conditions numbered (i) through (j) in \Cref{tab_ordcond}.
\end{proposition}

Given the $(r+1)$th-order convergence rate of the local truncation error in \eqref{eq_maxRn}, the global error is expected to be $\mO(\br{\Delta t}^{r})$, provided that \Cref{sch_rk} is stable. 
The stability analysis and global error estimates are deferred to the Supplementary Materials, as this work primarily focuses on consistency.

\section{Numerical tests}\label{sec_numtest}

In this section, we present numerical experiments to validate the consistency results established in \Cref{sec_theoana}.
We consider specific instances of \Cref{sch_rk} of orders up to $5$: the Euler scheme, RK($2$; $c_1$) and RK($3$; $c_1$, $c_2$) are given in \eqref{eq_sch_euler}, \eqref{eq_sch_RK2c1}, and \eqref{eq_sch_RK3c1c2}, respectively. 
The schemes RK($4$; $5$), RK($4$; $6$), RK($5$; $7$), and RK($5$; $8$) are constructed by numerically enforcing the order conditions in \Cref{tab_ordcond}; see \cref{sec_rk_sch} for details. 
The corresponding coefficients are listed in the Supplementary Materials.

All benchmark problems are instances of \eqref{bsde}. 
The conditional expectations $\Etn{\,\cdot\,}$ and $\Etni{\,\cdot\,}$ in \Cref{sch_rk} are approximated using the Sinc quadrature rule proposed in \cite[subsection 3.1]{Wang2022Sinc}. 
To reduce the cost of approximating the conditional expectations, we take $X$ and $W$ to be one-dimensional,
while allowing $(Y, Z)$ to be vector-valued to assess the order conditions for systems of BSDEs.
The time interval $[0,T]$ is uniformly partitioned into $N$ subintervals with time step $\Delta t := T/N$; the values of $N$ used are reported in the numerical results.
The spatial grid is uniform, namely $\R_h := \{i h : i = 0, \pm 1, \pm 2, \cdots\}$ with mesh size $h = 5 \times 10^{-3}$. 
All numerical experiments are conducted using Python 3.12.9 on a compute node equipped with an Intel Xeon 8358 processor and 256 GB of DDR4 RAM (3200 MHz).

We report the $L^{\infty}$ errors $\|Y^0 - Y_0\|_{\infty}$ and $\|Z^0 - Z_0\|_{\infty}$. The former is defined as
\begin{equation*}
    \|Y^0 - Y_0\|_{\infty} := \max_{x \in \R_{h} \cap [a, b]} \big|Y^0(x) - Y_0(x)\big|,
\end{equation*}
where $Y_0(x)$ denotes the exact solution value of $Y_t$ at $t=0$ with $X_0 = x$, i.e., $Y_0(x) := u(0, x)$ by \eqref{eq_Ytut}; $Y^0(x)$ denotes the numerical approximation to $Y_0(x)$ produced by \Cref{sch_rk}; $[a,b] \subset \R$ is the spatial domain of interest. 
The quantity $\|Z^0 - Z_0\|_{\infty}$ is defined analogously. 
Temporal convergence rates (CRs) are obtained by linear least-squares fits of $\log\|Y^0 - Y_0\|_{\infty}$ and $\log\|Z^0 - Z_0\|_{\infty}$ against $\log(T/N)$, with $N = 30,40,54,70,90$.
Running times (RTs) are measured with Python's standard function time.perf\_counter().

\subsection{Example 1}

\begin{table}[t]
\caption{Numerical results for Example 1}
\label{tab_explinsin}
\resizebox{\textwidth}{!}{%
\begin{tabular}{@{}lcccc|lcccc@{}}
\toprule
Scheme & $N$ & $\norm{Y^0 - Y_0}_{\infty}$ & $\norm{Z^0 - Z_0}_{\infty}$ & RT    (s) & Scheme & $N$ & $\norm{Y^0 - Y_0}_{\infty}$ & $\norm{Z^0 - Z_0}_{\infty}$ & RT    (s) \\ \midrule
Euler & 30 & 9.02E-02 & 5.96E-03 & 0.25 & RK(3; 3/4, 1/2) & 30 & 8.99E-06 & 1.50E-06 & 0.39 \\
 & 40 & 6.75E-02 & 4.46E-03 & 0.30 &  & 40 & 3.77E-06 & 6.25E-07 & 0.47 \\
 & 54 & 4.99E-02 & 3.29E-03 & 0.24 &  & 54 & 1.52E-06 & 2.52E-07 & 0.59 \\
 & 70 & 3.84E-02 & 2.54E-03 & 0.24 &  & 70 & 6.97E-07 & 1.15E-07 & 0.73 \\
 & 90 & 2.99E-02 & 1.97E-03 & 0.29 &  & 90 & 3.27E-07 & 5.39E-08 & 0.91 \\
 &  & \textbf{CR:} 1.007 & \textbf{CR:} 1.006 &  &  &  & \textbf{CR:} 3.017 & \textbf{CR:} 3.025 &  \\ \\
RK(2; 1/2) & 30 & 1.24E-03 & 1.23E-04 & 0.21 & RK(4; 5) & 30 & 5.42E-08 & 4.47E-09 & 0.83 \\
 & 40 & 6.94E-04 & 6.82E-05 & 0.25 &  & 40 & 1.72E-08 & 1.41E-09 & 1.03 \\
 & 54 & 3.78E-04 & 3.70E-05 & 0.32 &  & 54 & 5.21E-09 & 4.21E-10 & 1.29 \\
 & 70 & 2.24E-04 & 2.19E-05 & 0.39 &  & 70 & 1.85E-09 & 1.49E-10 & 1.59 \\
 & 90 & 1.35E-04 & 1.32E-05 & 0.49 &  & 90 & 6.77E-10 & 5.42E-11 & 1.95 \\
 &  & \textbf{CR:} 2.021 & \textbf{CR:} 2.032 &  &  &  & \textbf{CR:} 3.988 & \textbf{CR:} 4.016 &  \\ \\
RK(2; 2/3) & 30 & 8.43E-04 & 1.52E-04 & 0.26 & RK(4; 6) & 30 & 3.03E-08 & 2.36E-09 & 1.14 \\
 & 40 & 4.71E-04 & 8.47E-05 & 0.31 &  & 40 & 9.64E-09 & 7.44E-10 & 1.44 \\
 & 54 & 2.58E-04 & 4.62E-05 & 0.40 &  & 54 & 2.92E-09 & 2.23E-10 & 1.82 \\
 & 70 & 1.53E-04 & 2.73E-05 & 0.49 &  & 70 & 1.04E-09 & 7.88E-11 & 2.21 \\
 & 90 & 9.22E-05 & 1.65E-05 & 0.61 &  & 90 & 3.80E-10 & 2.88E-11 & 2.71 \\
 &  & \textbf{CR:} 2.014 & \textbf{CR:} 2.022 &  &  &  & \textbf{CR:} 3.984 & \textbf{CR:} 4.011 &  \\ \\
RK(3; 2/3, 1/3) & 30 & 1.90E-05 & 1.90E-06 & 0.34 & RK(5; 7) & 30 & 7.38E-10 & 5.06E-11 & 1.48 \\
 & 40 & 7.94E-06 & 7.91E-07 & 0.42 &  & 40 & 1.74E-10 & 1.18E-11 & 1.84 \\
 & 54 & 3.21E-06 & 3.18E-07 & 0.53 &  & 54 & 3.86E-11 & 2.61E-12 & 2.32 \\
 & 70 & 1.47E-06 & 1.45E-07 & 0.66 &  & 70 & 1.05E-11 & 7.13E-13 & 2.84 \\
 & 90 & 6.88E-07 & 6.77E-08 & 0.79 &  & 90 & 3.00E-12 & 2.08E-13 & 3.50 \\
 &  & \textbf{CR:} 3.020 & \textbf{CR:} 3.035 &  &  &  & \textbf{CR:} 5.011 & \textbf{CR:} 5.006 &  \\ \\
RK(3; 1/2, 1/4) & 30 & 3.35E-05 & 2.13E-06 & 0.39 & RK(5; 8) & 30 & 1.22E-10 & 5.53E-12 & 1.75 \\
 & 40 & 1.40E-05 & 8.83E-07 & 0.49 &  & 40 & 2.96E-11 & 1.22E-12 & 2.14 \\
 & 54 & 5.65E-06 & 3.54E-07 & 0.62 &  & 54 & 6.71E-12 & 2.58E-13 & 2.70 \\
 & 70 & 2.58E-06 & 1.61E-07 & 0.78 &  & 70 & 1.87E-12 & 6.58E-14 & 3.34 \\
 & 90 & 1.21E-06 & 7.54E-08 & 0.96 &  & 90 & 5.48E-13 & 1.69E-14 & 4.09 \\
 &  & \textbf{CR:} 3.023 & \textbf{CR:} 3.041 &  &  &  & \textbf{CR:} 4.927 & \textbf{CR:} 5.258 &  \\ \bottomrule
\end{tabular}%
}
\end{table}

We consider an example from \cite{zhao2010stable}, defined by \eqref{bsde} with
\begin{equation*}
\begin{aligned}
    X_t & = X_0 + W_t, \quad \varphi(x) = \exp \left(T^{2}\right) \ln\left(\sin x + 3\right),                                        \\
    f(t, x, y, z) & = \frac{1}{2}\rbr{\exp\left(t^{2}\right)-4 t y - 3 \exp \left(t^{2}-  y \exp \left(-t^{2}\right)\right) + z^{2} \exp \left(-t^{2}\right)}
\end{aligned}
\end{equation*}
for $t \in [0, T]$, $x \in \R$, $y \in \R$, and $z \in \R$.
The analytic solution is
\begin{equation*}
    Y_t = \exp\left(t^{2}\right) \ln\left(\sin X_t + 3\right),  \quad  Z_t = \exp \left(t^{2}\right) \frac{\cos X_t}{\sin X_t + 3}.
\end{equation*}
We set $T = 1$ and solve $\br{Y_0, Z_0}$ for $X_0 \in [0, \pi]$.
Table~\ref{tab_explinsin} summarizes the numerical results.
As stated at the beginning of this section, RK($r$; $\cdots$) denotes an instance of \Cref{sch_rk} whose coefficients satisfy the $r$th order conditions in \Cref{tab_ordcond}.
The results indicate that all schemes achieve the expected convergence orders, thereby supporting the consistency analysis in \Cref{sec_theoana}.

\subsection{Example 2}

\begin{table}[t]
\caption{Numerical results for Example 2}
\label{tab_expsigmoid}
\resizebox{\textwidth}{!}{%
\begin{tabular}{@{}lcccc|lcccc@{}}
\toprule
Scheme & $N$ & $\norm{Y^0 - Y_0}_{\infty}$ & $\norm{Z^0 - Z_0}_{\infty}$ & RT    (s) & Scheme & $N$ & $\norm{Y^0 - Y_0}_{\infty}$ & $\norm{Z^0 - Z_0}_{\infty}$ & RT    (s) \\ \midrule
Euler & 30 & 1.55E-03 & 5.22E-04 & 0.07 & RK(3; 3/4, 1/2) & 30 & 1.93E-07 & 9.25E-08 & 0.23 \\
 & 40 & 1.16E-03 & 3.90E-04 & 0.09 &  & 40 & 8.14E-08 & 3.90E-08 & 0.30 \\
 & 54 & 8.58E-04 & 2.88E-04 & 0.12 &  & 54 & 3.31E-08 & 1.59E-08 & 0.37 \\
 & 70 & 6.61E-04 & 2.22E-04 & 0.15 &  & 70 & 1.52E-08 & 7.28E-09 & 0.48 \\
 & 90 & 5.13E-04 & 1.72E-04 & 0.19 &  & 90 & 7.15E-09 & 3.43E-09 & 0.65 \\
 &  & \textbf{CR:} 1.009 & \textbf{CR:} 1.009 &  &  &  & \textbf{CR:} 3.000 & \textbf{CR:} 2.999 &  \\ \\
RK(2; 1/2) & 30 & 2.27E-05 & 5.77E-06 & 0.12 & RK(4; 5) & 30 & 5.19E-10 & 2.11E-10 & 0.51 \\
 & 40 & 1.28E-05 & 3.25E-06 & 0.15 &  & 40 & 1.64E-10 & 6.69E-11 & 0.67 \\
 & 54 & 7.00E-06 & 1.78E-06 & 0.21 &  & 54 & 4.94E-11 & 2.01E-11 & 0.90 \\
 & 70 & 4.17E-06 & 1.06E-06 & 0.28 &  & 70 & 1.75E-11 & 7.13E-12 & 1.18 \\
 & 90 & 2.52E-06 & 6.42E-07 & 0.36 &  & 90 & 6.41E-12 & 2.61E-12 & 1.55 \\
 &  & \textbf{CR:} 2.000 & \textbf{CR:} 1.999 &  &  &  & \textbf{CR:} 4.000 & \textbf{CR:} 4.000 &  \\ \\
RK(2; 2/3) & 30 & 2.27E-05 & 5.78E-06 & 0.15 & RK(4; 6) & 30 & 6.65E-10 & 2.72E-10 & 0.67 \\
 & 40 & 1.28E-05 & 3.25E-06 & 0.19 &  & 40 & 2.10E-10 & 8.59E-11 & 0.89 \\
 & 54 & 7.00E-06 & 1.79E-06 & 0.25 &  & 54 & 6.34E-11 & 2.59E-11 & 1.20 \\
 & 70 & 4.17E-06 & 1.06E-06 & 0.32 &  & 70 & 2.24E-11 & 9.16E-12 & 1.58 \\
 & 90 & 2.52E-06 & 6.43E-07 & 0.42 &  & 90 & 8.21E-12 & 3.35E-12 & 2.06 \\
 &  & \textbf{CR:} 2.000 & \textbf{CR:} 1.999 &  &  &  & \textbf{CR:} 4.000 & \textbf{CR:} 4.000 &  \\ \\
RK(3; 2/3, 1/3) & 30 & 1.92E-07 & 9.24E-08 & 0.21 & RK(5; 7) & 30 & 1.43E-12 & 9.06E-13 & 0.85 \\
 & 40 & 8.12E-08 & 3.90E-08 & 0.26 &  & 40 & 3.40E-13 & 2.15E-13 & 1.12 \\
 & 54 & 3.30E-08 & 1.59E-08 & 0.35 &  & 54 & 7.99E-14 & 4.82E-14 & 1.52 \\
 & 70 & 1.52E-08 & 7.28E-09 & 0.46 &  & 70 & 2.10E-14 & 1.34E-14 & 1.99 \\
 & 90 & 7.13E-09 & 3.43E-09 & 0.60 &  & 90 & 7.49E-15 & 3.48E-15 & 2.63 \\
 &  & \textbf{CR:} 3.000 & \textbf{CR:} 2.999 &  &  &  & \textbf{CR:} 4.823 & \textbf{CR:} 5.040 &  \\ \\
RK(3; 1/2, 1/4) & 30 & 1.92E-07 & 9.24E-08 & 0.25 & RK(5; 8) & 30 & 5.90E-12 & 3.71E-12 & 1.11 \\
 & 40 & 8.10E-08 & 3.90E-08 & 0.32 &  & 40 & 1.40E-12 & 8.81E-13 & 1.43 \\
 & 54 & 3.29E-08 & 1.59E-08 & 0.44 &  & 54 & 3.10E-13 & 1.97E-13 & 1.94 \\
 & 70 & 1.51E-08 & 7.28E-09 & 0.57 &  & 70 & 8.57E-14 & 5.39E-14 & 2.49 \\
 & 90 & 7.11E-09 & 3.43E-09 & 0.74 &  & 90 & 2.58E-14 & 1.50E-14 & 3.30 \\
 &  & \textbf{CR:} 3.000 & \textbf{CR:} 2.999 &  &  &  & \textbf{CR:} 4.956 & \textbf{CR:} 5.013 &  \\ \bottomrule
\end{tabular}%
}
\end{table}

This example has been considered by \cite{Chassagneux2014Linear,weinan2017deep,Warin2018Nesting}.
The BSDE takes the form \eqref{bsde} with
\begin{equation*}
    X_t = X_0 + \sigma W_t, \quad f(t, x, y, z) = \sigma \left(y - \dfrac{2+\sigma^2}{2\sigma^2}\right) z, \quad \varphi(x) = \dfrac{\exp(T + x)}{1+\exp(T + x)}
\end{equation*}
for $\sigma \in (0, \infty)$, $t \in [0, T]$, $x \in \R$, $y \in \R$ and $z \in \R$.
The analytic solution is
\begin{equation}
    Y_t = \dfrac{\exp(t + X_t)}{1 + \exp(t + X_t)}, \quad Z_t = \sigma Y_t (1 - Y_t).
\end{equation}
We set $T = 1$ and $\sigma = 0.25$ and compute $\br{Y_0, Z_0}$ for $X_0 \in [-1, 1]$. 
The numerical results are summarized in \Cref{tab_expsigmoid}. 
Similar to the previous example, the observed convergence rates here are in excellent agreement with the theoretical orders.

\subsection{Example 3}\label{sec_exam3}

\begin{table}[t]
\caption{Numerical results for Example 3}
\label{tab_exp_cdr}
\resizebox{\textwidth}{!}{%
\begin{tabular}{@{}lcccc|lcccc@{}}
\toprule
Scheme & $N$ & $\norm{Y^0 - Y_0}_{\infty}$ & $\norm{Z^0 - Z_0}_{\infty}$ & RT    (s) & Scheme & $N$ & $\norm{Y^0 - Y_0}_{\infty}$ & $\norm{Z^0 - Z_0}_{\infty}$ & RT    (s) \\ \midrule
Euler & 30 & 2.37E-03 & 4.54E-03 & 0.82 & RK(3; 3/4, 1/2) & 30 & 3.91E-07 & 6.08E-07 & 1.16 \\
 & 40 & 1.79E-03 & 3.42E-03 & 0.61 &  & 40 & 1.67E-07 & 2.60E-07 & 1.43 \\
 & 54 & 1.33E-03 & 2.55E-03 & 0.60 &  & 54 & 6.83E-08 & 1.07E-07 & 1.78 \\
 & 70 & 1.03E-03 & 1.97E-03 & 0.72 &  & 70 & 3.15E-08 & 4.93E-08 & 2.18 \\
 & 90 & 8.03E-04 & 1.54E-03 & 0.83 &  & 90 & 1.49E-08 & 2.33E-08 & 2.65 \\
 &  & \textbf{CR:} 0.986 & \textbf{CR:} 0.986 &  &  &  & \textbf{CR:} 2.974 & \textbf{CR:} 2.968 &  \\ \\
RK(2; 1/2) & 30 & 3.72E-05 & 7.18E-05 & 0.61 & RK(4; 5) & 30 & 3.47E-09 & 8.19E-09 & 2.74 \\
 & 40 & 2.11E-05 & 4.06E-05 & 0.73 &  & 40 & 1.12E-09 & 2.60E-09 & 3.36 \\
 & 54 & 1.17E-05 & 2.24E-05 & 0.94 &  & 54 & 3.41E-10 & 7.87E-10 & 4.27 \\
 & 70 & 7.00E-06 & 1.34E-05 & 1.13 &  & 70 & 1.22E-10 & 2.79E-10 & 5.13 \\
 & 90 & 4.25E-06 & 8.11E-06 & 1.38 &  & 90 & 4.49E-11 & 1.02E-10 & 6.31 \\
 &  & \textbf{CR:} 1.974 & \textbf{CR:} 1.984 &  &  &  & \textbf{CR:} 3.957 & \textbf{CR:} 3.988 &  \\ \\
RK(2; 2/3) & 30 & 3.56E-05 & 5.41E-05 & 0.76 & RK(4; 6) & 30 & 1.84E-09 & 5.37E-09 & 3.68 \\
 & 40 & 2.03E-05 & 3.07E-05 & 0.95 &  & 40 & 5.93E-10 & 1.71E-09 & 4.51 \\
 & 54 & 1.12E-05 & 1.69E-05 & 1.20 &  & 54 & 1.81E-10 & 5.16E-10 & 5.60 \\
 & 70 & 6.70E-06 & 1.01E-05 & 1.47 &  & 70 & 6.46E-11 & 1.83E-10 & 6.98 \\
 & 90 & 4.07E-06 & 6.14E-06 & 1.82 &  & 90 & 2.38E-11 & 6.72E-11 & 8.16 \\
 &  & \textbf{CR:} 1.976 & \textbf{CR:} 1.981 &  &  &  & \textbf{CR:} 3.958 & \textbf{CR:} 3.987 &  \\ \\
RK(3; 2/3,   1/3) & 30 & 5.14E-07 & 1.08E-06 & 1.00 & RK(5; 7) & 30 & 4.71E-11 & 1.88E-10 & 4.64 \\
 & 40 & 2.19E-07 & 4.59E-07 & 1.24 &  & 40 & 1.13E-11 & 4.49E-11 & 5.71 \\
 & 54 & 8.99E-08 & 1.87E-07 & 1.55 &  & 54 & 2.54E-12 & 1.01E-11 & 7.18 \\
 & 70 & 4.15E-08 & 8.61E-08 & 1.91 &  & 70 & 6.96E-13 & 2.75E-12 & 8.71 \\
 & 90 & 1.96E-08 & 4.06E-08 & 2.32 &  & 90 & 1.99E-13 & 7.84E-13 & 10.62 \\
 &  & \textbf{CR:} 2.972 & \textbf{CR:} 2.988 &  &  &  & \textbf{CR:} 4.976 & \textbf{CR:} 4.987 &  \\ \\
RK(3; 1/2,   1/4) & 30 & 6.79E-07 & 1.81E-06 & 1.27 & RK(5; 8) & 30 & 3.65E-11 & 1.46E-10 & 5.69 \\
 & 40 & 2.89E-07 & 7.69E-07 & 1.55 &  & 40 & 8.79E-12 & 3.52E-11 & 6.96 \\
 & 54 & 1.18E-07 & 3.14E-07 & 1.85 &  & 54 & 1.98E-12 & 7.92E-12 & 8.72 \\
 & 70 & 5.47E-08 & 1.45E-07 & 2.30 &  & 70 & 5.43E-13 & 2.18E-12 & 10.62 \\
 & 90 & 2.58E-08 & 6.84E-08 & 2.75 &  & 90 & 1.56E-13 & 6.24E-13 & 12.83 \\
 &  & \textbf{CR:} 2.976 & \textbf{CR:} 2.980 &  &  &  & \textbf{CR:} 4.966 & \textbf{CR:} 4.968 &  \\ \bottomrule
\end{tabular}%
}
\end{table}

We consider a system of BSDEs associated with the following three coupled convection-diffusion-reaction equations:
\begin{equation}\label{eq_cdr}
\left\{
\begin{aligned}
    &\frac{\partial u_i}{\partial t} + C_i \frac{\partial u_i}{\partial x} - D_i \frac{\partial^2 u_i}{\partial x^2} = r_i(u_1, u_2, u_3), \quad (t,x) \in (0,T] \times \R, \quad i=1,2,3,\\
    &u_i(0, x) = \varphi_i(x), \quad x \in \R, \quad i=1,2,3. 
\end{aligned}
\right.
\end{equation}
Here $u_i : (0,T] \times \R \to \R$ denotes the concentration of the $i$th chemical species; 
$C_i \in \R$ and $D_i > 0$ are its convection velocity and diffusion coefficient, respectively. 
The nonlinear reaction terms $r_i(u_1, u_2, u_3)$ are given by
\begin{equation}
\begin{aligned}
    r_1(u_1, u_2, u_3) &:= (C_1 - s) u_3 + D_1 k (1 - 2 u_1) u_3,\\
    r_2(u_1, u_2, u_3) &:= (C_2 - s) (1 - 2 u_1) u_3 + D_2 \br{k (1 - 4 u_2) + 2 u_3} u_3,\\
    r_3(u_1, u_2, u_3) &:= -k (C_3 - s) (1 - 2 u_1) u_3 - D_3 k \br{k (1 - 2 u_1)^2 + 2 u_3} u_3. 
\end{aligned} 
\end{equation}
The initial concentration profiles $\varphi_i(x)$ are selected so that the system \eqref{eq_cdr} admits an explicit travelling-wave solution
\begin{equation*}
    u_1(t, x) = U(x - s t), \quad u_2(t, x) = U(x - s t) \br{1 - U(x - s t)}, \quad u_3(t, x) = U^{\prime}(x - s t), 
\end{equation*}
where $U(z) := 1/ \vbr{\big}{1 + \exp(k z)}$, and $s, k > 0$ are parameters controlling the wave speed and steepness.

By the Feynman-Kac formula \eqref{eq_Ytut}, the solution to \eqref{eq_cdr} admits
\begin{equation*}
    u_i(T - t, X_{it}) = Y_{it}, \quad \partial_x u_i(T - t, X_{it}) \sqrt{2 D_i} = Z_{it},
\end{equation*}
where $X_{it} := X_{i0} + \sqrt{2 D_i}\, W_t$, which is valued in $\R$, and the pair $(Y_{it}, Z_{it})$ satisfies the BSDE system
\begin{equation}\label{eq_bsde_cdr}
    Y_{it} = \varphi_i(X_{iT}) + \int_t^T \Big[r_i(Y_{1s}, Y_{2s}, Y_{3s}) - \frac{C_i}{\sqrt{2 D_i}} Z_{is}\Big] \, \di s - \int_t^T Z_{is} \, \di W_s
\end{equation}
for $t \in [0, T]$ and $i = 1, 2, 3$.
In the numerical experiments, we set $T=1$, $C_i=0.2\,i$, $D_i=0.3+0.2\,i$, and $s=k=1$, and apply \Cref{sch_rk} to solve $(Y_{i0}, Z_{i0})$ for $X_{i0} \in [0,1]$ and $i=1,2,3$.
The results reported in \Cref{tab_exp_cdr} show that all schemes achieve the expected orders of convergence, thereby validating our consistency analysis of \Cref{sch_rk} for BSDE systems.

\section{Conclusions and future work}\label{sec_conclu}

In this work, we proposed a new class of explicit Runge--Kutta schemes for BSDEs.
A distinctive feature of these schemes is their concise formulation: the $Y$- and $Z$-components share the same Butcher coefficients, which enables a unified consistency analysis.
Building on this structure, we extended Butcher theory to the BSDE setting and obtained an elegant approach for deriving order conditions for schemes with an arbitrary number of stages and any prescribed target order.
The numerical experiments support the theoretical results.

A limitation of the current work is that the consistency theory applies only to BSDEs in \eqref{bsde} whose forward component has zero drift and a constant diffusion coefficient, namely, $X_t = X_0 + \sigma W_t$ with $\sigma \in \R^{d \times q}$ for $t \in [0, T]$.
This restriction is essential for transforming the integral equation \eqref{eq_intz} into the ODE system \eqref{eq_refode0} via the identity \eqref{eq_ydwt_z}.
For a more general forward SDE, for example,
\begin{equation*}
    X_t = X_0 + \int_0^t b(s, X_s) \, \di s + \int_0^t \sigma(s, X_s) \, \di W_s, \quad t \in [0, T],
\end{equation*}
where $b$ and $\sigma$ depend on $(s, X_s)$,
the term $\Et{Y_s \Delta W_{t, s}^{\top}}$ can no longer be expressed solely in terms of $Z_s$, as in \eqref{eq_ydwt_z}.
Consequently, a unified ODE formulation analogous to \eqref{eq_refode0} is no longer available. Instead, we obtain
\begin{equation*}
    \left\{
    \begin{aligned}
        \frac{\di \Et{Y_s}}{\di s} & = -\Et{f(s, X_s, Y_s, Z_s)}, \\
        \frac{\di \Et{Y_s \Delta W_{t, s}^{\top}}}{\di s} & = -\Et{f(s, X_s, Y_s, Z_s)\Delta W_{t, s}^{\top}} + \Et{Z_s},
    \end{aligned}\right.\quad s \in (t, T],
\end{equation*}
which is an integro-differential system rather than a pure ODE for $s \mapsto (\Et{Y_s}, \Et{Z_s})$.
Extending the Butcher theory to this broader setting remains an interesting and challenging direction for future research.

\bibliographystyle{spmpsci}
\bibliography{references}

\endgroup

\clearpage

\begingroup
\let\maketitle\relax
\renewcommand{\RequirePackage}[2][]{}
\renewcommand{\bibliographystyle}[1]{}
\renewcommand{\bibliography}[1]{}

\begin{center}
    \Large\textbf{Supplementary materials}
\end{center}

\maketitle
\appendix

This supplementary material is organized as follows.
\begin{itemize}
    \item \Cref{sec_ULNgraph} collects the graphs of ULN-trees in $\T_r$ together with their associated functions for orders $r \le 5$.
    \item \Cref{sec_RKnumopt} presents the coefficients of \Cref{sch_rk} for orders $r=4$ (with $m=5,6$ stages) and $r=5$ (with $m=7,8$ stages), obtained via numerical optimization.
    \item \Cref{sec_expcod} lists the explicit expressions of the order conditions $C(r)$ given in \eqref{eq_defCr} for orders $r \leq 5$.
    \item \Cref{sec_proof} provides the proof of \Cref{prop_Cr_table} given in \cref{sec_ord_condi}.
    \item \Cref{sec_errest} presents the stability analysis and global error estimates.
\end{itemize}

\section{ULN-trees with orders up to 5}\label{sec_ULNgraph}

In this section, each page is divided into two columns: the left column depicts a ULN-tree, with each node labeled by a number $k_i$; the right column lists all admissible choices of $k_i$ for which the ULN-tree shown on the left belongs to $\T_5$, along with the corresponding values of the functions $\abs{\cdot}$ and $\gamma(\cdot)$ defined in \Cref{def_order}.
To indicate the set $\T^{q}$ to which an ULN-tree $\Upsilon$ belongs, we define $L\br{\Upsilon} := \min \bbr{\ell \in \N^+ : \Upsilon \in \T^{\ell}}$.

\begin{figure}[H]
    \begin{minipage}{0.4\linewidth}
        \centering
        \tree{[$k_1$]}
        \figcaption{Graph of $[\;]_{k_1}$}
    \end{minipage}\quad
    \begin{minipage}{0.4\linewidth}
        \centering
        \tabcaption{Functions on $[\;]_k \in \T_5$}
        \begin{tabular}{@{}c|ccc@{}}
            \toprule
            $k_1$ & $L\br{\Upsilon}$ & $\abs{\Upsilon}$ & $\gamma\br{\Upsilon}$ \\ \midrule
            0 & 1 & 1 & 1 \\
            1 & 1 & 2 & 1 \\
            2 & 1 & 3 & 2 \\
            3 & 1 & 4 & 6 \\
            4 & 1 & 5 & 24 \\
            \bottomrule
        \end{tabular}
    \end{minipage}
\end{figure}
\begin{figure}[H]
    \begin{minipage}{0.4\linewidth}
        \centering
        \tree{[$k_1$ [$k_2$]]}
        \figcaption{Graph of $[[\;]_{k_2}]_{k_1}$}
    \end{minipage}\quad
    \begin{minipage}{0.45\linewidth}
        \centering
        \tabcaption{Functions on $[[\;]_{k_2}]_{k_1} \in \T_5$}
        \begin{tabular}{@{}cc|ccc@{}}
            \toprule
            $k_1$ & $k_2$ & $L\br{\Upsilon}$ & $\abs{\Upsilon}$ & $\gamma\br{\Upsilon}$ \\ \midrule
            0 & 0 & 2 & 2 & $1$  \\
            0 & 1 & 2 & 3 & $1$  \\
            0 & 2 & 2 & 4 & $2$  \\
            0 & 3 & 2 & 5 & $6$ \\
            1 & 0 & 2 & 3 & $1$  \\
            1 & 1 & 2 & 4 & $1$  \\
            1 & 2 & 2 & 5 & $2$  \\
            2 & 0 & 2 & 4 & $2$  \\
            2 & 1 & 2 & 5 & $2$  \\
            3 & 0 & 2 & 5 & $6$  \\
            \bottomrule
        \end{tabular}
    \end{minipage}
\end{figure}
\begin{figure}[H]
    \begin{minipage}{0.4\linewidth}
        \centering
        \tree{[$k_1$ [$k_3$] [$k_2$]]}
        \figcaption{Graph of $[[\;]_{k_2} [\;]_{k_3}]_{k_1}$}
    \end{minipage}\quad
    \begin{minipage}{0.5\linewidth}
        \centering
        \tabcaption{Functions on $[[\;]_{k_2} [\;]_{k_3}]_{k_1} \in \T_5$}
        \begin{tabular}{@{}ccc|ccc@{}}
            \toprule
            $k_1$ & $k_2$ & $k_3$ & $L\br{\Upsilon}$ & $\abs{\Upsilon}$ & $\gamma\br{\Upsilon}$ \\ \midrule
            0 & 0 & 0 & 2 & 3 & $2$  \\
            0 & 0 & 1 & 2 & 4 & $2$  \\
            0 & 0 & 2 & 2 & 5 & $4$  \\
            0 & 1 & 1 & 2 & 5 & $2$  \\
            1 & 0 & 0 & 2 & 4 & $2$  \\
            1 & 0 & 1 & 2 & 5 & $2$  \\
            2 & 0 & 0 & 2 & 5 & $4$  \\
            \bottomrule
        \end{tabular}
    \end{minipage}
\end{figure}

\begin{figure}[H]
    \begin{minipage}{0.25\linewidth}
        \centering
        \tree{[$k_1$ [$k_4$] [$k_3$] [$k_2$]]}        \figcaption{Graph of $[[\;]_{k_2} [\;]_{k_3} [\;]_{k_4}]_{k_1}$}
    \end{minipage}\quad\quad\quad
    \begin{minipage}{0.6\linewidth}
        \centering
        \tabcaption{Functions on $[[\;]_{k_2} [\;]_{k_3} [\;]_{k_4}]_{k_1} \in \T_5$}
        \begin{tabular}{@{}cccc|ccc@{}}
            \toprule
            $k_1$ & $k_2$ & $k_3$ & $k_4$ & $L\br{\Upsilon}$ & $\abs{\Upsilon}$ & $\gamma\br{\Upsilon}$ \\ \midrule
            0 & 0 & 0 & 0 & 2 & 4 & $6$  \\
            0 & 0 & 0 & 1 & 2 & 5 & $6$  \\
            1 & 0 & 0 & 0 & 2 & 5 & $6$  \\
            \bottomrule
        \end{tabular}
    \end{minipage}
\end{figure}

\begin{figure}[H]
    \begin{minipage}{0.25\linewidth}
        \centering
        \tree{[$k_1$ [$k_5$] [$k_4$] [$k_3$] [$k_2$]]}
        \figcaption{Graph of $[[\;]_{k_2} [\;]_{k_3} [\;]_{k_4} [\;]_{k_5}]_{k_1}$}
    \end{minipage}
    \begin{minipage}{0.74\linewidth}
        \centering
        \tabcaption{Functions on $[[\;]_{k_2} [\;]_{k_3} [\;]_{k_4} [\;]_{k_5}]_{k_1} \in \T_5$}
        \begin{tabular}{@{}ccccc|ccc@{}}
            \toprule
            $k_1$ & $k_2$ & $k_3$ & $k_4$ & $k_5$ & $L\br{\Upsilon}$ & $\abs{\Upsilon}$ & $\gamma\br{\Upsilon}$ \\ \midrule
            0 & 0 & 0 & 0 & 0 & 2 & 5 & $24$  \\
            \bottomrule
        \end{tabular}
    \end{minipage}
\end{figure}

\begin{figure}[H]
    \begin{minipage}{0.35\linewidth}
        \centering
        \tree{[$k_1$ [$k_2$ [$k_3$]]]}
        \figcaption{Graph of $[[[\;]_{k_3}]_{k_2}]_{k_1}$}
    \end{minipage}
    \begin{minipage}{0.64\linewidth}
        \centering
        \tabcaption{Functions on $[[[\;]_{k_3}]_{k_2}]_{k_1} \in \T_5$}
        \begin{tabular}{@{}ccc|ccc@{}}
            \toprule
            $k_1$ & $k_2$ & $k_3$ & $L\br{\Upsilon}$ & $\abs{\Upsilon}$ & $\gamma\br{\Upsilon}$ \\ \midrule
            0 & 0 & 0 & 3 & 3 & $1$  \\
            0 & 0 & 1 & 3 & 4 & $1$  \\
            0 & 0 & 2 & 3 & 5 & $2$  \\
            0 & 1 & 0 & 3 & 4 & $1$ \\
            0 & 1 & 1 & 3 & 5 & $1$ \\
            0 & 2 & 0 & 3 & 5 & $2$ \\
            1 & 0 & 0 & 3 & 4 & $1$  \\
            1 & 0 & 1 & 3 & 5 & $1$  \\
            1 & 1 & 0 & 3 & 5 & $1$  \\
            2 & 0 & 0 & 3 & 5 & $2$  \\
            \bottomrule
        \end{tabular}
    \end{minipage}
\end{figure}

\begin{figure}[H]
    \begin{minipage}{0.4\linewidth}
        \centering
        \tree{[$k_1$ [$k_4$] [$k_2$ [$k_3$]]]}        \figcaption{Graph of $[[[\;]_{k_3}]_{k_2} [\;]_{k_4}]_{k_1}$}
    \end{minipage}
    \begin{minipage}{0.59\linewidth}
        \centering
        \tabcaption{Functions on $[[[\;]_{k_3}]_{k_2} [\;]_{k_4}]_{k_1} \in \T_5$}
        \begin{tabular}{@{}cccc|ccc@{}}
            \toprule
            $k_1$ & $k_2$ & $k_3$ & $k_4$ & $L\br{\Upsilon}$ & $\abs{\Upsilon}$ & $\gamma\br{\Upsilon}$ \\ \midrule
            0 & 0 & 0 & 0 & 3 & 4 & $2$  \\
            0 & 0 & 0 & 1 & 3 & 5 & $2$  \\
            0 & 0 & 1 & 0 & 3 & 5 & $2$  \\
            0 & 1 & 0 & 0 & 3 & 5 & $2$  \\
            1 & 0 & 0 & 0 & 3 & 5 & $2$  \\
            \bottomrule
        \end{tabular}
    \end{minipage}
\end{figure}

\begin{figure}[H]
    \begin{minipage}{0.4\linewidth}
        \centering
        \tree{[$k_1$ [$k_2$ [$k_4$] [$k_3$]]]}        \figcaption{Graph of $[[[\;]_{k_3} [\;]_{k_4}]_{k_2}]_{k_1}$}
    \end{minipage}
    \begin{minipage}{0.59\linewidth}
        \centering
        \tabcaption{Functions on $[[[\;]_{k_3} [\;]_{k_4}]_{k_2}]_{k_1} \in \T_5$}
        \begin{tabular}{@{}cccc|ccc@{}}
            \toprule
            $k_1$ & $k_2$ & $k_3$ & $k_4$ & $L\br{\Upsilon}$ & $\abs{\Upsilon}$ & $\gamma\br{\Upsilon}$ \\ \midrule
            0 & 0 & 0 & 0 & 3 & 4 & $2$  \\
            0 & 0 & 0 & 1 & 3 & 5 & $2$  \\
            0 & 1 & 0 & 0 & 3 & 5 & $2$  \\
            1 & 0 & 0 & 0 & 3 & 5 & $2$  \\
            \bottomrule
        \end{tabular}
    \end{minipage}
\end{figure}

\begin{figure}[H]
    \begin{minipage}{0.3\linewidth}
        \centering
        \tree{[$k_1$ [$k_2$ [$k_5$] [$k_4$] [$k_3$]]]}       \figcaption{Graph of $[[[\;]_{k_3} [\;]_{k_4} [\;]_{k_5}]_{k_2}]_{k_1}$}
    \end{minipage}\quad\quad
    \begin{minipage}{0.64\linewidth}
        \centering
        \tabcaption{Functions on $[[[\;]_{k_3} [\;]_{k_4} [\;]_{k_5}]_{k_2}]_{k_1} \in \T_5$}
        \begin{tabular}{@{}ccccc|ccc@{}}
            \toprule
            $k_1$ & $k_2$ & $k_3$ & $k_4$ & $k_5$ & $L\br{\Upsilon}$ & $\abs{\Upsilon}$ & $\gamma\br{\Upsilon}$ \\ \midrule
            0 & 0 & 0 & 0 & 0 & 3 & 5 & $6$  \\
            \bottomrule
            \end{tabular}
    \end{minipage}
\end{figure}

\begin{figure}[H]
    \begin{minipage}{0.3\linewidth}
        \centering
        \tree{[$k_1$ [$k_4$] [$k_2$ [$k_5$] [$k_3$]]]}       \figcaption{Graph of $[[[\;]_{k_3} [\;]_{k_5}]_{k_2} [\;]_{k_4}]_{k_1}$}
    \end{minipage}\quad\quad
    \begin{minipage}{0.64\linewidth}
        \centering
        \tabcaption{Functions on $[[[\;]_{k_3} [\;]_{k_5}]_{k_2} [\;]_{k_4}]_{k_1} \in \T_5$}
        \begin{tabular}{@{}ccccc|ccc@{}}
            \toprule
            $k_1$ & $k_2$ & $k_3$ & $k_4$ & $k_5$ & $L\br{\Upsilon}$ & $\abs{\Upsilon}$ & $\gamma\br{\Upsilon}$ \\ \midrule
            0 & 0 & 0 & 0 & 0 & 3 & 5 & $4$  \\
            \bottomrule
            \end{tabular}
    \end{minipage}
\end{figure}

\begin{figure}[H]
    \begin{minipage}{0.3\linewidth}
        \centering
        \tree{[$k_1$ [$k_5$] [$k_4$] [$k_2$ [$k_3$]]]}       \figcaption{Graph of $[[[\;]_{k_3}]_{k_2} [\;]_{k_4} [\;]_{k_5}]_{k_1}$}
    \end{minipage}\quad\quad
    \begin{minipage}{0.64\linewidth}
        \centering
        \tabcaption{Functions on $[[[\;]_{k_3}]_{k_2} [\;]_{k_4} [\;]_{k_5}]_{k_1} \in \T_5$}
        \begin{tabular}{@{}ccccc|ccc@{}}
            \toprule
            $k_1$ & $k_2$ & $k_3$ & $k_4$ & $k_5$ & $L\br{\Upsilon}$ & $\abs{\Upsilon}$ & $\gamma\br{\Upsilon}$ \\ \midrule
            0 & 0 & 0 & 0 & 0 & 3 & 5 & $6$  \\
            \bottomrule
            \end{tabular}
    \end{minipage}
\end{figure}

\begin{figure}[H]
    \begin{minipage}{0.3\linewidth}
        \centering
        \tree{[$k_1$ [$k_4$ [$k_5$]] [$k_2$ [$k_3$]]]}       \figcaption{Graph of $[[[\;]_{k_3}]_{k_2} [[\;]_{k_5}]_{k_4}]_{k_1}$}
    \end{minipage}\quad\quad
    \begin{minipage}{0.64\linewidth}
        \centering
        \tabcaption{Functions on $[[[\;]_{k_3}]_{k_2} [[\;]_{k_5}]_{k_4}]_{k_1} \in \T_5$}
        \begin{tabular}{@{}ccccc|ccc@{}}
            \toprule
            $k_1$ & $k_2$ & $k_3$ & $k_4$ & $k_5$ & $L\br{\Upsilon}$ & $\abs{\Upsilon}$ & $\gamma\br{\Upsilon}$ \\ \midrule
            0 & 0 & 0 & 0 & 0 & 3 & 5 & $2$  \\
            \bottomrule
            \end{tabular}
    \end{minipage}
\end{figure}

\begin{figure}[H]
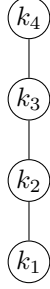

    \begin{minipage}{0.4\linewidth}
        \centering
        \tree{[$k_1$ [$k_2$ [$k_3$ [$k_4$]]]]}
        \figcaption{Graph of $[[[[\;]_{k_4}]_{k_3}]_{k_2}]_{k_1}$}
    \end{minipage}
    \begin{minipage}{0.59\linewidth}
        \centering
        \tabcaption{Functions on $[[[[\;]_{k_4}]_{k_3}]_{k_2}]_{k_1} \in \T_5$}
        \begin{tabular}{@{}cccc|ccc@{}}
            \toprule
            $k_1$ & $k_2$ & $k_3$ & $k_4$ & $L\br{\Upsilon}$ & $\abs{\Upsilon}$ & $\gamma\br{\Upsilon}$ \\ \midrule
            0 & 0 & 0 & 0 & 4 & 4 & $1$  \\
            0 & 0 & 0 & 1 & 4 & 5 & $1$  \\
            0 & 0 & 1 & 0 & 4 & 5 & $1$  \\
            0 & 1 & 0 & 0 & 4 & 5 & $1$  \\
            1 & 0 & 0 & 0 & 4 & 5 & $1$  \\
            \bottomrule
        \end{tabular}
    \end{minipage}
\end{figure}

\begin{figure}[H]
    \begin{minipage}{0.3\linewidth}
        \centering
        \tree{[$k_1$ [$k_5$] [$k_2$ [$k_3$ [$k_4$ ]]]]}
        \figcaption{Graph of $[[[[\;]_{k_4}]_{k_3}]_{k_2} [\;]_{k_5}]_{k_1}$}
    \end{minipage}\quad\quad
    \begin{minipage}{0.64\linewidth}
        \centering
        \tabcaption{Functions on $[[[[\;]_{k_4}]_{k_3}]_{k_2} [\;]_{k_5}]_{k_1} \in \T_5$}
        \begin{tabular}{@{}ccccc|ccc@{}}
            \toprule
            $k_1$ & $k_2$ & $k_3$ & $k_4$ & $k_5$ & $L\br{\Upsilon}$ & $\abs{\Upsilon}$ & $\gamma\br{\Upsilon}$ \\ \midrule
            0 & 0 & 0 & 0 & 0 & 4 & 5 & $2$  \\
            \bottomrule
        \end{tabular}
    \end{minipage}
\end{figure}

\begin{figure}[H]
    \begin{minipage}{0.3\linewidth}
        \centering
        \tree{[$k_1$ [$k_2$ [$k_5$] [$k_3$ [$k_4$ ]]]]}
        \figcaption{Graph of $[[[[\;]_{k_4}]_{k_3} [\;]_{k_5}]_{k_2}]_{k_1}$}
    \end{minipage}\quad\quad
    \begin{minipage}{0.64\linewidth}
        \centering
        \tabcaption{Functions on $[[[[\;]_{k_4}]_{k_3} [\;]_{k_5}]_{k_2}]_{k_1} \in \T_5$}
        \begin{tabular}{@{}ccccc|ccc@{}}
            \toprule
            $k_1$ & $k_2$ & $k_3$ & $k_4$ & $k_5$ & $L\br{\Upsilon}$ & $\abs{\Upsilon}$ & $\gamma\br{\Upsilon}$ \\ \midrule
            0 & 0 & 0 & 0 & 0 & 4 & 5 & $2$  \\
            \bottomrule
        \end{tabular}
    \end{minipage}
\end{figure}

\begin{figure}[H]
    \begin{minipage}{0.3\linewidth}
        \centering
        \tree{[$k_1$ [$k_2$ [$k_3$ [$k_5$] [$k_4$ ]]]]}
        \figcaption{Graph of $[[[[\;]_{k_4} [\;]_{k_5}]_{k_3}]_{k_2}]_{k_1}$}
    \end{minipage}\quad\quad
    \begin{minipage}{0.64\linewidth}
        \centering
        \tabcaption{Functions on $[[[[\;]_{k_4} [\;]_{k_5}]_{k_3}]_{k_2}]_{k_1} \in \T_5$}
        \begin{tabular}{@{}ccccc|ccc@{}}
            \toprule
            $k_1$ & $k_2$ & $k_3$ & $k_4$ & $k_5$ & $L\br{\Upsilon}$ & $\abs{\Upsilon}$ & $\gamma\br{\Upsilon}$ \\ \midrule
            0 & 0 & 0 & 0 & 0 & 4 & 5 & $2$  \\
            \bottomrule
        \end{tabular}
    \end{minipage}
\end{figure}

\begin{figure}[H]
    \begin{minipage}{0.3\linewidth}
        \centering
        \tree{[$k_1$ [$k_2$ [$k_3$ [$k_4$ [$k_5$]]]]]}
        \figcaption{Graph of $[[[[[\;]_{k_5}]_{k_4}]_{k_3}]_{k_2}]_{k_1}$}
    \end{minipage}\quad\quad
    \begin{minipage}{0.64\linewidth}
        \centering
        \tabcaption{Functions on $[[[[[\;]_{k_5}]_{k_4}]_{k_3}]_{k_2}]_{k_1} \in \T_5$}
        \begin{tabular}{@{}ccccc|ccc@{}}
            \toprule
            $k_1$ & $k_2$ & $k_3$ & $k_4$ & $k_5$ & $L\br{\Upsilon}$ & $\abs{\Upsilon}$ & $\gamma\br{\Upsilon}$ \\ \midrule
            0 & 0 & 0 & 0 & 0 & 5 & 5 & $1$  \\
            \bottomrule
        \end{tabular}
    \end{minipage}
\end{figure}

\section{Runge-Kutta schemes found by numerical optimization}\label{sec_RKnumopt}

Below we list the coefficients of Scheme~\ref{sch_rk}, obtained by numerically minimizing $\sum_{i, j=1}^m a_{ij}^2 + \sum_{i=1}^m b_i^2$ over all $a_{ij}$ and $b_i$, subject to the $r$th order conditions in Table~\ref{tab_ordcond}, with $c_i = 1 - i/m$ for $i = 0, 1, \cdots, m$.
The stage number $m$ is chosen as $m = 5, 6$ for $r = 4$, and $m = 7, 8$ for $r = 5$.
All computations are performed in double precision.
Coefficients not listed are zero.

\begin{itemize}
    \item the $5$-stage $4$th-order scheme:
        \begin{equation*}
            \text{RK}(4; 5): \quad \begin{aligned}
                &a_{45} = 0.2, \quad a_{35} = -0.13242233706502626\\
                &a_{34} = 0.5324223370650263, \quad a_{25} = 0.019886552336434136, \\
                &a_{24} = 0.24294264564990678, \quad a_{23} = 0.3371708020136592, \\
                &a_{15} = 0.1963917348918804, \quad a_{14} = -0.005472324607713924, \\
                &a_{13} = 0.2464094480422048, \quad a_{12} = 0.36267114167362885, \\
                &b_5 = 0.04623081469521772, \quad b_4 = 0.27341007455246313, \\
                &b_3 = 0.31905155483797115, \quad b_2 = -0.1432565921142024, \\
                &b_1 = 0.5045641480285504, \quad c_i = 1 - 0.2 i \sptext{for} i = 5, 4, \cdots, 0.
            \end{aligned}
        \end{equation*}
    \item the $6$-stage $4$th-order scheme:
        \begin{equation*}
            \text{RK}(4; 6): \quad \begin{aligned}
                &a_{56} = 1/6, \quad a_{46} = -0.025061513556965075, \\
                &a_{45} = 0.3583948468902984, \quad a_{36} = 0.13841067828110443, \\
                &a_{35} = 0.1693249424852804, \quad a_{34} = 0.1922643792336152\\
                &a_{26} = -0.0038511700476725353, \quad a_{25} = 0.09581249968889799, \\
                &a_{24} = 0.22659933188130044, \quad a_{23} = 0.34810600514414075, \\
                &a_{16} = 0.16161918775745476, \quad a_{15} = 0.10244745111384074, \\
                &a_{14} = 0.1276198408649451, \quad a_{13} = 0.12764301041543388, \\
                &a_{12} = 0.31400384318165875, \quad b_6 = 0.03548077842498064, \\
                &b_5 = 0.2095257367170438, \quad b_4 = 0.3070892688820228, \\
                &b_3 = -0.033230011198136004, \quad b_2 = 0.12968537675712494, \\
                &b_1 = 0.3514488504169638, \quad c_i = 1 - i/6 \sptext{for} i = 6, 5, \cdots, 0.
            \end{aligned}
        \end{equation*}
    \item the $7$-stage $5$th-order scheme:
        \begin{equation*}
            \text{RK}(5; 7): \quad \begin{aligned}
                &a_{67} = 1/7, \quad a_{57} = -0.00849922984652299, \\
                &a_{56} = 0.2942135155608087, \quad a_{47} = 0.012111601574373559, \\
                &a_{46} = 0.027194944061276653, \quad a_{45} = 0.3892648829357783,\\
                &a_{37} = 0.17649344547970258, \quad a_{36} = 0.038519474866316364, \\
                &a_{35} = 0.013295135548829034, \quad a_{34} = 0.3431205155337234, \\
                &a_{27} = 0.20775831541493936, \quad a_{26} = -0.0013035026399939747, \\
                &a_{25} = 0.05244696764836481, \quad a_{24} = 0.24416745620882727, \\
                &a_{23} = 0.21121647765357676, \quad a_{17} = 0.04201635027665935, \\
                &a_{16} = 0.16939620723268453, \quad a_{15} = 0.17374202049423548, \\
                &a_{14} = 0.1849766339309189, \quad a_{13} = -0.053226048370035584, \\
                &a_{12} = 0.3402376935783944, \quad b_7 = 0.09233840526733901, \\
                &b_6 = 0.018048818525724868, \quad b_5 = 0.3316242728057178, \\
                &b_4 = 0.021246618173088202, \quad b_3 = 0.2976169110457616, \\
                &b_2 = -0.12974529206631724, \quad b_1 = 0.3688702662486858\\
                &c_i = 1 - i/7 \sptext{for} i = 7, 6, \cdots, 0.
            \end{aligned}
        \end{equation*}
    \item the $8$-stage $5$th-order scheme:
        \begin{equation*}
            \text{RK}(5; 8): \quad \begin{aligned}
                &a_{78} = 1/8, \quad a_{68} = -0.019695026811446536, \\
                &a_{67} = 0.2696950268114465, \quad a_{58} = 0.12027883922164721, \\
                &a_{57} = -0.025389113361355534, \quad a_{56} = 0.28011027413970835, \\
                &a_{48} = 0.17348303333838974, \quad a_{47} = 0.15121241994785925, \\
                &a_{46} = 0.11261192763186804, \quad a_{45} = 0.06269261908188303, \\
                &a_{38} = 0.027377326034266746, \quad a_{37} = 0.048508067745817715, \\
                &a_{36} = 0.14412177807635515, \quad a_{35} = 0.2178666996537625, \\
                &a_{34} = 0.18712612848979787, \quad a_{28} = 0.12076055643080501, \\
                &a_{27} = 0.07237536565936713, \quad a_{26} = 0.1519401323946871, \\
                &a_{25} = 0.16024192906367374, \quad a_{24} = -0.07416278224686454, \\
                &a_{23} = 0.3188447986983316, \quad a_{18} = 0.03658677645498778, \\
                &a_{17} = 0.08834233979368933, \quad a_{16} = 0.20873825484619646, \\
                &a_{15} = 0.22487804433654632, \quad a_{14} = -0.017586277300601805, \\
                &a_{13} = 0.03916189598076387, \quad a_{12} = 0.29487896588841805, \\
                &b_{8} = 0.10264676180489864, \quad b_7 = -0.005963693910170281, \\
                &b_6 = 0.14240515226981773, \quad b_5 = 0.3625707619097686, \\
                &b_4 = -0.01436402197198115, \quad b_3 = 0.044156049175500874, \\
                &b_2 = 0.10972598940350983, \quad b_1 = 0.25882300131865577, \\
                &c_i = 1 - i/8 \sptext{for} i = 8, 7, \cdots, 0.
            \end{aligned}
        \end{equation*}
\end{itemize}

\clearpage
\section[Explicit expression of C(r)]{Explicit expression of $C(r)$ in \eqref{eq_defCr}}\label{sec_expcod}

The order conditions $C(r)$ in \eqref{eq_defCr} can be stated explicitly by substituting the elementary weights $A_i^m(\cdot)$ from \Cref{def_weight} and the enumeration of ULN-trees from \cref{eq1_enum_T,eq2_enum_T,eq3_enum_T} into \eqref{eq_defCr}.
Specifically, for $r = 1$, we have 
\begin{equation}\label{eq_exp_C1}
    C(1):  
    \left\{
    \begin{aligned}
    &A_i^m\br{[\;]_0} = \sum_{j=1}^m a_{ij} - c_i = 0, \quad i = 1, 2,\cdots, m, \\
    &A_0^m\br{[\;]_0} = \sum_{j=1}^m b_j - 1 = 0, 
    \end{aligned}
    \right.
\end{equation}
where we note that $C(1)$ does not involve the condition $A_0^m\br{\Upsilon} = 0$, $\Upsilon \in \T_{1-}$ because $\T_{1-} = \emptyset$.
To derive $C(2)$, recalling \cref{eq1_enum_T}, we have $\T_{2-} = \T_{1-} \bigcup \T_{2-}$; hence $C(2)$ can be stated as $C(1)$ together with the additional conditions $A_0^m\br{\Upsilon} = 0$ for $\Upsilon \in \T_{2-} \backslash \T_{1-}$, i.e.,
\begin{equation}\label{eq_exp_C2}
    C(2):
    \left\{
    \begin{aligned}
        &C(1) \text{ holds}, \\
        &A_0^m\br{[\;]_1} = \sum_{j=1}^m b_j\br{1 - c_j} - \frac{1}{2} = 0. 
    \end{aligned}
    \right. 
\end{equation}
Similarly, we obtain $C(3)$, $C(4)$, and $C(5)$ as follows:
\begin{equation}\label{eq_exp_C3}
    C(3):
    \left\{
    \begin{aligned}
        &C(2) \text{ holds}, \\
        &A_0^m\br{[\;]_2} = \sum_{j=1}^m b_j\br{1 - c_j}^2 - \frac{1}{3} = 0, \\
        &A_0^m\br{[[\;]_{1}]_{0}} = \sum_{i = 1}^{m} b_i \vbr{\Big}{\sum_{j=1}^m a_{ij} \br{c_i - c_j} - \frac{c_i^2}{2}} = 0; 
    \end{aligned}
    \right. 
\end{equation}
\begin{equation}\label{eq_exp_C4}
    C(4):
    \left\{
    \begin{aligned}
        &C(3) \text{ holds}, \\
        &A_0^m\br{[\;]_3} = \sum_{j=1}^m b_j\br{1 - c_j}^3 - \frac{1}{4} = 0, \\
        &A_0^m\br{[[\;]_2]_0} = \sum_{j = 1}^{m} b_j \vbr{\Big}{\sum_{k=1}^m a_{jk} \br{c_j - c_k}^2 - \frac{c_j^{3}}{3}} = 0, \\
        &A_0^m\br{[[\;]_1]_1} = \sum_{j = 1}^{m} b_j \vbr{\Big}{\sum_{k=1}^m \br{1 - c_j} a_{jk} \br{c_j - c_k} - \br{1 - c_j} \frac{c_j^{2}}{2}} = 0, \\
        &A_0^m\br{[[[\;]_1]_0]_0} = \sum_{i = 1}^{m} b_i \sum_{j = 1}^{m} a_{ij} \vbr{\Big}{\sum_{k=1}^m a_{jk} \br{c_j - c_k} - \frac{c_j^{2}}{2}} = 0;
    \end{aligned}
    \right.
\end{equation}
\begin{equation}\label{eq_exp_C5}
    C(5):
    \left\{
    \begin{aligned}
        &C(4) \text{ holds}, \\
        &A_0^m\br{[\;]_4} = \sum_{j=1}^m b_j\br{1 - c_j}^4 - \frac{1}{5} = 0, \\
        &A_0^m\br{[[\;]_3]_0} = \sum_{j = 1}^{m} b_j \vbr{\Big}{\sum_{k=1}^m a_{jk} \br{c_j - c_k}^3 - \frac{c_j^{4}}{4}} = 0, \\
        &A_0^m\br{[[\;]_2]_1} = \sum_{i = 1}^{m} b_i \br{1 - c_i} \vbr{\Big}{\sum_{j=1}^{m} a_{ij} \br{c_i - c_j}^{2} - \frac{c_i^{3}}{3}} = 0, \\
        &A_0^m\br{[[\;]_1]_2} = \sum_{i = 1}^{m} b_i \br{1 - c_i}^2 \vbr{\Big}{\sum_{j=1}^{m} a_{ij} \br{c_i - c_j} - \frac{c_i^{2}}{2}} = 0, \\
        &A_0^m\br{[[[\;]_2]_0]_0} = \sum_{i=1}^m b_i \sum_{j=1}^{m} a_{ij} \vbr{\Big}{\sum_{k=1}^{m} a_{jk} \br{c_j - c_k}^2 - \frac{c_j^{3}}{3}} = 0, \\
        &A_0^m\br{[[\;]_1 [\;]_1]_0} = \sum_{i=1}^m b_i \vbr{\Big}{\sum_{j=1}^{m} a_{ij} \br{c_i - c_j} - \frac{c_i^{2}}{2}}^2 = 0,\\
        &A_0^m\br{[[[\;]_1]_1]_0} = \sum_{i=1}^m b_i \sum_{j = 1}^{m} a_{ij} \br{c_i - c_j} \vbr{\Big}{\sum_{k=1}^{m} a_{jk} \br{c_j - c_k} - \frac{c_j^{2}}{2}} = 0, \\
        &A_0^m\br{[[[\;]_1]_0]_1} = \sum_{i=1}^m b_i (1 - c_i) \sum_{j = 1}^{m} a_{ij} \vbr{\Big}{\sum_{k=1}^{m} a_{jk} \br{c_j - c_k} - \frac{c_j^{2}}{2}} = 0, \\
        &A_0^m\br{[[[[\;]_1]_0]_0]_0} = \sum_{i=1}^m b_i \sum_{j=1}^{m} a_{ij} \sum_{k=1}^{m} a_{jk} \vbr{\Big}{\sum_{l=1}^{m} a_{kl} \br{c_k - c_l} - \frac{c_k^{2}}{2}} = 0.
    \end{aligned}
    \right.
\end{equation}

\section[Proof of the order-condition proposition]{Proof of \Cref{prop_Cr_table} given in \cref{sec_ord_condi}}\label{sec_proof}

\begin{proof}[Proof of \Cref{prop_Cr_table}]
    
    It is obvious that the last two conditions in \eqref{eq_ajieq0_maxabc} of \Cref{assu_coeff} imply, respectively, the conditions (2) and (1) in \Cref{tab_ordcond}.
    The remaining proof is divided into five parts.
    
    \medskip
    \noindent
    \textbf{(i) Proof of} $C(1) \iff$ (3) -- (4). 
    Comparing \eqref{eq_exp_C1} with the conditions (3) and (4) in Table~\ref{tab_ordcond} immediately yields the result.  

    \medskip
    \noindent
    \textbf{(ii) Proof of} $C(2) \iff$ (3) -- (5). 
    By \eqref{eq_exp_C2}, $C(2)$ includes $C(1)$. 
    By (i), $C(1)$ is equivalent to the conditions (3)--(4) in Table~\ref{tab_ordcond}.
    Thus, it suffices to show that, under (3) -- (4), the second condition in \eqref{eq_exp_C2} is equivalent to the condition (5) in Table~\ref{tab_ordcond}.
    Actually, this equivalence can be obtained by
    \begin{align*}
            &\text{the second condition in \eqref{eq_exp_C2}} \\
        \iff&\sum_{j=1}^m b_j - \sum_{j=1}^m b_j c_j - \frac{1}{2} = 0, \\
        \iff&1 - \sum_{j=1}^m b_j c_j - \frac{1}{2} = 0,\\
        \iff&\text{the condition (5) in Table~\ref{tab_ordcond}},
    \end{align*}
    where the second equivalence follows from the condition (4) in Table~\ref{tab_ordcond}.
    
    \medskip
    \noindent
    \textbf{(iii) Proof of} $C(3) \iff$ (3) -- (7). 
    By \eqref{eq_exp_C3}, $C(3)$ includes $C(2)$.
    By (ii), $C(2)$ is equivalent to the conditions (3) -- (5) in Table~\ref{tab_ordcond}.
    Thus it suffices to show that, under (3) -- (5), the second and third conditions in \eqref{eq_exp_C3} are equivalent to the conditions (6) and (7) in Table~\ref{tab_ordcond}, respectively. 
    
    Under (3) -- (5) in Table~\ref{tab_ordcond}, we have
    \begin{align*}
            &\text{the second condition in \eqref{eq_exp_C3}}\\
        \iff&\sum_{j=1}^m b_j - 2 \sum_{j=1}^m b_j c_j + \sum_{j=1}^m b_j c_j^2 - \frac{1}{3} = 0\\
        \iff&1 - 2 \times \frac{1}{2} + \sum_{j=1}^m b_j c_j^2 - \frac{1}{3} = 0 \\
        \iff&\text{the condition (6) in Table~\ref{tab_ordcond}},
    \end{align*}
    where the second equivalence is obtained by inserting the conditions (4) and (5) in Table~\ref{tab_ordcond};
    \begin{align*}
            &\text{the third condition in \eqref{eq_exp_C3}}\\
        \iff&\sum_{i = 1}^{m} b_i c_i \sum_{j=1}^m a_{ij}  - \sum_{i = 1}^{m} \sum_{j=1}^m b_i a_{ij} c_j - \frac{1}{2} \sum_{i = 1}^{m} b_i c_i^2 = 0 \\
        \iff&\sum_{i = 1}^{m} b_i c_i^2  - \sum_{i = 1}^{m} \sum_{j=1}^m b_i a_{ij} c_j - \frac{1}{2} \sum_{i = 1}^{m} b_i c_i^2 = 0 \\
        \iff&\frac{1}{3}  - \sum_{i = 1}^{m} \sum_{j=1}^m b_i a_{ij} c_j - \frac{1}{2}\times \frac{1}{3} = 0\\
        \iff&\text{the condition (7) in Table~\ref{tab_ordcond}},
    \end{align*}
    where the second and third equivalences are derived by substituting the conditions (3) and (6) from Table~\ref{tab_ordcond}, respectively.

    \medskip
    \noindent
    \textbf{(iv) Proof of } $C(4) \iff$ (3) -- (11).
    By \eqref{eq_exp_C4}, $C(4)$ includes $C(3)$.
    By (iii), $C(3)$ is equivalent to the conditions (3) -- (7) in Table~\ref{tab_ordcond}. 
    Hence, it suffices to show that, under (3) -- (7), the second through fifth conditions in \eqref{eq_exp_C4} are, respectively, equivalent to the conditions (8) -- (11) in Table~\ref{tab_ordcond}.
    
    Under (3) -- (7) in Table~\ref{tab_ordcond}, we have
    \begin{equation}\label{eq1_proof_C4}
    \begin{aligned}
            &\text{the second condition in \eqref{eq_exp_C4}}\\
        \iff&\sum_{j=1}^m b_j - 3 \sum_{j=1}^m b_j c_j + 3 \sum_{j=1}^m b_j c_j^2 - \sum_{j=1}^m b_j c_j^3 - \frac{1}{4} = 0\\
        \iff&1 - 3 \times \frac{1}{2} + 3 \times \frac{1}{3} - \sum_{j=1}^m b_j c_j^3 - \frac{1}{4} = 0\\
        \iff&\text{the condition (8) in Table~\ref{tab_ordcond}},
    \end{aligned}
    \end{equation}
    where the second equivalence is obtained by inserting the conditions (4), (5) and (6) from Table~\ref{tab_ordcond}. 
    The equivalence given in \eqref{eq1_proof_C4} allows us to use the condition (8) in Table~\ref{tab_ordcond} to simplify the remaining conditions in \eqref{eq_exp_C4}. 
    
    We temporarily skip the third condition in \eqref{eq_exp_C4} and prove the equivalences between the fourth condition and (9) first.
    Specifically, under (3) -- (8) in Table~\ref{tab_ordcond}, we have
    \begin{equation}\label{eq2_proof_C4}
    \begin{aligned}
            &\text{the fourth condition in \eqref{eq_exp_C4}},\\
        \iff&\sum_{j = 1}^{m} b_j (1 - c_j) \vbr{\Big}{c_j \sum_{k=1}^m a_{jk} - \sum_{k=1}^m a_{jk} c_k - \frac{1}{2} c_j^2} = 0 \\
        \iff&\sum_{j = 1}^{m} b_j (1 - c_j) \vbr{\Big}{\frac{1}{2} c_j^2 - \sum_{k=1}^m a_{jk} c_k} = 0 \\
        \iff& \frac{1}{2} \sum_{j = 1}^{m} b_j c_j^{2} - \frac{1}{2} \sum_{j = 1}^{m} b_j c_j^{3} - \sum_{j = 1}^{m} \sum_{k=1}^m b_j a_{jk} c_k + \sum_{j = 1}^{m} \sum_{k=1}^m b_j c_j a_{jk} c_k = 0\\
        \iff& \frac{1}{2} \times \frac{1}{3} - \frac{1}{2} \times \frac{1}{4} - \frac{1}{6} + \sum_{j = 1}^{m} \sum_{k=1}^m b_j c_j a_{jk} c_k = 0\\
        \iff&\text{the condition (9) in Table~\ref{tab_ordcond}},
    \end{aligned}
    \end{equation}
    where the second equivalence follows from the condition (3); the fourth equivalence is derived by substituting the conditions (6), (7) and (8) from Table~\ref{tab_ordcond}.
    The equivalence given in \eqref{eq2_proof_C4} allows us to use the condition (9) in Table~\ref{tab_ordcond} to simplify the remaining conditions in \eqref{eq_exp_C4}.
    
    Now we turn to handle the third condition in \eqref{eq_exp_C4}. 
    Under (3) -- (9) in Table~\ref{tab_ordcond}, 
    \begin{align*}
            &\text{the third condition in \eqref{eq_exp_C4}}\\
        \iff&\sum_{j = 1}^{m} b_j c_j^2 \sum_{k=1}^m a_{jk} - 2 \sum_{j = 1}^{m} \sum_{k=1}^m b_j c_j a_{jk} c_k + \sum_{j = 1}^{m} \sum_{k=1}^m b_j a_{jk} c_k^2 - \frac{1}{3}\sum_{j = 1}^{m} b_j c_j^3 = 0 \\
        \iff&\sum_{j = 1}^{m} b_j c_j^3 - 2 \sum_{j = 1}^{m} \sum_{k=1}^m b_j c_j a_{jk} c_k + \sum_{j = 1}^{m} \sum_{k=1}^m b_j a_{jk} c_k^2 - \frac{1}{3}\sum_{j = 1}^{m} b_j c_j^3 = 0\\
        \iff&\frac{1}{4} - 2 \times \frac{1}{8} + \sum_{j = 1}^{m} \sum_{k=1}^m b_j a_{jk} c_k^2 - \frac{1}{3} \times \frac{1}{4} = 0\\
        \iff&\text{the condition (10) in Table~\ref{tab_ordcond}},
    \end{align*}
    where the second equivalence is obtained by inserting the condition (3); the third equivalence follows from the conditions (8) and (9) in Table~\ref{tab_ordcond}.
    Similarly, under (3) -- (10) in Table~\ref{tab_ordcond}, we have
    \begin{align*}
            &\text{the fifth condition in \eqref{eq_exp_C4}}\\
        \iff&\sum_{i = 1}^{m} b_i \sum_{j = 1}^{m} a_{ij} \vbr{\Big}{\sum_{k=1}^m a_{jk} \br{c_j - c_k} - \frac{c_j^{2}}{2}} = 0 \\
        \iff&\sum_{i = 1}^{m} b_i \sum_{j = 1}^{m} a_{ij} \vbr{\Big}{\frac{c_j^{2}}{2} - \sum_{k=1}^m a_{jk} c_k} = 0 \\
        \iff& \frac{1}{2} \times \frac{1}{12} - \sum_{i = 1}^{m} \sum_{j = 1}^{m} \sum_{k=1}^m b_i a_{ij} a_{jk} c_k = 0 \\
        \iff&\text{the condition (11) in Table~\ref{tab_ordcond}},
    \end{align*}
    where the second equivalence follows from the condition (3); the third equivalence is derived by substituting the condition (10) from Table~\ref{tab_ordcond}.
    
    \medskip
    \noindent
    \textbf{(v) Proof of } $C(5) \iff$ (3) -- (20).
    Proceeding as in (i) -- (iv), it suffices to verify:
    \begin{enumerate}
        \item under (3) -- (11), the condition (12) $\iff$ the second condition in \eqref{eq_exp_C5};
        \item under (3) -- (12), the condition (13) $\iff$ the fifth condition in \eqref{eq_exp_C5};
        \item under (3) -- (13), the condition (14) $\iff$ the fourth condition in \eqref{eq_exp_C5};
        \item under (3) -- (14), the condition (15) $\iff$ the ninth condition in \eqref{eq_exp_C5};
        \item under (3) -- (15), the condition (16) $\iff$ the seventh condition in \eqref{eq_exp_C5};
        \item under (3) -- (16), the condition (17) $\iff$ the third condition in \eqref{eq_exp_C5};
        \item under (3) -- (17), the condition (18) $\iff$ the eighth condition in \eqref{eq_exp_C5};
        \item under (3) -- (18), the condition (19) $\iff$ the sixth condition in \eqref{eq_exp_C5};
        \item under (3) -- (19), the condition (20) $\iff$ the tenth condition in \eqref{eq_exp_C5}. 
    \end{enumerate}
    The details are as follows.
    \begin{align*}
            &\text{the second condition in \eqref{eq_exp_C5}}\\
        \iff&\sum_{j=1}^m b_j - 4 \sum_{j=1}^m b_j c_j + 6 \sum_{j=1}^m b_j c_j^2 - 4 \sum_{j=1}^m b_j c_j^3 + \sum_{j=1}^m b_j c_j^4 - \frac{1}{5} = 0\\
        \iff&1 - 4 \times \frac{1}{2} + 6 \times \frac{1}{3} - 4 \times \frac{1}{4} + \sum_{j=1}^m b_j c_j^4 - \frac{1}{5} = 0\\
        \iff&\text{the condition (12) in Table~\ref{tab_ordcond}},
    \end{align*}
    where the second equivalence is obtained by inserting the conditions (4), (6) and (8) from Table~\ref{tab_ordcond};
    \begin{align*}
        &\text{the fifth condition in \eqref{eq_exp_C5}}\\
        \iff& \sum_{i = 1}^{m} \sum_{j=1}^{m} b_i c_i a_{ij}  - \sum_{i = 1}^{m} \sum_{j=1}^{m} b_i a_{ij} c_j - 2 \sum_{i = 1}^{m} \sum_{j=1}^{m}  b_i c_i c_i a_{ij} + 2 \sum_{i = 1}^{m} \sum_{j=1}^{m} b_i c_i a_{ij} c_j \\
        & - \frac{1}{2} \sum_{i = 1}^{m} b_i c_i^{2} + \sum_{i = 1}^{m} b_i c_i c_i^{2} + \sum_{i = 1}^{m} \sum_{j=1}^{m} b_i c_i^2 c_i a_{ij} - \sum_{i = 1}^{m} \sum_{j=1}^{m} b_i c_i^2 a_{ij} c_j = \frac{1}{2} \sum_{i = 1}^{m} b_i c_i^2 c_i^{2}  \\
        \iff& \sum_{i = 1}^{m} b_i c_i^2  - \sum_{i = 1}^{m} \sum_{j=1}^{m} b_i a_{ij} c_j - 2 \sum_{i = 1}^{m} b_i c_i^3 + 2 \sum_{i = 1}^{m} \sum_{j=1}^{m} b_i c_i a_{ij} c_j \\
        & - \frac{1}{2} \sum_{i = 1}^{m} b_i c_i^{2} + \sum_{i = 1}^{m} b_i c_i^{3} + \sum_{i = 1}^{m} b_i c_i^4 - \sum_{i = 1}^{m} \sum_{j=1}^{m} b_i c_i^2 a_{ij} c_j - \frac{1}{2} \sum_{i = 1}^{m} b_i c_i^4 = 0\\
        \iff& \frac{1}{3} - \frac{1}{6} - 2 \times \frac{1}{4} + 2 \times \frac{1}{8} - \frac{1}{2} \times \frac{1}{3} + \frac{1}{4} + \frac{1}{5} - \sum_{i = 1}^{m} \sum_{j=1}^{m} b_i c_i^2 a_{ij} c_j - \frac{1}{2} \times \frac{1}{5} = 0 \\
        \iff& \text{the condition (13) in Table~\ref{tab_ordcond}},
    \end{align*}
    where the second equivalence is obtained by inserting the condition (3); the third equivalence follows from the conditions (6), (7), (8), (9) and (12) in Table~\ref{tab_ordcond};
    \begin{align*}
        &\text{the fourth condition in \eqref{eq_exp_C5}}\\
        \iff&\sum_{i = 1}^{m} b_i (1 - c_i) \vbr{\Big}{\sum_{j=1}^{m} a_{ij} \br{c_i - c_j}^{2} - \frac{c_i^{3}}{3}} = 0 \\
        \iff&  \sum_{i = 1}^{m}  \sum_{j=1}^{m} b_i c_i^2 a_{ij} - 2 \sum_{i = 1}^{m} \sum_{j=1}^{m} b_i c_i a_{ij} c_j + \sum_{i = 1}^{m} \sum_{j=1}^{m} b_i a_{ij} c_j^2 - \frac{1}{3} \sum_{i = 1}^{m} b_i c_i^{3} \\
        & - \sum_{i = 1}^{m} \sum_{j=1}^{m} b_i c_i^3 a_{ij} + 2 \sum_{i = 1}^{m} \sum_{j=1}^{m} b_i c_i^2 a_{ij} c_j - \sum_{i = 1}^{m} \sum_{j=1}^{m} b_i c_i a_{ij} c_j^2 + \frac{1}{3} \sum_{i = 1}^{m} b_i c_i^{4} = 0\\
        \iff&  \sum_{i = 1}^{m} b_i c_i^3 - 2 \sum_{i = 1}^{m} \sum_{j=1}^{m} b_i c_i a_{ij} c_j + \sum_{i = 1}^{m} \sum_{j=1}^{m} b_i a_{ij} c_j^2 - \frac{1}{3} \sum_{i = 1}^{m} b_i c_i^{3} \\
        & - \sum_{i = 1}^{m} \sum_{j=1}^{m} b_i c_i^4 + 2 \sum_{i = 1}^{m} \sum_{j=1}^{m} b_i c_i^2 a_{ij} c_j - \sum_{i = 1}^{m} \sum_{j=1}^{m} b_i c_i a_{ij} c_j^2 + \frac{1}{3} \sum_{i = 1}^{m} b_i c_i^{4}  = 0\\
        \iff&  \frac{1}{4} - 2 \times \frac{1}{8} + \frac{1}{12} - \frac{1}{3} \times \frac{1}{4} - \frac{1}{5} + 2 \times \frac{1}{10} - \sum_{i = 1}^{m} \sum_{j=1}^{m} b_i c_i a_{ij} c_j^2 + \frac{1}{3} \times \frac{1}{5}  = 0\\
        \iff& \text{the condition (14) in Table~\ref{tab_ordcond}},
    \end{align*}
    where the fourth equivalence is obtained by inserting the conditions (8), (9), (10), (12) and (13) from Table~\ref{tab_ordcond};
    \begin{align*}
        &\text{the ninth condition in \eqref{eq_exp_C5}}\\
        \iff& \sum_{i=1}^m \sum_{j = 1}^{m} \sum_{k=1}^{m} b_i a_{ij} c_j a_{jk} - \sum_{i=1}^m \sum_{j = 1}^{m} \sum_{k=1}^{m} b_i a_{ij} a_{jk} c_k - \frac{1}{2} \sum_{i=1}^m \sum_{j = 1}^{m} b_i a_{ij} c_j^{2} \\
        & - \sum_{i=1}^m \sum_{j = 1}^{m} \sum_{k=1}^{m} b_i c_i a_{ij} c_j a_{jk} + \sum_{i=1}^m \sum_{j = 1}^{m} \sum_{k=1}^{m} b_i c_i a_{ij} a_{jk} c_k + \frac{1}{2} \sum_{i=1}^m \sum_{j = 1}^{m} b_i c_i a_{ij} c_j^{2} = 0\\
        \iff& \sum_{i=1}^m \sum_{j = 1}^{m} b_i a_{ij} c_j^2 - \sum_{i=1}^m \sum_{j = 1}^{m} \sum_{k=1}^{m} b_i a_{ij} a_{jk} c_k - \frac{1}{2} \sum_{i=1}^m \sum_{j = 1}^{m} b_i a_{ij} c_j^{2} \\
        & - \sum_{i=1}^m \sum_{j = 1}^{m} b_i c_i a_{ij} c_j^2 + \sum_{i=1}^m \sum_{j = 1}^{m} \sum_{k=1}^{m} b_i c_i a_{ij} a_{jk} c_k + \frac{1}{2} \sum_{i=1}^m \sum_{j = 1}^{m} b_i c_i a_{ij} c_j^{2} = 0\\
        \iff& \frac{1}{12} - \frac{1}{24} - \frac{1}{2} \times \frac{1}{12} - \frac{1}{2} \times \frac{1}{15} + \sum_{i=1}^m \sum_{j = 1}^{m} \sum_{k=1}^{m} b_i c_i a_{ij} a_{jk} c_k = 0\\
        \iff& \text{the condition (15) in Table~\ref{tab_ordcond}},
    \end{align*}
    where the third equivalence follows from conditions (10), (11) and (14) in Table~\ref{tab_ordcond};
    \begin{align*}
        &\text{the seventh condition in \eqref{eq_exp_C5}}\\
        \iff& \sum_{i=1}^m b_i \vbr{\Big}{\sum_{j=1}^{m} a_{ij} \br{c_i - c_j} - \frac{c_i^{2}}{2}}^2 = \sum_{i=1}^m b_i \vbr{\Big}{\frac{1}{2} c_i^2 - \sum_{j=1}^{m} a_{ij} c_j}^2 = 0  \displaybreak\\
        \iff& \frac{1}{4} \sum_{i=1}^m b_i c_i^4 - \sum_{i=1}^m \sum_{j=1}^{m} b_i c_i^2 a_{ij} c_j + \sum_{i=1}^m b_i \vbr{\Big}{\sum_{j=1}^{m} a_{ij} c_j}^2 = 0 \\
        \iff& \frac{1}{4} \times \frac{1}{5} - \frac{1}{10} + \sum_{i=1}^m b_i \vbr{\Big}{\sum_{j=1}^{m} a_{ij} c_j}^2 = 0 \\
        \iff& \text{the condition (16) in Table~\ref{tab_ordcond}},
    \end{align*}
    where the third equivalence follows from the conditions (12) and (13) in Table~\ref{tab_ordcond};
    \begin{align*}
        &\text{the third condition in \eqref{eq_exp_C5}}\\
        \iff&\sum_{j = 1}^{m} b_j \vbr{\Big}{\sum_{k=1}^m a_{jk} \br{c_j - c_k}^3 - \frac{c_j^{4}}{4}} = 0 \\
        \iff&\sum_{j = 1}^{m} b_j c_j^4 - 3 \sum_{j = 1}^{m} \sum_{k=1}^m  b_j c_j^2 a_{jk} c_k + 3 \sum_{j = 1}^{m} \sum_{k=1}^m b_j c_j a_{jk} c_k^2 - \sum_{j = 1}^{m} \sum_{k=1}^m b_j a_{jk} c_k^3 - \frac{1}{4}\sum_{j = 1}^{m} b_j c_j^{4} = 0\\
        \iff&\frac{1}{5} - 3 \times \frac{1}{10} + 3 \times \frac{1}{15} - \sum_{j = 1}^{m} \sum_{k=1}^m b_j a_{jk} c_k^3 - \frac{1}{4}\times \frac{1}{5} = 0\\
        \iff&\text{the condition (17) in Table~\ref{tab_ordcond}},
    \end{align*}
    where the third equivalence is derived by substituting the conditions (12), (13) and (14) from Table~\ref{tab_ordcond};
    \begin{align*}
        &\text{the eighth condition in \eqref{eq_exp_C5}}\\
        \iff& \sum_{i=1}^m \sum_{j = 1}^{m} \sum_{k = 1}^{m} b_i c_i a_{ij} c_j a_{jk}  - \sum_{i=1}^m \sum_{j = 1}^{m} \sum_{k=1}^{m}  b_i c_i a_{ij} a_{jk} c_k - \frac{1}{2} \sum_{i=1}^m \sum_{j = 1}^{m} b_i c_i a_{ij} c_j^{2} \\
        & - \sum_{i=1}^m \sum_{j = 1}^{m} \sum_{k=1}^m b_i a_{ij} c_j^2 a_{jk} + \sum_{i=1}^m \sum_{j = 1}^{m} \sum_{k=1}^{m} b_i a_{ij} c_j a_{jk} c_k + \frac{1}{2} \sum_{i=1}^m \sum_{j = 1}^{m} b_i a_{ij} c_j^{3} = 0\\
        \iff& \sum_{i=1}^m \sum_{j = 1}^{m} b_i c_i a_{ij} c_j^2  - \sum_{i=1}^m \sum_{j = 1}^{m} \sum_{k=1}^{m}  b_i c_i a_{ij} a_{jk} c_k - \frac{1}{2} \sum_{i=1}^m \sum_{j = 1}^{m} b_i c_i a_{ij} c_j^{2} \\
        & - \sum_{i=1}^m \sum_{j = 1}^{m} b_i a_{ij} c_j^3 + \sum_{i=1}^m \sum_{j = 1}^{m} \sum_{k=1}^{m} b_i a_{ij} c_j a_{jk} c_k + \frac{1}{2} \sum_{i=1}^m \sum_{j = 1}^{m} b_i a_{ij} c_j^{3} = 0 \\
        \iff& \frac{1}{2}\sum_{i=1}^m \sum_{j = 1}^{m} b_i c_i a_{ij} c_j^2  - \sum_{i=1}^m \sum_{j = 1}^{m} \sum_{k=1}^{m}  b_i c_i a_{ij} a_{jk} c_k \\
        & - \frac{1}{2} \sum_{i=1}^m \sum_{j = 1}^{m} b_i a_{ij} c_j^3 + \sum_{i=1}^m \sum_{j = 1}^{m} \sum_{k=1}^{m} b_i a_{ij} c_j a_{jk} c_k = 0  \\
        \iff& \frac{1}{2} \times \frac{1}{15} - \frac{1}{30} - \frac{1}{2} \times \frac{1}{20} + \sum_{i=1}^m \sum_{j = 1}^{m} \sum_{k=1}^{m} b_i a_{ij} c_j a_{jk} c_k = 0\\
        \iff& \text{the condition (18) in Table~\ref{tab_ordcond}},
    \end{align*}
    where the third equivalence is obtained by inserting the conditions (14), (15) and (17) from Table~\ref{tab_ordcond};
    \begin{align*}
        &\text{the sixth condition in \eqref{eq_exp_C5}}\\
        \iff& \sum_{i=1}^m \sum_{j=1}^{m} \sum_{k=1}^{m} b_i a_{ij} a_{jk} c_j^2 - 2 \sum_{i=1}^m \sum_{j=1}^{m} \sum_{k=1}^{m} b_i a_{ij} a_{jk} c_j c_k \\
        &\quad\quad + \sum_{i=1}^m \sum_{j=1}^{m} \sum_{k=1}^{m} b_i a_{ij} a_{jk} c_k^2 - \frac{1}{3}\sum_{i=1}^m \sum_{j=1}^{m} b_i a_{ij} c_j^{3} = 0\\
        \iff& \sum_{i=1}^m \sum_{j=1}^{m} b_i a_{ij} c_j^3 - 2 \sum_{i=1}^m \sum_{j=1}^{m} \sum_{k=1}^{m} b_i a_{ij} c_j a_{jk} c_k \\
        &\quad\quad + \sum_{i=1}^m \sum_{j=1}^{m} \sum_{k=1}^{m} b_i a_{ij} a_{jk} c_k^2 - \frac{1}{3} \sum_{i=1}^m \sum_{j=1}^{m} b_i a_{ij} c_j^{3} = 0\\
        \iff& \frac{2}{3}\sum_{i=1}^m \sum_{j=1}^{m} b_i a_{ij} c_j^3 - 2 \sum_{i=1}^m \sum_{j=1}^{m} \sum_{k=1}^{m} b_i a_{ij} c_j a_{jk} c_k + \sum_{i=1}^m \sum_{j=1}^{m} \sum_{k=1}^{m} b_i a_{ij} a_{jk} c_k^2 = 0\\
        \iff& \frac{2}{3} \times \frac{1}{20} - 2 \times \frac{1}{40} + \sum_{i=1}^m \sum_{j=1}^{m} \sum_{k=1}^{m} b_i a_{ij} a_{jk} c_k^2 = 0\\
        \iff& \text{the condition (19) in Table~\ref{tab_ordcond}},
    \end{align*}
    where the fourth equivalence follows from the conditions (17) and (18) in Table~\ref{tab_ordcond};
    \begin{align*}
        &\text{the tenth condition in \eqref{eq_exp_C5}}\\
        \iff& \sum_{i=1}^m \sum_{j=1}^{m} \sum_{k=1}^{m} \sum_{l=1}^{m}  b_i a_{ij} a_{jk} c_k a_{kl} - \sum_{i=1}^m \sum_{j=1}^{m} \sum_{k=1}^{m} \sum_{l=1}^{m} b_i a_{ij} a_{jk} a_{kl} c_l - \frac{1}{2} \sum_{i=1}^m \sum_{j=1}^{m} \sum_{k=1}^{m} b_i a_{ij} a_{jk} c_k^{2} = 0\\
        \iff& \frac{1}{2}\sum_{i=1}^m \sum_{j=1}^{m} \sum_{k=1}^{m} b_i a_{ij} a_{jk} c_k^2 - \sum_{i=1}^m \sum_{j=1}^{m} \sum_{k=1}^{m} \sum_{l=1}^{m} b_i a_{ij} a_{jk} a_{kl} c_l = 0\\
        \iff& \frac{1}{2} \times \frac{1}{60} - \sum_{i=1}^m \sum_{j=1}^{m} \sum_{k=1}^{m} \sum_{l=1}^{m} b_i a_{ij} a_{jk} a_{kl} c_l = 0\\
        \iff& \text{the condition (20) in Table~\ref{tab_ordcond}},
    \end{align*}
    where the third equivalence is obtained by inserting the condition (19) from Table~\ref{tab_ordcond}.
\end{proof}
 
\section{Global error estimates}\label{sec_errest}

In this section, we derive global error estimates for \Cref{sch_rk} using the order conditions stated in \Cref{thm_consistent}.

We first introduce some notation.
For any $\R^d$-valued $\mathcal{F}$-measurable random variable $\eta$, we define the norms $\norm{\eta}_{L^{\infty}} := \inf\{M \in \R: \abs{\eta} < M, \text{a.s.}\}$ and $\norm{\eta}_{L^p} := \br{\E{\abs{\eta}^p}}^{1/p}$ for $p \in [1, \infty)$.
For $p \in [1, \infty]$, denote by $L^p\br{\R^d, \mathcal{F}}$ the space of random variables $\eta: \Omega \to \R^d$ that are $\mathcal{F}$-measurable and $L^p$-integrable, i.e., $\norm{\eta}_{L^p} < \infty$.
For $r \in [1, \infty)$, denote by $\mathbb{L}^r(\R^d)$ the space of stochastic processes $\Gamma: [0, T] \times \Omega \to \R^d$ that are $\mathbb F$-progressively measurable and $L^r$-integrable, i.e., $\mathbb{E}[\int_{[0, T]} \abs{\Gamma_s}^r \di s] < \infty$.
Denote by $\mathbb{L}^{\infty}(\R^d)$ the space of bounded, $\mathbb F$-progressively measurable processes $\Gamma: [0, T] \times \Omega \to \R^d$.
For $r \in [1, \infty]$, define the discrete-time stochastic process space as
\begin{equation*}
\begin{aligned}
    \mathbb{L}_{N}^r\br{\R^{p} \times \R^{p\times q}} := \left\{\left.\left\{\Gamma^n\right\}_{n=0}^N \right|\Gamma^n = (\Gamma_y^n, \Gamma_z^n) \in L^r\br{\R^{p} \times \R^{p\times q}, \mathcal{F}_{t_n}}, \;\; 0 \leq n \leq N \right\}.
\end{aligned}
\end{equation*}
For $a = \rbr{a_{ijl}} \in \R^{p\times p \times q}$ and $z = [z_{jl}] \in \R^{p \times q}$, define $a:z = [c_i] \in \R^p$, where $c_i = \sum_{j=1}^p \sum_{l=1}^q a_{ijl} z_{jl}$ for $i = 1, 2, \cdots, p$.

To analyze the stability of \Cref{sch_rk}, we impose the following assumption.

\begin{assumption}\label{assu_FLipVar}
    There exist positive constants $C_{F1}, C_{F2} > 0$ and a seminorm $\rhof{\cdot}$ on $\R^{p} \times \R^{p\times q}$, such that, for any $\Theta^{\prime}$, $\Theta^{\prime\prime} \in \mathbb{L}^2\br{\R^{p} \times \R^{p\times q}}$ and any $0 \leq t < s \leq T$, 
    \begin{equation}\label{eq_L2normEtDelFts}
    \begin{aligned}
        \E{\srhof{\Et{\Delta F_{t, s}}}} \leq\;& C_{F1} \E{\srhof{\Delta \Theta_{s}}} + \frac{C_{F2}}{s - t} \E{\srhof{\Delta \Theta_s} - \srhof{\Et{\Delta \Theta_s}}},
    \end{aligned}
    \end{equation}
    where $\Delta F_{t, s} := F_t\br{s, \Theta_s^{\prime}} - F_t\br{s, \Theta_s^{\prime\prime}}$ and $\Delta \Theta_s := \Theta_s^{\prime} - \Theta_s^{\prime\prime}$. 
    Moreover, the seminorm $\rhof{\cdot}$ takes the form
    \begin{equation}\label{eq_hilbert_rho}
        \rhof{\theta} = \abs{L\theta} \sptext{for all} \theta \in \R^{p} \times \R^{p\times q},
    \end{equation}
    where $L: \R^{p} \times \R^{p\times q} \to \R^k$ is a deterministic linear map for some integer $k \geq 1$.
\end{assumption}

Assumption~\ref{assu_FLipVar} is introduced for technical reasons.
For clarity, we recall two lemmas from \cite{Fang2023ODE} that impose sufficient conditions on the generator $f$ to ensure Assumption~\ref{assu_FLipVar}.
Lemma~\ref{lemm_fnoz} shows that Assumption~\ref{assu_FLipVar} holds when $f\br{t, x, y, z}$ is independent of $z$ and Lipschitz with respect to $y$.
This setting covers many physical and chemical models, such as the Allen-Cahn equation, the nonconvective Gray-Scott system \cite{John1993Complex}, and more general reaction-diffusion equations without convection.

\begin{lemma}[\!\cite{Fang2023ODE}]\label{lemm_fnoz}
    Assume the generator $f$ of \eqref{bsde} is independent of $z$ and is Lipschitz continuous in $y$ with constant $C$, i.e.,
    \begin{equation}\label{eq_Lipf}
        f(t, x, y, z) = f(t, x, y), \quad \abs{f(t, x, y) - f(t, x, y^{\prime})} \leq C \abs{y - y^{\prime}}
    \end{equation}
    for all $t \in [0, T]$, $x \in \R^d$, $y \in \R^p$, $y^{\prime} \in \R^p$, and $z \in \R^{p\times q}$.
    Then \eqref{eq_L2normEtDelFts} holds with $C_{F1} = C^2$, $C_{F2} = 0$, and the seminorm $\rhof{\cdot} = \abs{\cdot}_{\R^p}$ defined by
    \begin{equation}\label{eq_defnormRp}
        \abs{\br{y, z}}_{\R^p} := \abs{y} =  \vbr{\Big}{\sum_{i=1}^p y_i^2}^{\frac{1}{2}} \sptext{for} \br{y, z} \in \R^p \times \R^{p\times q}.
    \end{equation}
\end{lemma}

The following lemma shows that Assumption~\ref{assu_FLipVar} holds when $f\br{t, x, y, z}$ is linear in $\br{y, z}$.
This linearity assumption on $f$ is standard in classical von Neumann stability analyses; see, for instance, \cite{Chassagneux2015Numerical}.

\begin{lemma}[\!\cite{Fang2023ODE}]\label{lemm_linearf}
    Assume the generator $f$ in \eqref{bsde} is of the form
    \begin{equation}\label{eq_linearf}
        f(t, X_t, y, z) = f_{t}^0 + f_{t}^1 y + f_{t}^2 : z
    \end{equation}
    where $f^0 \in \mathbb{L}^2\br{\R^p}$, $f^1 \in \mathbb{L}^{\infty}\br{\R^{p \times p}}$, and $f^2 \in \mathbb{L}^{\infty}\br{\R^{p \times p\times q}}$. Moreover, there exists a constant $C$, independent of $t$ and $s$, such that
    \begin{equation}\label{eq_fs1}
        \abs{f_s^1} + \abs{f_s^2} \leq C, \;\; \abs{f_s^1 - \mathbb{E}_t \rbr{f_s^1}} + \abs{f_s^2 - \mathbb{E}_t \rbr{f_s^2}} \leq C \sqrt{s-t}
    \end{equation}
    for all $0 \leq t \leq s \leq T$, almost surely.
    Then \eqref{eq_L2normEtDelFts} holds with $C_{F1} = C_{F2} = 24 q C^2$, and $\rhof{\cdot}$ taken as the Euclidean norm $\abs{\cdot}$ on $\R^p \times \R^{p \times q}$.
\end{lemma}

We now prove the stability of \Cref{sch_rk}.
For notational brevity, we rewrite \Cref{sch_rk} as a one-step scheme of the form
\begin{equation}\label{eq_rk1s}
    \Theta^{n} = \Etn{\Theta^{n+1}} + \Delta t_n \Etn{\Phi^{n+1}\br{\Theta^{n+1}}}, \;\; n = N-1, N-2, \cdots, 0,
\end{equation}
where $\Phi^{n+1}\br{\cdot}$ is the operator defined by
\begin{gather}
    \Phi^{n+1}\br{\Gamma^{n+1}} := \sum_{i=1}^m b_{i} F_{t_{n}}\br{t_{n, i}, \Gamma^{n, i}}, \label{eq_Phirk}\\
    \Gamma^{n, i} := \Etni{\Gamma^{n+1}} + \Delta t_n \sum_{j=i+1}^m a_{ij}\Etni{F_{t_{n, i}}\br{t_{n, j}, \Gamma^{n, j}}}, \quad i = 1, 2, \cdots, m\notag 
\end{gather}
for $\Gamma^{n+1} \in L^2\br{\R^p \times \R^{p\times q}, \mathcal{F}_{t_{n+1}}}$. 
We then present the following lemma, which shows that $\Phi^{n+1}\br{\cdot}$ inherits the property of $F_t\br{s, \cdot}$ given in \eqref{eq_L2normEtDelFts}.
This property will be used to derive the stability result for \Cref{sch_rk}.

\begin{lemma}\label{lemm_LipVarRK}
    Let \Cref{assu_FLipVar,assu_coeff} hold, and $\Theta_1, \Theta_2 \in \mathbb{L}_{N}^2(\R^{p} \times \R^{p\times q})$ be arbitrary discrete processes. 
    For $n = 0, 1, \cdots, N-1$, it holds that
    \begin{equation}\label{eq_proveDPhi}
        \E{\srhof{\Etn{\Delta \Phi^{n+1}}}} \leq \bar{C}_F \E{\srhof{\Delta \Theta^{n+1}}} + \frac{\bar{C}_F}{\Delta t_n} \E{\srhof{\Delta \Theta^{n+1}} - \srhof{\Etn{\Delta \Theta^{n+1}}}}
    \end{equation} 
    for $\Delta t_n \leq \min\{1, (4 m^2 \bar{a}^2 (C_{F1} + 2 C_{F2} \delta_c^{-1}))^{-1}\}$, where
    \begin{align*}
        &\bar{C}_F := 4 m^3 \bar{b}^2 \br{C_{F1} + 2 C_{F2} \delta_c^{-1}}, \quad \bar{a} := \max_{i,j=1,2\cdots,m} \abs{a_{ij}}, \\
        &\bar{b} := \max_{i,j=1,2\cdots,m} \abs{b_{i}}, \quad \delta_c := \min_{0 \leq i < j \leq m} (c_i - c_j),
    \end{align*}
    \begin{equation}\label{eq_defDphi}
        \Delta \Phi^{n+1} := \Phi^{n+1}\br{\Theta_1^{n+1}} - \Phi^{n+1}\br{\Theta_2^{n+1}}, \;\; \Delta \Theta^{n+1} := \Theta_1^{n+1} - \Theta_2^{n+1}. 
    \end{equation}
\end{lemma}

\begin{proof}
    Let $\Theta_1$ and $\Theta_2$ be arbitrary discrete processes in $\mathbb{L}_{N}^2\br{\R^{p} \times \R^{p\times q}}$.
    For $n = 0, 1, \cdots, N-1$, $i = 1, 2, \cdots, m$, $l = 1, 2$ and $t < t_{n, i}$, define
    \begin{align}
         & \Theta_l^{n, i} := \Etni{\Theta_l^{n+1}} + \Delta t_n \sum_{j=i+1}^m a_{ij}\Etni{F_{t_{n, i}}\br{t_{n, j}, \Theta_l^{n, j}}},\label{eq_Gammani} \\
         & \Delta \Theta^{n, i} := \Theta_1^{n, i} - \Theta_2^{n, i}, \;\; \Delta \widetilde{F}_{t}^{n, i} := F_{t}\br{t_{n, i}, \Theta_1^{n, i}} - F_{t}\br{t_{n, i}, \Theta_2^{n, i}}.  \label{eq_DeltaPhirk}
    \end{align}
    By substituting \eqref{eq_Phirk} into the right-hand side of the first equation in \eqref{eq_defDphi} and then using the second equation in \eqref{eq_DeltaPhirk},
    we obtain $\Delta \Phi^{n+1} = \sum_{i=1}^m b_{i} \Delta \widetilde{F}_{t_n}^{n, i}$.
    Taking $\E{\srhof{\cdot}}$ on both sides of this equation, we obtain
    \begin{equation*}
    \refstepcounter{equation}\tag{\theequation}\label{eq_estDelPhirk}
    \begin{aligned}
        \E{\srhof{\Etn{\Delta \Phi^{n+1}}}} = \E{\srhof{\sum_{i=1}^m b_{i} \Etn{\Delta \widetilde{F}_{t_n}^{n, i}}}} \leq\;& \E{\bbr{\sum_{i=1}^m |b_{i}| \rhof{\Etn{\Delta \widetilde{F}_{t_n}^{n, i}}}}^2} \notag \\
        \leq\; & \sum_{j=1}^m |b_{j}|^2 \sum_{i=1}^m \E{\srhof{\Etn{\Delta \widetilde{F}_{t_n}^{n, i}}}} \notag \\
        \leq\; & m \bar{b}^2 \sum_{i=1}^m \E{\srhof{\Etn{\Delta \widetilde{F}_{t_n}^{n, i}}}},
    \end{aligned}
    \end{equation*}
    where the first and the second inequality follows from the triangle inequality and the Cauchy-Schwarz inequality, respectively.

    For $0 \leq i < j \leq m$, inserting \eqref{eq_Gammani} into the right side of the first equation in \eqref{eq_DeltaPhirk}, and then using the second equation in \eqref{eq_DeltaPhirk}, we obtain
    \begin{equation*}
        \refstepcounter{equation}\tag{\theequation}\label{eq_DelThetani}
        \Delta \Theta^{n, j} = \Etnj{\Delta \Theta^{n+1}} + \Delta t_n \sum_{l=j+1}^m a_{jl} \Etnj{\Delta \widetilde{F}_{t_{n, j}}^{n, l}}, 
    \end{equation*}
    where $\Delta \Theta^{n+1}$ is given in \eqref{eq_defDphi}.
    Taking $\Etni{\cdot}$ on both sides of \eqref{eq_DelThetani}, we obtain
    \begin{equation}\label{eq_DelThetanj}
        \Etni{\Delta \Theta^{n, j}} = \Etni{\Delta \Theta^{n+1}} + \Delta t_n \sum_{l=j+1}^m a_{jl} \Etni{\Delta \widetilde{F}_{t_{n, j}}^{n, l}}.  
    \end{equation}
    Subtracting \eqref{eq_DelThetanj} from \eqref{eq_DelThetani}, we obtain
    \begin{equation}\label{eq_Del2Thetani}
    \begin{aligned}
        \Delta \Theta^{n, j} - \Etni{\Delta \Theta^{n, j}} =\; & \Etnj{\Delta \Theta^{n+1}} - \Etni{\Delta \Theta^{n+1}}\\
        & + \Delta t_n \sum_{l=j+1}^m a_{jl} \br{\Etnj{\Delta \widetilde{F}_{t_{n, j}}^{n, l}} - \Etni{\Delta \widetilde{F}_{t_{n, j}}^{n, l}}}.
    \end{aligned}
    \end{equation}
    Taking $\E{\srhof{\cdot}}$ on both sides of \eqref{eq_DelThetani} and \eqref{eq_Del2Thetani}, and then using Jensen's inequality, we obtain
    \begin{equation}\label{eq_norm2DelTheta}
    \begin{aligned}
        \E{\srhof{\Delta \Theta^{n, j}}} \leq\; & 2 \E{\srhof{\Etnj{\Delta \Theta^{n+1}}}} + 2 m \bar{a}^2 \br{\Delta t_n}^2 \sum_{l=j+1}^m \E{\srhof{\Etnj{\Delta \widetilde{F}_{t_{n, j}}^{n, l}}}} \\
        \leq\; & 2 \E{\srhof{\Delta \Theta^{n+1}}} + 2 m \bar{a}^2 \br{\Delta t_n}^2 \sum_{l=j+1}^m \E{\srhof{\Etnj{\Delta \widetilde{F}_{t_{n, j}}^{n, l}}}},
    \end{aligned}
    \end{equation}
    and
    \begin{equation}\label{eq_norm2Del2Theta}
        \begin{aligned}
            \E{\srhof{\Delta \Theta^{n, j} - \Etni{\Delta \Theta^{n, j}}}} \leq\; & 2 \E{\srhof{\Etnj{\Delta \Theta^{n+1}} - \Etni{\Delta \Theta^{n+1}}}}\\
            & + 4 m \bar{a}^2 \br{\Delta t_n}^2 \sum_{l=j+1}^m \E{\srhof{\Etnj{\Delta \widetilde{F}_{t_{n, j}}^{n, l}}}},
        \end{aligned}
    \end{equation}
    where $\bar{a}$ is given in Lemma~\ref{lemm_LipVarRK}.
    
    Recalling \eqref{eq_DeltaPhirk} and Assumption~\ref{assu_FLipVar}, we have
    \begin{equation}\label{eq_EtniDelFnj}
        \begin{aligned}
            \E{\srhof{\Etni{\Delta \widetilde{F}_{t_{n, i}}^{n, j}}}} \leq\; & C_{F1} \E{\srhof{\Delta \Theta^{n, j}}} + \frac{C_{F2}}{t_{n, j} - t_{n, i}} \E{\srhof{\Delta \Theta^{n, j} - \Etni{\Delta \Theta^{n, j}}}}.
        \end{aligned}
    \end{equation}
    Inserting \eqref{eq_norm2DelTheta} and \eqref{eq_norm2Del2Theta} into the right side of \eqref{eq_EtniDelFnj}, we obtain
    \begin{equation}\label{eq_EtniDelFnj2}
        \begin{aligned}
            & \E{\srhof{\Etni{\Delta \widetilde{F}_{t_{n, i}}^{n, j}}}}\\
            \leq\; & 2 C_{F1} \E{\srhof{\Delta \Theta^{n+1}}} + \frac{2 C_{F2}}{t_{n, j} - t_{n, i}} \E{\srhof{\Etnj{\Delta \Theta^{n+1}} - \Etni{\Delta \Theta^{n+1}}}} \\
            & + 4 m \bar{a}^2 C_{F2} \Delta t_n \frac{\Delta t_n}{t_{n, j} - t_{n, i}} \sum_{l=j+1}^m \E{\srhof{\Etnj{\Delta \widetilde{F}_{t_{n, j}}^{n, l}}}}   \\
            & + 2 m \bar{a}^2 C_{F1} \br{\Delta t_n}^2 \sum_{l=j+1}^m \E{\srhof{\Etnj{\Delta \widetilde{F}_{t_{n, j}}^{n, l}}}}.
        \end{aligned}
    \end{equation}
    
    Now we estimate the second term on the right-hand side of \eqref{eq_EtniDelFnj2}: 
    \begin{equation}\label{eq_estEtniDelFnj}
    \begin{aligned}
        \E{\srhof{\Etnj{\Delta \Theta^{n+1}} - \Etni{\Delta \Theta^{n+1}}}}
        =\;& \E{\abs{\Etnj{L \Delta \Theta^{n+1}} - \Etni{L \Delta \Theta^{n+1}}}^2}\\
        =\;& \E{\abs{\Etnj{L \Delta \Theta^{n+1}}}^2} - \E{\abs{\Etni{L \Delta \Theta^{n+1}}}^2} \\
        \leq\;& \E{\abs{L \Delta \Theta^{n+1}}^2} - \E{\abs{\Etn{L \Delta \Theta^{n+1}}}^2}\\
        =\;& \E{\srhof{\Delta \Theta^{n+1}} - \srhof{\Etn{\Delta \Theta^{n+1}}}},
    \end{aligned}
    \end{equation}
    where the first and the last line follow from \eqref{eq_hilbert_rho}; 
    the second line is obtained by orthogonal decomposition;
    the third line follows from the facts
    \begin{equation*}
        \E{\abs{\Etnj{L \Delta \Theta^{n+1}}}^2} \leq \E{\abs{L \Delta \Theta^{n+1}}^2}, \;\; \E{\abs{\Etni{L \Delta \Theta^{n+1}}}^2} \geq \E{\abs{\Etn{L \Delta \Theta^{n+1}}}^2}. 
    \end{equation*}

    Inserting \eqref{eq_estEtniDelFnj} into the right-hand side of \eqref{eq_EtniDelFnj2}, and using the last condition in \eqref{eq_ajieq0_maxabc}, we obtain
    \begin{equation}\label{eq_estDelFtnij}
    \begin{aligned}
        \E{\srhof{\Etni{\Delta \widetilde{F}_{t_{n, i}}^{n, j}}}} \leq\; & 2 C_{F1} \E{\srhof{\Delta \Theta^{n+1}}} + \frac{2 C_{F2}}{\delta_c \Delta t_n} \E{\srhof{\Delta \Theta^{n+1}} - \srhof{\Etn{\Delta \Theta^{n+1}}}}  \\
        & + 2 m \bar{a}^2 \br{C_{F1} \Delta t_n + 2 C_{F2} \delta_c^{-1}} \Delta t_n \sum_{l=j+1}^m \E{\srhof{\Etnj{\Delta \widetilde{F}_{t_{n, j}}^{n, l}}}}.
    \end{aligned}
    \end{equation}
    Taking $\sum_{i=0}^{m-1} \sum_{j=i+1}^m$ on both side of \eqref{eq_estDelFtnij} yields
    \begin{equation*}
    \begin{aligned}
        & \sum_{i=0}^{m-1} \sum_{j=i+1}^m \E{\srhof{\Etni{\Delta \widetilde{F}_{t_{n, i}}^{n, j}}}}\\
        \leq\; & 2 m^2 C_{F1} \E{\srhof{\Delta \Theta^{n+1}}} + \frac{2 m^2 C_{F2}}{\delta_c \Delta t_n} \E{\srhof{\Delta \Theta^{n+1}} - \srhof{\Etn{\Delta \Theta^{n+1}}}}\\
        & + 2 m^2 \bar{a}^2 \br{C_{F1} \Delta t_n + 2 C_{F2} \delta_c^{-1}} \Delta t_n \sum_{j=0}^{m-1} \sum_{l=j+1}^m \E{\srhof{\Etnj{\Delta \widetilde{F}_{t_{n, j}}^{n, l}}}}. 
    \end{aligned}
    \end{equation*}
    Rearranging the terms leads to
    \begin{equation*}
    \begin{aligned}
        & \vbr{\Big}{1 - 2 m^2 \bar{a}^2 \br{C_{F1} \Delta t_n + 2 C_{F2} \delta_c^{-1}} \Delta t_n} \sum_{i=0}^{m-1} \sum_{j=i+1}^m \E{\srhof{\Etni{\Delta \widetilde{F}_{t_{n, i}}^{n, j}}}} \\
        \leq\; & 2 m^2 C_{F1} \E{\srhof{\Delta \Theta^{n+1}}} + \frac{2 m^2 C_{F2}}{\delta_c \Delta t_n} \E{\srhof{\Delta \Theta^{n+1}} - \srhof{\Etn{\Delta \Theta^{n+1}}}}.
    \end{aligned}
    \end{equation*}
    Letting $\Delta t_n \leq \min\{1, \br{4 m^2 \bar{a}^2 \br{C_{F1} + 2 C_{F2} \delta_c^{-1}}}^{-1}\}$, the above inequality can be simplified into 
    \begin{equation}\label{eq1_estSumrho}
        \begin{aligned}
            & \sum_{i=0}^{m-1} \sum_{j=i+1}^m \E{\srhof{\Etni{\Delta \widetilde{F}_{t_{n, i}}^{n, j}}}} \\
            \leq\; & 4 m^2 C_{F1} \E{\srhof{\Delta \Theta^{n+1}}} + \frac{4 m^2 C_{F2}}{\delta_c \Delta t_n} \E{\srhof{\Delta \Theta^{n+1}} - \srhof{\Etn{\Delta \Theta^{n+1}}}}. 
        \end{aligned}
    \end{equation}
    Furthermore, the first line of \eqref{eq1_estSumrho} can be bounded from below as follows:
    \begin{equation}\label{eq2_estSumrho}
        \sum_{j=1}^m \E{\srhof{\Etn{\Delta \widetilde{F}_{t_{n}}^{n, j}}}} \leq \sum_{i=0}^{m-1} \sum_{j=i+1}^m \E{\srhof{\Etni{\Delta \widetilde{F}_{t_{n, i}}^{n, j}}}}
    \end{equation}
    which follows from the fact that the left side is a summand of the right side.
    
    Combining \eqref{eq1_estSumrho} with \eqref{eq2_estSumrho} and substituting the resulting inequality into \eqref{eq_estDelPhirk}, we complete the proof.
\end{proof}

To analyze the stability of \Cref{sch_rk}, we introduce a perturbation process $\zeta \in \mathbb{L}_{N}^2 \br{\R^{p} \times \R^{p\times q}}$ into \Cref{sch_rk}, leading to the following perturbed scheme:
\begin{equation}\label{eq_pert1s}
    \Theta_{\zeta}^{n} = \Etn{\Theta_{\zeta}^{n+1}} + \Delta t_n \Etn{\Phi^{n+1}\br{\Theta_{\zeta}^{n+1}}} + \zeta^n
\end{equation}
for $n = N-1, N-2, \cdots, 0$, with perturbed terminal values $\Theta_{\zeta}^{N} = \Theta^N + \zeta^{N}$.
The stability result of \Cref{sch_rk} is given by \Cref{thm_rk_stable} below. 

\begin{theorem}[Stability]\label{thm_rk_stable}
    Under \Cref{assu_coeff,assu_FLipVar}, \Cref{sch_rk} is stable with respect to the seminorm $\rhof{\cdot}$; namely, there exist positive constants $C$ and $\bar{\delta}$, both independent of the time partition $\pi_N$, such that, for any $\Delta t \in (0, \bar{\delta}]$,
    \begin{equation*}
        \max_{0\leq n\leq {N-1}} \E{\srhof{\Theta_{\zeta}^n - \Theta^n}} \leq C \vbr{\Big}{\E{\srhof{\zeta^N}} + \frac{1}{\Delta t} \sum_{i=0}^{N-1} \E{\srhof{\zeta^i}}}.
    \end{equation*}
\end{theorem}

\begin{proof}

    For $n = 0, 1, \cdots, N-1$, define $\mathcal{E}^n := \Theta_{\zeta}^n - \Theta^n$.
    By subtracting \eqref{eq_rk1s} from \eqref{eq_pert1s}, we obtain
    \begin{equation}\label{eq_1sEn}
        \mathcal{E}^n = \Etn{\mathcal{E}^{n+1}} + \Delta t_n \Etn{\mathcal{E}_{\Phi}^{n+1}} + \zeta^n,
    \end{equation}
    where
    \begin{equation}\label{eq_defEphi}
        \mathcal{E}_{\Phi}^{n+1} := \Phi^{n+1}\br{\Theta_{\zeta}^{n+1}} - \Phi^{n+1}\br{\Theta^{n+1}}.
    \end{equation}

    Let $\delta_2 > 0$ be a constant to be specified later. For any $\eta > 0$ and any random vectors $a$, $b$, and $c$ in $L^2\br{\R^p \times \R^{p\times q}, \mathcal{F}}$, the following inequality holds
    \begin{equation*}
    \begin{aligned}
        & \srhof{a + b +c} \leq \abs{\rhof{a + b} + \rhof{c}}^2 \\
        \leq\; & \br{1+\eta} \abs{\rhof{a} + \rhof{b}}^2 + \br{1+\frac{1}{\eta}}\srhof{c}\\
        \leq\; & \br{1+\eta}\br{1+\frac{\eta}{\delta_2}}\srhof{a} + \br{1+\eta}\br{1+\frac{\delta_2}{\eta}}\srhof{b} + \br{1+\frac{1}{\eta}}\srhof{c}.
    \end{aligned}
    \end{equation*}
    Applying $\E{\srhof{\cdot}}$ to both sides of \eqref{eq_1sEn} with $\eta = \Delta t_n$, and then using the above inequality, we obtain
    \begin{equation}\label{eq_rkEn2}
    \begin{aligned}
        \E{\srhof{\mathcal{E}^n}} \leq\; & \br{1 + \Delta t_n} \br{1 + \frac{\Delta t_n}{\delta_2}} \E{\srhof{\Etn{\mathcal{E}^{n+1}}}}                                    \\
        & + \br{\Delta t_n}^2 \br{1+\Delta t_n} \br{1 + \frac{\delta_2}{\Delta t_n}} \E{\srhof{\Etn{\mathcal{E}_{\Phi}^{n+1}}}} + \br{1+\frac{1}{\Delta t_n}} \E{\srhof{\zeta^n}},
    \end{aligned}
    \end{equation}
    Recalling the definition in \eqref{eq_defEphi}, Lemma~\ref{lemm_LipVarRK} implies that there exists a constant $c_0$, independent of $\pi_N$, such that
    \begin{equation*}
        \E{\srhof{\Etn{\mathcal{E}_{\Phi}^{n+1}}}} \leq c_0 \br{1 + \frac{1}{\Delta t_n}} \E{\srhof{\mathcal{E}^{n+1}}} - \frac{c_0}{\Delta t_n} \E{\srhof{\Etn{\mathcal{E}^{n+1}}}}.
    \end{equation*}
    Inserting the above inequality into the right side of \eqref{eq_rkEn2}, we obtain
    \begin{equation}\label{eq_rkEn3}
    \begin{aligned}
        & \E{\srhof{\mathcal{E}^n}}\\
        \leq\; & \vbr{\Big}{\br{1 + \Delta t_n} \br{1 + \frac{\Delta t_n}{\delta_2}} - c_0 \Delta t_n \br{1+\Delta t_n} \br{1 + \frac{\delta_2}{\Delta t_n}}} \E{\srhof{\Etn{\mathcal{E}^{n+1}}}}\\
        & + c_0 \br{\Delta t_n}^2 \br{1+\Delta t_n} \br{1 + \frac{\delta_2}{\Delta t_n}} \br{1 + \frac{1}{\Delta t_n}} \E{\srhof{\mathcal{E}^{n+1}}} + \br{1+\frac{1}{\Delta t_n}} \E{\srhof{\zeta^n}}.
    \end{aligned}
    \end{equation}
    
    By Jensen's inequality, we have 
    \begin{equation*}
        \E{\srhof{\Etn{\mathcal{E}^{n+1}}}} \leq \E{\srhof{\mathcal{E}^{n+1}}}. 
    \end{equation*}
    Inserting the above inequality into \eqref{eq_rkEn3} and taking $\delta_2 = \min\{c_0^{-1}, 1\}$, we obtain 
    \begin{equation}\label{eq_rkEEn2}
    \begin{aligned}
        \E{\srhof{\mathcal{E}^n}} \leq\; & \br{1+\Delta t_n}^2 \br{1 + c_0 \Delta t_n} \E{\srhof{\mathcal{E}^{n+1}}} + \br{1+\frac{1}{\Delta t_n}} \E{\srhof{\zeta^n}} \\
        \leq\; & \br{1 + C_3 \Delta t} \E{\srhof{\mathcal{E}^{n+1}}} + \frac{C_4}{\Delta t} \E{\srhof{\zeta^n}}
    \end{aligned}
    \end{equation}
    for $\Delta t_n \in (0, 1)$, where $C_3 := (4 c_0 + 3) C_0$ and $C_4 := 1 + C_0$ with $C_0$ given in \eqref{eq_timepart}. 
    For any $m \in \{0, 1, \cdots, N-1\}$, taking the summation $\sum_{n=m}^{N-1}$ on both sides of \eqref{eq_rkEEn2} and canceling the common terms $\sum_{n=m+1}^{N-1} \E{\srhof{\mathcal{E}^n}}$ yields 
    \begin{equation*}
        \E{\srhof{\mathcal{E}^m}} \leq C_3 \Delta t \sum_{n=m+1}^N \E{\srhof{\mathcal{E}^{n}}} + \E{\srhof{\zeta^N}} + \frac{C_4}{\Delta t} \sum_{n=0}^{N-1} \E{\srhof{\zeta^n}}.
    \end{equation*}
    By the discrete Gronwall inequality \cite[Lemma 3]{zhao2010stable}, the above inequality implies
    \begin{equation}
        \max_{0\leq n \leq N-1} \E{\srhof{\mathcal{E}^n}} \leq 2 C_4 e^{C_0 C_3T}\vbr{\Big}{\E{\srhof{\zeta^N}} + \frac{1}{\Delta t} \sum_{n=0}^{N-1} \E{\srhof{\zeta^n}}},
    \end{equation}
    which implies the desired result.
    
\end{proof}

Invoking Theorem~\ref{thm_rk_stable}, we derive the following error estimate for \Cref{sch_rk}.

\begin{theorem}\label{thm_errest1s}
    Let \Cref{assu_regu,assu_coeff,assu_FLipVar} hold. 
    Then there exist constants $C$ and $\bar{\delta}$ both independent of the time partition $\pi_N$, 
    such that, for $\Delta t \in (0, \bar{\delta})$,
    \begin{equation*}
        \max_{0\leq n\leq {N-1}} \E{\srhof{\Theta_{t_n} - \Theta^n}} \leq C \vbr{\Big}{\E{\srhof{\Theta_{t_N} - \Theta^N}} + \frac{1}{\Delta t} \sum_{i=0}^{N-1} \E{\abs{R^i}^2}}.
    \end{equation*}
\end{theorem}

\begin{proof}
    Recall the definition of the local truncation error in the second equation of \eqref{eq_Rni} and the perturbed scheme \eqref{eq_pert1s}. Set
    \begin{equation*}
        \zeta^N = \Theta_{t_N} - \Theta^N, \quad \zeta^n = -R^n, \quad n = N-1, N-2, \cdots, 0,
    \end{equation*}
    Then the perturbed solution $\Theta_{\zeta}$ coincides with the analytic solution $\Theta$.
    Applying Theorem~\ref{thm_rk_stable}, there exist constants $C_1$ and $\bar{\delta}$ both independent of $\pi_N$, such that
    \begin{equation}\label{eq_est_srhofTheta}
    \begin{aligned}
        \max_{0\leq n\leq {N-1}} \E{\srhof{\Theta_{t_n} - \Theta^n}} \leq\;& C_1 \vbr{\Big}{\E{\srhof{\Theta_{t_N} - \Theta^N}} + \frac{1}{\Delta t} \sum_{i=0}^{N-1} \E{\srhof{R^i}}}, \quad \Delta t \in (0, \bar{\delta}). 
    \end{aligned}
    \end{equation}
    By the properties of seminorms, we have
    \begin{equation}\label{eq_Crho}
        \abs{\rho\br{\theta}}
        \leq C_{\rho} \abs{\theta} \sptext{for} \theta \in \R^{p \times (1 + q)}
    \end{equation}
    with the constant $C_{\rho} := \sup\bbr{\rho\br{\theta^{\prime}}: \theta^{\prime} \in \R^{p \times (1 + q)}, \; \abs{\theta^{\prime}} \leq 1} < +\infty$.
    Combining \eqref{eq_est_srhofTheta} with \eqref{eq_Crho} yields the desired result.
\end{proof}

By combining \Cref{thm_consistent,thm_errest1s}, we immediately obtain the following theorem on the convergence order of \Cref{sch_rk}.

\begin{theorem}\label{thm_converge}
    Suppose that Assumptions~\ref{assu_regu}, \ref{assu_coeff} and \ref{assu_FLipVar} hold.
    Further suppose that the terminal error satisfies $\E{\srhof{\Theta_{t_N} - \Theta^N}} \leq C \abs{\Delta t}^{2r}$ for some constant $C$ independent of the time partition $\pi_N$.
    If the coefficients of \Cref{sch_rk} satisfy the conditions $C(r)$ in \eqref{eq_defCr},
    then \Cref{sch_rk} converges with order $r$ in the sense that
    \begin{equation*}
        \max_{0\leq n\leq {N-1}} \E{\srhof{\Theta_{t_n} - \Theta^n}} = \mO\br{\br{\Delta t}^{2r}}, \sptext{as} \Delta t \to 0.
    \end{equation*}
\end{theorem}

\begin{proof}
    Under the assumed conditions $\E{\srhof{\Theta_{t_N} - \Theta^N}} \leq C \abs{\Delta t}^{2r}$, Theorem~\ref{thm_errest1s} implies that
    \begin{equation*}
        \max_{0\leq n\leq {N-1}} \E{\srhof{\Theta_{t_n} - \Theta^n}} \leq C^{\prime} \vbr{\Big}{\abs{\Delta t}^{2r} + \frac{1}{\Delta t} \sum_{i=0}^{N-1} \E{\abs{R^i}^2}} \sptext{for} \Delta t \leq \bar{\delta},
    \end{equation*}
    where $C^{\prime}$ and $\bar{\delta}$ are some constants independent of $\pi_N$.
    Applying \Cref{thm_consistent} to bound the local truncation errors on the right side of the above inequality yields the desired result.
\end{proof}

The convergence result in Theorem~\ref{thm_converge} relies on the seminorm $\rho\br{\cdot}$ introduced in the technical Assumption~\ref{assu_FLipVar}.
In practice, $\rho\br{\cdot}$ can be instantiated as the seminorms specified in Lemma~\ref{lemm_fnoz} and Lemma~\ref{lemm_linearf}, which yields the following corollary.

\begin{corollary}\label{coro_error_est}
Suppose \Cref{assu_regu,assu_coeff} hold.
Assume the coefficients of \Cref{sch_rk} satisfy \eqref{eq_ord_cond1} or the conditions $C\br{r}$ in \eqref{eq_defCr}.
Then the following statements hold:
\begin{itemize}
    \item[(i)] If the generator $f$ satisfies \eqref{eq_Lipf}, and the terminal values for \Cref{sch_rk} satisfy 
    \begin{equation*}
        \E{\abs{Y_{t_N} - Y^N}^2} \leq C \abs{\Delta t}^{2r}
    \end{equation*}
    for some constant $C$ independent of the time partition $\pi_N$, then the $Y$-component of \Cref{sch_rk} converges with order $r$, i.e.,
    \begin{equation*}
        \max_{0\leq n\leq {N-1}}\E{\abs{Y_{t_n} - Y^n}^2} = \mO\br{\abs{\Delta t}^{2r}}, \sptext{as} \Delta t \to 0.
    \end{equation*}
    \item[(ii)] If the generator $f$ satisfies \eqref{eq_linearf} and \eqref{eq_fs1},
    and the terminal values for \Cref{sch_rk} satisfy $\E{\abs{\Theta_{t_N} - \Theta^N}^2} \leq C \abs{\Delta t}^{2r}$
    for some constant $C$ independent of the time partition $\pi_N$,
    then \Cref{sch_rk} converges with order $r$, i.e.,
    \begin{equation*}
        \max_{0\leq n\leq {N-1}}\E{\abs{\Theta_{t_n} - \Theta^n}^2} = \mO\br{\abs{\Delta t}^{2r}}, \sptext{as} \Delta t \to 0.
    \end{equation*}
\end{itemize}
\end{corollary}

We emphasize that the conditions on $f$ in Corollary~\ref{coro_error_est} are sufficient, but not necessary, for the convergence of \Cref{sch_rk}.
As evidenced by the numerical experiments in \Cref{sec_numtest}, \Cref{sch_rk} remains stable and achieves convergence rates consistent with the theoretical predictions even when $f\br{t, x, y, z}$ depends nonlinearly on $\br{y, z}$.
More broadly, the conditions on $f$ in Corollary~\ref{coro_error_est} serve solely to verify Assumption~\ref{assu_FLipVar} and to guarantee the stability of \Cref{sch_rk}.
Whether Assumption~\ref{assu_FLipVar} is necessary for the stability of \Cref{sch_rk} remains an open question.
Since this work focuses primarily on the order conditions for the coefficients of \Cref{sch_rk}, a thorough investigation into the necessity of Assumption~\ref{assu_FLipVar} and the stability of \Cref{sch_rk} for general generators $f$ is left for future work.

\bibliographystyle{plain}
\bibliography{references}

\endgroup

\bibliographystyle{plain}
\bibliography{references}

\end{document}